 \documentclass[12pt]{article}

\usepackage{amsmath,amssymb,amsthm,latexsym}

\input epsf
\usepackage{epsfig}

\usepackage{color}
\usepackage{comment}
\allowdisplaybreaks

\newtheorem{theorem}{Theorem}[section]
\newtheorem{thm}[theorem]{Theorem}
\newtheorem{corollary}[theorem]{Corollary}

\newtheorem{lemma}[theorem]{Lemma}
\newtheorem{proposition}[theorem]{Proposition}
\newtheorem{prop}[theorem]{Proposition}
\newtheorem{definition}[theorem]{Definition}
\newtheorem{remark}[theorem]{Remark}
\newtheorem{example}[theorem]{Example}
\numberwithin{equation}{section}

\def\be{\begin{equation}}
\def\ee{\end{equation}}
\def\bes{\begin{equation*}}
\def\ees{\end{equation*}}

\def\vp{{\varphi}}

\def\wt{\widetilde}

\def\eps{\varepsilon}
\def\lam{{\lambda}}
\def\ol{\overline}
\def\Gam{\Gamma}

\def\qed{{\hfill $\square$ \bigskip}}

\def\esssup{{\mathop{\rm ess \; sup \, }}}
\def\essinf{{\mathop{\rm ess \; inf \, }}}

\def\sE {{\cal E}}
\def\sF {{\cal F}}
\def\sL {{\cal L}}
\def\sD{{\cal D}}

\def\sN {{\cal N}}

\def\bE {{\mathbb E}} \def\bN {{\mathbb N}}
\def\bP {{\mathbb P}} \def\bR {{\mathbb R}} \def\bZ {{\mathbb Z}}

\def\ms{\medskip}
\def\sms{\smallskip}

\def\PHI{\mathrm{PHI}}
\def\UHK{\mathrm{UHK}}

\def\LHK{\mathrm{LHK}}
\def\HK{\mathrm{HK}}
\def\CSA{\mathrm{CSA}}
\def\CSJ{\mathrm{CSJ}}
\def\SCSJ{\mathrm{SCSJ}}
\def\CSAJ{\mathrm{CSAJ}}
\def\PI{\mathrm{PI}}

\def\FK{\mathrm{FK}}
\def\VD{\mathrm{VD}}
\def\RVD{\mathrm{RVD}}

\def\PHI{\mathrm{PHI}}
\def\EP{\mathrm{EP}}
\def\UHKD{\mathrm{UHKD}}
\definecolor{dred}{rgb}{0.8, 0.0, 0.0}

\def\Nash{\mathrm{Nash}}

\def\E{\mathrm{E}}
\def\J{\mathrm{J}}
\def\<{\langle}
\def\>{\rangle}
\def\FF{{\cal F}}

\def\T{\mathrm{Tail}}

\begin{document}

\title{\bf Stability of heat kernel estimates for symmetric
non-local Dirichlet forms}

\author{{\bf Zhen-Qing Chen}\footnote{Research partially supported
by NSF grant DMS-1206276.}, \quad {\bf Takashi Kumagai}\footnote{Research
partially supported by the Grant-in-Aid for Scientific Research (A)
25247007 and 17H01093.} \quad and \quad
{\bf Jian Wang}\footnote{Research partially supported by
the National Natural Science Foundation of China (No.\ 11522106), Fok Ying Tung Education Foundation (No.\ 151002)
the JSPS postdoctoral fellowship (26$\cdot$04021),
National Science Foundation of Fujian Province (No.\ 2015J01003), the Program for Nonlinear
Analysis and Its Applications (No.\ IRTL1206), and Fujian Provincial Key Laboratory of Mathematical
Analysis and its Applications (FJKLMAA).
}}

\date{}
\maketitle

\begin{abstract} In this paper,
we consider symmetric jump processes of  mixed-type  on metric measure spaces
under general volume doubling condition, and establish   stability of
two-sided heat kernel estimates and  heat kernel upper bounds.
We obtain their stable equivalent characterizations
in terms of the
jumping kernels, variants
of cut-off Sobolev inequalities, and
the Faber-Krahn inequalities.
In particular, we establish stability of heat kernel estimates for $\alpha$-stable-like processes even with $\alpha\ge 2$ when
the underlying spaces have walk dimensions larger than $2$, which has been one of the major open problems in this area.
\end{abstract}



{\small
 \begin{tableofcontents}
 \end{tableofcontents} }

\section{Introduction and Main Results} \label{sec:intro}

\subsection{Setting}
Let $(M,d)$ be a locally compact separable metric space,
and
$\mu$
a positive Radon measure on $M$ with full support.
We will refer to such a triple $(M, d, \mu)$ as a \emph{metric measure space},
and denote by
$\langle  \cdot , \cdot \rangle$ the inner product in $L^2(M; \mu)$.
Throughout the paper, we
assume that all balls are relatively compact and
assume for simplicity that $\mu (M)=\infty$.
 We would emphasize that in this paper we do not assume $M$ to be connected nor
 $(M, d)$ to be geodesic.

We consider a regular \emph{Dirichlet form} $(\sE, \sF)$
on $L^2(M; \mu)$. By the Beurling-Deny formula, such form can be decomposed
into three terms --- the strongly local term, the pure-jump term and the killing term (see \cite[Theorem 4.5.2]{FOT}).
Throughout this paper, we consider the form that consists of the pure-jump term only; namely
there exists a symmetric Radon measure $J(\cdot,\cdot)$ on $M\times M\setminus \textrm{diag}$,
 where \textrm{diag} denotes the diagonal set $\{(x, x): x\in M\}$,  such that
\begin{equation}\label{e:1.1}
\sE(f,g)=\int_{M\times M\setminus \textrm{diag}}(f(x)-f(y)(g(x)-g(y))\,J(dx,dy),\quad f,g\in \sF.
\end{equation}

Since $(\sE,\sF)$ is regular, each function $f\in \sF$ admits a
quasi-continuous version $\wt f$ on $M$ (see \cite[Theorem
2.1.3]{FOT}). Throughout the paper,
we will always take a quasi-continuous version of $f\in \sF$ without denoting it by $\tilde f$.
Let $(\sL, \sD (\sL))$
be the (negative definite) $L^2$-\emph{generator} of $(\sE,
\sF)$ on $L^2(M; \mu)$; this is, $\sL$  is the self-adjoint operator
in $L^2(M; \mu)$
whose domain $\sD (\sL)$ consists exactly those of $f\in \sF$ that there is some (unique) $u\in L^2(M; \mu)$
so that
$$
  \sE (f,g) =   \langle u,g \rangle \quad \hbox{ for all }
  g \in \sF ,
$$
and $\sL f := - u$.
Let $\{P_t\}_{t\geq 0}$  be
the associated \emph{semigroup}. Associated with the regular
Dirichlet form $(\sE, \sF)$ on $L^2(M; \mu)$ is a $\mu$-symmetric
\emph{Hunt process} $X=\{X_t, t \ge 0; \, \bP^x, x \in M\setminus
\sN\}$. Here
$\sN \subset M$
is a properly exceptional set for $(\sE, \sF)$ in
the sense that
 $\mu(\sN)=0$ and $M_\partial\setminus \sN$ is $X$-invariant;
that is,
$$
\bP^x( X_t \in M_\partial\setminus\sN \hbox{ and } X_{t-}\in M_\partial\setminus\sN \hbox{ for all } t\ge0 )=0
$$
for all $x \in M\setminus\sN$
with the convention that $X_{0-}:=X_0$. Here
 $M_{\partial}:=M\cup\{\partial\}$ is the one-point compactification of $M$.
This Hunt process is unique up to a
properly exceptional set --- see \cite[Theorem 4.2.8]{FOT}.
We fix $X$ and $\sN$, and write $ M_0 = M \setminus \sN.$
While the semigroup $\{P_t\}_{t\ge0}$ associated with $\sE$ is defined on $L^2(M; \mu)$,
a more precise version with better regularity properties can be obtained,
if we set, for any bounded Borel measurable function $f$ on $M$,
$$ P_t f(x) = \bE^x f(X_t), \quad x \in M_0. $$
The \emph{heat kernel} associated with the semigroup $\{P_t\}_{t\ge0}$ (if it exists) is a
measurable function
$p(t, x,y):M_0 \times M_0 \to (0,\infty)$ for every $t>0$, such that
\begin{align} \label{e:hkdef}
\bE^x f(X_t) &= P_tf(x) = \int p(t, x,y) f(y) \, \mu(dy), \quad  x \in M_0,
f \in L^\infty(M;\mu), \\
p(t, x,y) &= p(t, y,x) \quad \hbox{for all } t>0,\, x ,y \in M_0, \\
\label{e:ck}
p(s+t , x,z) &= \int p(s, x,y) p(t, y,z) \,\mu(dy)
\quad \hbox{for all } s>0, t>0,
\,\, x,z \in M_0.
\end{align}
While \eqref{e:hkdef} only determines  $p(t, x, \cdot)$ $\mu$-a.e.,
using the Chapman-Kolmogorov equation \eqref{e:ck} one can
regularize $p(t, x,y)$ so that \eqref{e:hkdef}--\eqref{e:ck} hold
for every point in $M_0$. See \cite[Theorem 3.1]{BBCK} and
\cite[Section 2.2]{GT} for details. We call $p(t, x,y)$ the {\em
heat kernel} on the \emph{metric measure Dirichlet space} (or
\emph{MMD space}) $(M, d, \mu, \sE)$.
By \eqref{e:hkdef}, sometime we also call $p(t,x,y)$ the \emph{transition density function} with respect to the measure $\mu$ for the process $X$.
 Note that in some arguments of
our paper, we can extend (without further mention) $p(t,x,y)$ to all
$x$, $y\in M$ by setting $p(t,x,y)=0$ if $x$ or $y$ is outside
$M_0$. The existence of the heat kernel allows to extend the
definition of $P_tf$ to all measurable functions $f$ by choosing a
Borel measurable version of $f$ and noticing that the integral
\eqref{e:hkdef}  does not change if function $f$ is changed on a set
of measure zero.

Denote the ball centered at $x$ with radius $r$ by
$B(x, r)$ and $\mu (B(x, r))$ by $V(x, r)$.
When
the metric measure space $M$
is  an Alhfors $d$-regular set on $\bR^n$ with $d\in (0, n]$ (that is,   $V(x, r)\asymp r^d$ for all $x\in \bR^n$ and $r\in (0, 1]$),
and
the Radon measure $J(dx,dy)=J(x,y)\,\mu(dx)\,\mu(dy)$ for some non-negative symmetric function $J(x,y)$ such that
\begin{equation}\label{e:1.6}
J(x, y) \asymp \frac{1}{d(x, y)^{d+\alpha}}, \quad x, y \in M
\end{equation}
 for some $0<\alpha<2$,  it is established in \cite{CK1}  that
the corresponding Markov process $X$ has infinite lifetime,  and has
a jointly H\"older
continuous transition density function $p(t, x, y)$ with respect to the measure $\mu$, which enjoys
the following two-sided estimate
\begin{equation}\label{e:1.7}
 p(t, x, y) \asymp  t^{-d/\alpha} \wedge \frac{t}{d(x, y)^{d+\alpha}}
 \end{equation}
for any $(t, x, y)\in (0, 1]\times M\times M$. Here for two positive
functions $f, g$, notation $f\asymp g$ means $f/g$ is bounded
between two positive constants, and $a\wedge b:= \min \{a, b\}$.
Moreover, if $M$ is a global $d$-set; that is, if $V(x, r)\asymp
r^d$
 holds for all $x\in \bR^n$ and $r>0$, then the estimate \eqref{e:1.7} holds for all $(t,x,y)\in(0,\infty)
 \times M\times M$.
We call the above Hunt process $X$  an $\alpha$-stable-like process on $M$.
 Note that when $M=\bR^d$ and $J(x, y)= c|x-y|^{-(d+\alpha)}$ for all $x,y\in \bR^d$ and some constants $\alpha\in(0,2)$ and $c>0$,
 $X$ is  a rotationally symmetric $\alpha$-stable L\'evy process on $\bR^d$.
The estimate \eqref{e:1.7} can be regarded as the jump process counterpart of the celebrated Aronson estimates for diffusions.
Since $J(x, y)$ is the weak limit of $p(t, x, y)/t$ as $t\to 0$, heat kernel estimate \eqref{e:1.7} implies
\eqref{e:1.6}. Hence the results from \cite{CK1} give a stable characterization for $\alpha$-stable-like  heat kernel estimates
when $\alpha \in (0, 2)$ and
the metric measure space $M$
is a $d$-set for some constant $d >0$.
 This result has later  been extended to mixed stable-like processes
 on more general metric measure spaces  in \cite{CK2} and to diffusions with jumps  on Euclidean spaces  in \cite{CK3},
with some growth  condition on the rate function $\phi$ such as
\be\label{eq:into2fb}
\int_0^r \frac{s}{\phi(s)} \,ds\le \frac{c\,
r^2}{\phi(r)} \quad \hbox{for all }  r>0
\ee
with some constant $c>0$. For $\alpha$-stable-like
processes where $\phi (r)=r^\alpha$, condition \eqref{eq:into2fb}
corresponds exactly to $0<\alpha<2$.
Some of the key methods used in \cite{CK1} were inspired by a previous work
\cite{BL} on random walks on integer lattice $\bZ^d$.

The notion of $d$-set arises in the theory of function spaces and in
fractal geometry. Geometrically, self-similar sets are typical examples of $d$-sets.
There are many self-similar fractals on which there exist fractal diffusions with walk dimension $d_w>2$
(that is,  diffusion processes with scaling relation ${\it time} \approx {\it space}^{d_w}$).
This is the case, for example,
for the Sierpinski gasket in $\bR^n$ ($n\ge 2$) which is a $d$-set with $d=\log (n+1)/\log 2$
and has walk dimension $d_w=\log (n+3)/\log 2$,
and for the Sierpinski carpet in $\bR^n$ ($n\ge 2$) which is a $d$-set with $d=\log (3^n-1)/\log 3$
 and has walk dimension $d_w>2$;
 see  \cite{B}.
  A direct calculation shows (see \cite{BSS, Sto})
that  the $\beta$-subordination of the fractal diffusions on these fractals are jump processes
whose Dirichlet forms $(\sE, \sF)$ are of the form  given above with $\alpha = \beta d_w$ and their transition density functions
have two-sided estimate \eqref{e:1.7}. Note that as $\beta \in (0, 1)$, $\alpha \in (0, d_w)$ so $\alpha$ can be larger than 2.
When $\alpha >2$, the approach in \cite{CK1} ceases to work as it is hopeless to construct good cut-off functions a priori in this case.
A long standing open problem in the field is whether estimate \eqref{e:1.7} holds for generic jump processes with jumping kernel
of the form \eqref{e:1.6} for any $\alpha \in (0, d_w)$. A related open question is to find a characterization for heat kernel estimate \eqref{e:1.7} that is stable under ``rough isometries".   Do they hold on general metric measure spaces with volume doubling
(VD) and reverse volume doubling (RVD) properties (see Definition \ref{D:1.1} below for these two terminologies)?
These are the questions we will address in this paper.

For diffusions on manifolds with walk dimension 2, a remarkable fundamental result obtained independently by Grigor'yan \cite{Gr2} and
Saloff-Coste \cite{Sa} asserts that the following are equivalent: (i)  Aronson-type Gaussian bounds for heat kernel,
(ii) parabolic Harnack equality, and (iii) VD and Poincar\'e inequality.
This result is then extended to strongly local Dirichlet forms on metric measure spaces in \cite{BM, St1, St2}
and to graphs in \cite{De}. For diffusions on fractals with walk dimension larger than 2,
the above equivalence  still holds  but one needs to replace (iii)
by (iii') VD, Poincar\'e inequality and a cut-off Sobolev inequality; see \cite{BB2, BBK1, AB}.
For heat kernel estimates of
symmetric jump processes in general metric measure spaces,
as mentioned above, when $\alpha \in (0, 2)$ and
 the metric measure space $M$ is a $d$-set,
 characterizations of $\alpha$-stable-like heat kernel estimates were obtained in \cite{CK1}
which are stable under rough isometries;
see \cite{CK2, CK3} for further extensions.
For  the equivalent characterizations of heat kernel estimates for symmetric jump processes
analogous to the situation when $\alpha \geq 2$,
there are some efforts such as \cite[Theorem 1.2]{BGK1} and \cite[Theorem 2.3]{GHL2}
 but none of these characterizations are stable under rough isometries.
In \cite[Theorem 0.3]{BGK1}, assuming that
  $(\sE, \sF)$ is conservative,  $V(x, r) \leq c_1r^d$ for all $x\in M$, $r>0$ and some constants $c_1,d>0$,
and that  $p(t, x, x) \leq c_2 t^{-d/\alpha}$ for any $x\in M$, $t>0$ and some constant $c_2>0$,
an equivalent characterization for the heat kernel upper bound
estimate in \eqref{e:1.7}
is given in terms of certain exit time estimates.
Under the assumption that $(\sE, \sF)$ is conservative,
  the Radon measure $J(dx,dy)=J(x,y)\,\mu(dx)\,\mu(dy)$ for some non-negative symmetric function $J(x,y)$,
  and
 $V(x, r) \leq cr^d$ for all $x\in M$, $r>0$ and some constants $c,d>0$,
it is shown in \cite{GHL2} that heat kernel  upper bound estimate
in \eqref{e:1.7} holds if and only if  there are some constants $c_1,c_2>0$ such that $p(t, x,x)\leq c_1t^{d/\alpha}$ and
$J(x, y) \leq c_2 d(x, y)^{-(d+\alpha)}$ for all $x,y\in M$ and $t>0$, and the following survival estimate holds:
there are constants $\delta, \eps \in (0, 1)$ so that
$\bP^x(\tau_{B(x, r)}\leq t) \leq \eps$  for all $x\in M$ and $r, t>0$ with $t^{1/\alpha} \leq \delta r$.
In both \cite{BGK1, GHL2}, $\alpha$ can be larger than 2.
 We note that when $\alpha<2$, further equivalent
 characterizations
of heat kernel estimates are given for jump processes on graphs
\cite[Theorem 1.5]{BBK2}, some of which are stable under rough
isometries. Also, when the Dirichlet form of the jump process is
parabolic (namely the capacity of any non-empty compact subset of
$M$ is positive \cite[Definition 6.3]{GHL2}, which is equivalent to
that every singleton has positive capacity), an equivalent
characterization of heat kernel estimates is given in  \cite[Theorem
6.17]{GHL2}, which is stable under rough isometries.

\subsection{Heat kernel}

In this paper, we are concerned with
 both upper bound and  two-sided estimates on $p(t, x,y)$
 for mixed stable-like processes on general metric measure spaces
 including $\alpha$-stable-like processes with $\alpha \geq 2$.
\ms
To state our results precisely, we need a number of
definitions.

\begin{definition} \label{D:1.1} {\rm
(i) We say that $(M,d,\mu)$ satisfies the {\it volume doubling property} ($\VD$) if
there exists a constant ${C_\mu}\ge 1$ such that
for all $x \in M$ and $r > 0$,
\be \label{e:vd}
 V(x,2r) \le {C_\mu} V(x,r).
 \ee
(ii) We say that $(M,d,\mu)$ satisfies the {\it reverse volume doubling property} ($\RVD$)
if there exist constants
$d_1>0$, ${c_\mu} >0$ such that
for all $x \in M$ and $0<r\le R$,
\be \label{e:rvd}
\frac{V(x,R)}{V(x,r)}\ge {c_\mu}  \Big(\frac Rr\Big)^{d_1}.\ee
 }\end{definition}

VD condition \eqref{e:vd} is equivalent to the existence of $d_2>0$ and ${\wt C_\mu}>0$ so that
\be \label{e:vd2}
\frac{V(x,R)}{V(x,r)} \leq   {\wt C_\mu}  \Big(\frac Rr\Big)^{d_2} \quad \hbox{for all } x\in M \hbox{ and }
0<r\le R,
\ee
while RVD condition \eqref{e:rvd} is equivalent to the existence of
${l_\mu} >1$ and ${\wt c_\mu}>1$ so that
\be \label{e:rvd2}
 V(x,{l_\mu}  r) \geq  {\wt c_\mu} V(x,r) \quad \hbox{for all } x\in M \hbox{ and } r>0.
 \ee

Since $\mu$ has full support on $M$, we have $\mu (B(x, r))>0$
for every $x\in M$ and $r>0$.
Under VD condition, we have from \eqref{e:vd2} that for all $x\in M$ and $0<r\le R$,
\be\label{eq:vdeno}
\frac{V(x,R)}{V(y,r)}\le
\frac{V(y, d(x, y)+R)}{V(y,r)} \leq {\wt C_\mu}\Big(\frac {d(x,y)+R}r\Big)^{d_2}.
\ee
On the other hand, under $\RVD$, we have  from \eqref{e:rvd2} that
$$
\mu\big(B(x_0,{l_\mu} r)\setminus B(x_0,r)\big)>0 \quad \hbox{for each } x_0\in M \hbox{ and } r>0  .
$$
It is  known that $\VD$ implies
$\RVD$
if $M$ is connected and unbounded.
See, for example \cite[Proposition 5.1 and Corollary 5.3]{GH}.

\sms

Let $\bR_+:=[0,\infty)$, and $\phi: \bR_+\to \bR_+$ be a strictly increasing continuous
function  with $\phi (0)=0$ ,
$\phi(1)=1$
and satisfying that there exist constants $c_1,c_2>0$ and $\beta_2\ge \beta_1>0$ such that
\be\label{polycon}
 c_1 \Big(\frac Rr\Big)^{\beta_1} \leq
\frac{\phi (R)}{\phi (r)}  \ \leq \ c_2 \Big(\frac
Rr\Big)^{\beta_2}
\quad \hbox{for all }
0<r \le R.
\ee
Note that (\ref{polycon}) is
equivalent to the existence of constants
$ c_3, {l_0}>1$
such that
$$   c_3^{-1}\phi (r) \leq \phi ({l_0}r)
\leq c_3 \, \phi (r)\quad \hbox{for all } r>0.
$$

\begin{definition}{\rm We say $\J_\phi$ holds if there exists a non-negative symmetric
function $J(x, y)$ so that for $\mu\times \mu $-almost  all $x, y \in M$,
\begin{equation}\label{e:1.2}
J(dx,dy)=J(x, y)\,\mu(dx)\, \mu (dy),
\end{equation} and
\begin{equation}\label{jsigm}
 \frac{c_1}{V(x,d(x, y)) \phi (d(x, y))}\le J(x, y) \le \frac{c_2}{V(x,d(x, y)) \phi (d(x, y))}
 \end{equation}
We say that $\J_{\phi,\le}$ (resp. $\J_{\phi,\ge}$) if \eqref{e:1.2} holds and the upper bound (resp. lower bound) in \eqref{jsigm} holds.}
\end{definition}

\begin{remark}\label{R:1.2} \rm \begin{description}
\item{(i)}
Since changing the value of $J(x, y)$ on a subset of $M\times M$ having zero $\mu\times \mu $-measure
does not affect the definition of the Dirichlet form $(\sE, \sF)$ on $L^2(M; \mu)$, without loss of generality,
we may and do assume that in condition $\J_\phi$ ($\J_{\phi, \geq}$ and $\J_{\phi, \leq}$, respectively)
that \eqref{jsigm} (and the corresponding inequality) holds for every $x, y \in M$.
In addition, by the symmetry of $J(\cdot,\cdot)$, we may and do assume that $J(x,y)=J(y,x)$ for  all $x, y\in M$.

\item{(ii)}
Note that, under $\VD$, for every $\lambda >0$, there are constants $0<c_1<c_2$ so that for every $r>0$,
\be\label{eq:noefohrj}
c_1 V(y, r)\le V(x,r)\le c_2 V(y,r) \quad \hbox{for }   x,y \in M
\hbox{ with }  d (x, y) \leq \lambda  r.
\ee
Indeed, by \eqref{eq:vdeno}, we have for every $r>0$ and $x, y\in M$ with $d(x, y) \leq \lambda r$,
\[
{\wt C_\mu}^{-1} (1+\lambda)^{-d_2}\le \frac{V(x, r)}{V(y,r )}\le {\wt C_\mu} (1+\lambda)^{d_2} .
\]
Taking $\lambda =1$ and $r=d (x, y)$ in \eqref{eq:noefohrj} shows that, under $\VD$ the bounds in condition \eqref{jsigm}
are consistent with the symmetry of $J(x, y)$.
\end{description}
\end{remark}

\begin{definition} \rm
Let $U \subset V$ be open sets of $M$ with
$U \subset \ol U \subset V$.
We say a non-negative bounded measurable function $\vp$ is a {\it cut-off function for $U \subset V$},
if $\vp  = 1$ on $U$,  $\vp=0$ on $V^c$ and $0\leq \vp \leq 1$ on $M$.
\end{definition}

 For $f, g\in \sF$, we define the
 carr\'e du champ $\Gamma (f, g)$ for the non-local Dirichlet form
$(\sE, \sF)$ by
$$
\Gamma (f, g) (dx) = \int_{y\in M} (f(x)-f(y))(g(x)-g(y))\,J(d x,d y)  .
$$
Clearly $\sE(f,g)=  \Gamma(f,g) (M)$.

Let $\sF_b= \sF\cap L^\infty(M,\mu)$. It can be verified (see \cite[Lemma 3.5 and Theorem 3.7]{CKS}) that for any
$f\in \sF_b$, $\Gamma(f,f)$ is the unique
Borel measure (called the \emph{energy measure}) on $M$ satisfying
$$
\int_M g \, d\Gamma(f,f)=\sE(f, fg)-\frac 12\sE(f^2,g),\quad f,g\in \sF_b.
$$
Note that the following chain rule holds: for  $f,g,h \in \sF_b$,
$$
\int_M\,d\Gamma(f g,h)=\int_Mf \, d\Gamma(g,h)+\int_Mg\, d\Gamma(f,h).
$$
Indeed, this can be easily seen by the following equality
\[
f(x)g(x)-f(y)g(y)=f(x)(g(x)-g(y))+g(y)(f(x)-f(y)), \quad x,y\in M.
\]
 We now introduce a condition that controls
the energy of cut-off functions.

\begin{definition} \rm Let $\phi$ be an increasing function on $\bR_+$.
\begin{itemize}
\item[(i)] {\bf (Condition $\CSJ(\phi)$)}\quad
We say that condition $\CSJ(\phi)$ holds if there exist constants $C_0\in (0,1]$ and $C_1, C_2>0$
such that for every
$0<r\le R$, almost all $x_0\in M$ and any $f\in \sF$, there exists
a cut-off function $\vp\in \sF_b$ for $B(x_0,R) \subset B(x_0,R+r)$ so that the following holds:
\be \label{e:csj1} \begin{split}
 \int_{B(x_0,R+(1+C_0)r)} f^2 \, d\Gamma (\vp,\vp)
\le &C_1 \int_{U\times U^*}(f(x)-f(y))^2\,J(dx,dy) \\
&+ \frac{C_2}{\phi(r)}  \int_{B(x_0,R+(1+C_0)r)} f^2  \,d\mu,
\end{split}
\ee
where $U=B(x_0,R+r)\setminus B(x_0,R)$ and $U^*=B(x_0,R+(1+C_0)r)\setminus B(x_0,R-C_0r)$.

\item[(ii)] {\bf (Condition $\SCSJ(\phi)$)}\quad
We say that condition $\SCSJ(\phi)$ holds if there exist constants $C_0\in (0,1]$ and $C_1, C_2>0$
such that for every
$0<r\le R$ and almost all $x_0\in M$, there exists
a cut-off function $\vp\in \sF_b$ for $B(x_0,R) \subset B(x_0,R+r)$ so that \eqref{e:csj1} holds for any $f\in \sF$.

\end{itemize}

\end{definition}

Clearly  $\SCSJ(\phi)  \Longrightarrow \CSJ(\phi)$.

\begin{remark}\label{rek:scj}\rm
\begin{itemize}
\item[(i)] $\SCSJ(\phi)$ is a modification of $\CSA(\phi)$ that was introduced in \cite{AB} for strongly local Dirichlet
forms as a weaker version of the so called cut-off Sobolev
inequality {\rm CS}$(\phi)$ in \cite{BB2,BBK1}. For strongly local
Dirichlet forms the inequality corresponding to $\CSJ(\phi)$ is
called generalized capacity condition in \cite{GHL3}.
As we will see
in Theorem \ref{T:main-1} below,
  $\SCSJ(\phi)$ and  $\CSJ(\phi)$
 are equivalent under $\FK(\phi)$ (see Definition \ref{Faber-Krahnww} below) and $\J_{\phi,\le}$.
\item[(ii)] The main difference
between $\CSJ(\phi)$ here and $\CSA(\phi)$ in \cite{AB} is that the
integrals in the left hand side and in the second term of the right
hand side of the inequality \eqref{e:csj1} are over $B(x,R+(1+C_0)r)$ (containing
$U^*$) instead of over $U$ for \cite{AB}. Note that the integral
over $U^c$ is zero in the left hand side of \eqref{e:csj1} for the
case of strongly local Dirichlet forms. As we see in the arguments of the stability of heat kernel estimates for jump processes, it is important to fatten the annulus and integrate over
$U^*$ rather than over $U$. Another difference from $\CSA(\phi)$ is
that in \cite{AB} the first term of the right hand side is
$\frac{1}{8}\int_{U}\varphi^2\, d\Gamma (f,f)$. However,   we will prove
in Proposition \ref{L:cswk} that $\CSJ(\phi)$ implies the
stronger inequality $\CSJ(\phi)_+$ under some regular conditions
$\VD$, \eqref{polycon} and $\J_{\phi,\le}$.
See   \cite[Lemma 5.1]{AB} for the case of strongly local Dirichlet forms.

\item[(iii)] As will be proved in Proposition \ref{CSJ-equi}\,(4), under $\VD$ and \eqref{polycon}, if \eqref{e:csj1} holds for some $C_0\in(0,1]$,
then it holds for all $C_0'\in[C_0,1]$
(with possibly different $C_2>0$).

\item[(iv)] By the definition above, it is clear that if $\phi_1\le \phi_2$, then
$\CSJ(\phi_2)$ implies $\CSJ(\phi_1)$.

 \item[(v)] Denote by $\sF_{loc}$ the space of functions
locally in $\sF$; that is, $f\in \sF_{loc}$
if and only if for any relatively compact open set $U\subset M$ there exists $g\in \sF$ such that $f=g$ $\mu$-a.e. on $U$.
 Since each ball is relatively compact and
 \eqref{e:csj1} uses the property of $f$ on $B(x_0,R+(1+C_0)r)$ only, both $\SCSJ(\phi)$ and  $\CSJ(\phi)$ also hold for any $f\in \sF_{loc}.$

\end{itemize}
\end{remark}

\begin{remark}\label{thm:csjrem}\rm
Under $\VD$, \eqref{polycon} and $\J_{\phi,\le}$, $\SCSJ(\phi)$ always holds if $\beta_2<2$, where $\beta_2$ is the exponent in \eqref{polycon}. In particular, $\SCSJ(\phi)$  holds for $\phi (r)=r^\alpha$ with $\alpha< 2$. Indeed, for any fixed $x_0\in M$ and $r,R>0$, we choose a non-negative cut-off function $\vp(x)=h(d(x_0,x))$, where $h\in C^1([0,\infty))$ such that $0\leq h\leq 1$,
$h(s)=1$ for all $s\le R$, $h(s)=0$ for $s\ge R+r$ and $|h'(s)|\le 2/r$ for all $s\ge0.$ Then, by $\J_{\phi,\le}$,  for almost every $x\in M$,
\begin{align*}
\frac{d\Gamma (\vp, \vp)}{d\mu } (x)&=\int (\vp(x)-\vp(y))^2 J(x,y)\,\mu(dy)\\
&\le  \int_{\{d(x,y)\ge r\}} J(x,y)\,\mu(dy)+ \frac{4}{r^2}\int_{\{d(x,y)\le r\}} d(x,y)^2 J(x,y) \,\mu(dy) \\
&\le \int_{\{d(x,y)\ge r\}} J(x,y)\,\mu(dy)+ \frac{4}{r^2}\sum_{i=0}^\infty\int_{\{2^{-i-1}r< d(x,y)\le 2^{-i}r\}} d(x,y)^2 J(x,y) \,\mu(dy ) \\
&\le  \frac{c_1}{\phi(r)}+ \frac{c_1}{r^2} \sum_{i=0}^\infty \frac{V(x,2^{-i}r) 2^{-2i} r^2}{ V(x,2^{-i-1}r)\phi(2^{-i-1}r)}\\
&\le  \frac{c_1}{\phi(r)}+ \frac{c_2}{\phi(r)} \sum_{i=0}^\infty 2^{-i(2-\beta_2)}\le \frac{c_3}{\phi(r)},
\end{align*} where in the third inequality we have used Lemma \ref{intelem} below, and the fourth inequality is due to $\VD$ and \eqref{polycon}.
 Thus \eqref{e:csj1} holds.
\end{remark}

\medskip

We next introduce the Faber-Krahn inequality, see
\cite[Section 3.3]{GT} for more details. For $\lambda>0$, we define
$$
\sE_\lambda
(f, g)=\sE(f, g)+\lambda \int_M f(x)g(x)\,\mu (dx) \quad \hbox{for }
f, g\in \sF.
$$
For any open set $D \subset M$, $\sF_D$ is defined to be the
$\sE_1$-closure in $\sF$ of  $\sF\cap C_c(D)$.
Define
\be \label{e:lam1}
 \lam_1(D)
= \inf \left\{ \sE(f,f):  \,  f \in \sF_D \hbox{ with }  \|f\|_2 =1 \right\},
\ee
the bottom of the Dirichlet spectrum of $-\sL$ on $D$.

\begin{definition}\label{Faber-Krahnww}
{\rm The MMD space $(M,d,\mu,\sE)$ satisfies the {\em Faber-Krahn
inequality} $\FK(\phi)$, if there exist positive constants $C$ and
$\nu$ such that for any ball $B(x,r)$ and any open set $D \subset
B(x,r)$, \be \label{e:fki}
 \lam_1 (D) \ge \frac{C}{\phi(r)} (V(x,r)/\mu(D))^{\nu}.
\ee
} \end{definition}

We remark that since $V(x,r)\ge \mu(D)$
for $D\subset B(x, r)$,
 if \eqref{e:fki} holds for some $\nu=\nu_0>0$, it holds
for every $\nu \in (0, \nu_0)$.  So without loss of generality, we may and do assume $0<\nu<1$.

\medskip

Recall that $X=\{X_t\}$ is the Hunt process associated with the regular Dirichlet form $(\sE,\sF)$ on $L^2(M; \mu)$
with properly exceptional set $\sN$, and $M_0:= M\setminus \sN$.
For a set
$A\subset M$, define the exit time $\tau_A = \inf\{ t >0 : X_t
\notin A\}.$

\begin{definition}{\rm We say that $\E_\phi$ holds if
there is a constant $c_1>1$ such that for all $r>0$ and  all $x\in M_0$,
$$c_1^{-1}\phi(r)\le \bE^x [ \tau_{B(x,r)} ] \le c_1\phi(r).$$ We say that $\E_{\phi,\le}$ (resp. $\E_{\phi,\ge}$) holds
if the upper bound (resp. lower bound) in the inequality above holds. }
\end{definition}

Under \eqref{polycon}, it is easy to see that $\E_{\phi,\ge}$ and $\E_{\phi,\le}$ imply the following statements respectively:
\begin{align*}
\bE^y [ \tau_{B(x,r)} ] \geq c_2\phi(r)  &\quad\mbox{for   all } x\in M, \ y\in B(x,r/2)\cap M_0, \   r>0;\\
\bE^y [\tau_{B(x,r)} ] \le c_3\phi(r) &\quad\mbox{for all } x\in M, \ y\in M_0, \  r>0.
\end{align*}
Indeed, for $y\in B(x,r/2)\cap M_0$, we have
$\bE^y [\tau_{B(x,r)} ] \ge \bE^y [\tau_{B(y,r/2)}] \ge c_1^{-1}\phi(r/2)\ge c_2 \phi(r)$.
Similarly, for $y\in B(x,r)\cap M_0 $, we have
$\bE^y [ \tau_{B(x,r)} ] \le \bE^y [\tau_{B(y,2r)}] \le c_1\phi(2r)\le c_3 \phi(r)$ (and
$\bE^y [\tau_{B(x,r)}]=0$ for $y \in M_0\setminus B(x,r)$).

\begin{definition}\rm We say $\EP_{\phi,\le}$
holds if there is a constant $c>0$ such that for all $r,t>0$ and  all $x\in M_0$,
 $$
 \bP^x(  \tau_{B(x,r)} \le t ) \le \frac{ct}{\phi(r)}.
 $$
 We say  $\EP_{\phi,\le, \varepsilon}$ holds,
 if there exist constants $\varepsilon, \delta\in (0,1)$ such that for any ball $B=B(x_0, r)$ with radius $r>0$,
$$
\bP^x (  \tau_{B} \le \delta \phi(r) ) \le \eps
\quad \hbox{for all } x\in  B(x_0, r/4)  \cap M_0 .
$$
\end{definition}
It is clear that $\EP_{\phi,\le}$ implies $\EP_{\phi,\le,\eps}$. We will prove in Lemma \ref{E} below that under \eqref{polycon}, $\E_\phi$ implies $\EP_{\phi,\le, \eps}$.

\begin{definition}\label{D:1.11}   \rm \begin{description}
\item{\rm (i)} We say that $\HK(\phi)$ holds if there exists a heat kernel $p(t, x,y)$
of the semigroup $\{P_t\}$ associated with $(\sE,\sF)$,
 which has the
following estimates for all $t>0$ and all $x,y\in M_0$,
\begin{equation}\label{HKjum}\begin{split}
&   c_1\Big(\frac 1{V(x,\phi^{-1}(t))}  \wedge
\frac{t}{V(x,d(x,y))\phi(d(x,y))}\Big) \\
  & \le \   p(t, x,y)  \le c_2\Big(\frac 1{V(x,\phi^{-1}(t))}\wedge
\frac{t}{V(x,d(x,y))\phi(d(x,y))}\Big),
\end{split}
\end{equation}
where $c_1, c_2>0$ are constants independent of $x,y\in M_0$ and $t>0$.
Here the inverse function of the strictly increasing function
$t\mapsto \phi (t)$ is denoted by $\phi^{-1}(t)$.

\item{(ii)} We say $\UHK(\phi)$ (resp. $\LHK(\phi)$) holds if the
upper bound (resp. the lower bound) in \eqref{HKjum} holds.

\item{(iii)}  We say $\UHKD(\phi)$ holds if there is a constant $c>0$ such that for all $t>0$ and all $x\in M_0$,
$$p(t, x,x)\le \frac c{V(x,\phi^{-1}(t))}.$$
\end{description}
\end{definition}

\medskip

\begin{remark}\label{R:1.22}
\rm We have three remarks about this definition.
\begin{itemize}
\item[(i)] First, note that under $\VD$
\be\label{Vphidbif}
 \frac 1{V(y,\phi^{-1}(t))}\wedge\frac{t}{V(y,d(x,y))\phi(d(x,y))}
\asymp  \frac 1{V(x,\phi^{-1}(t))}\wedge
\frac{t}{V(x,d(x,y))\phi(d(x,y))} .
 \ee
Therefore we can replace  $V(x,d(x,y))$ by $V(y,d(x,y))$ in \eqref{HKjum} by modifying the values of
$c_1$ and $c_2$.
This is because
$$\frac 1{V(x,\phi^{-1}(t))}\le  \frac{t}{V(x,d(x,y))\phi(d(x,y))}$$ if and only if
$d(x,y)\le \phi^{-1}(t)$, and by \eqref{eq:vdeno},
\[
\tilde {C_\mu}^{-1}\Big(1+\frac{d(x,y)}{\phi^{-1}(t)}\Big)^{-d_2}\le
\frac{V(x,\phi^{-1}(t))}{V(y,\phi^{-1}(t))}\le \tilde {C_\mu}\Big(1+\frac{d(x,y)}{\phi^{-1}(t)}\Big)^{d_2}.
\]
This together with \eqref{eq:noefohrj} yields \eqref{Vphidbif}.

\item[(ii)] By the Cauchy-Schwarz inequality,  one can easily see that
$\UHKD(\phi)$ is equivalent to the existence of $c_1>0$ so that
\begin{equation*}\label{uodest-0}
p(t, x,y)\le \frac{c_1}{\sqrt{V(x,\phi^{-1}(t))V(y,\phi^{-1}(t))}}
\quad \hbox{for } x, y\in M_0 \hbox{ and } t>0.
\end{equation*}
Consequently,  by Remark \ref{R:1.2}(ii), under $\VD$, $\UHKD(\phi) $  implies that for every $c_1>0$ there is a constant $c_2>0$ so that
\begin{equation*} \label{e:6.2}
p(t, x, y) \leq \frac{c_2}{V(x,\phi^{-1}(t))} \quad \hbox{for } x, y\in M_0 \hbox{ with }
 d(x, y) \leq c_1 \phi^{-1}(t).
\end{equation*}

\item[(iii)]  It will be implied by Theorem \ref{T:main} and Lemma \ref{L:holder1} below that if $\VD$,  \eqref{polycon} and $\HK(\phi)$ hold, then the heat kernel $p(t, x,y)$ is H\"{o}lder continuous on $(x,y)$ for every $t>0$, and so \eqref{HKjum} holds for all $x,y\in M$.
\end{itemize}
\end{remark}

In the following, we say
{\it $(\sE, \sF)$ is conservative}   if its associated Hunt process $X$
has infinite lifetime.  This is equivalent to $P_t 1 =1$ a.e. on $M_0$ for every $t>0$.
It follows from Proposition \ref{P:3.1}(2) that $\LHK (\phi)$ implies
that $(\sE, \sF)$ is conservative.
We can now state the stability of the heat kernel estimates $\HK(\phi)$.
The following is the main result  of this  paper.

\begin{thm} \label{T:main}
Assume that the metric measure space $(M, d , \mu)$ satisfies $\VD$ and $\RVD$, and $\phi$ satisfies \eqref{polycon}.
Then the following are equivalent: \\
$(1)$ $\HK(\phi)$. \\
$(2)$ $\J_\phi$ and $\E_\phi$.  \\
$(3)$ $\J_\phi$ and $\SCSJ(\phi)$.\\
$(4)$ $\J_\phi$ and $\CSJ(\phi)$.
\end{thm}

\begin{remark} \rm
\begin{description}
\item{(i)}
 When $\phi$ satisfies \eqref{polycon} with $\beta_2<2$,
by Remark \ref{thm:csjrem}, $\SCSJ(\phi)$ holds
and so in this case we have by Theorem \ref{T:main} that $\HK (\phi )
\Longleftrightarrow \J_\phi$.
Thus Theorem \ref{T:main}
not only recovers but also extends
 the main results in \cite{CK1, CK2}
 except for the cases where $J(x,y)$ decays exponentially when
$d(x,y)$ is large,
in the sense that the underlying spaces here are general metric measure spaces
satisfying VD and RVD.

\item{(ii)}
A new point of Theorem \ref{T:main} is that it
gives us the stability of heat kernel estimates for general
symmetric jump processes of mixed-type,  including
$\alpha$-stable-like processes with  $\alpha\ge 2$,
 on general metric measure spaces when the
underlying spaces have walk dimension larger than 2.
In particular, if $(M, d, \mu)$ is a metric measure space on which there is
an anomalous diffusion with walk dimension $d_w>2$ such as Sierpinski gaskets or carpets,
 one can deduce from
the subordinate anomalous diffusion the two-sided heat kernel estimates of
any symmetric jump processes with jumping
kernel $J(x, y)$ of $\alpha$-stable type or mixed stable type;   see Section \ref{Sectin-ex} for details.
This in particular answers a long standing problem
in the field.
\end{description}
\end{remark}

In the process of establishing  Theorem \ref{T:main}, we also obtain the following characterizations for $\UHK(\phi)$.

\begin{thm} \label{T:main-1}
Assume that the metric measure space $(M, d, \mu)$ satisfies $\VD$ and $\RVD$, and $\phi$ satisfies \eqref{polycon}.
Then the following are equivalent: \\
$(1)$ $\UHK(\phi)$ and
$(\sE, \sF)$ is conservative. \\
$(2)$ $\UHKD(\phi)$, $\J_{\phi,\le}$ and $\E_\phi$.\\
$(3)$ $\FK(\phi)$, $\J_{\phi,\le}$ and $\SCSJ(\phi)$.\\
$(4)$ $\FK(\phi)$, $\J_{\phi,\le}$ and $\CSJ(\phi)$.
\end{thm}

We point out that $\UHK(\phi)$ alone
does not imply the conservativeness of the associated Dirichlet form $(\sE, \sF)$.
For example,
censored (also called resurrected) $\alpha$-stable processes in  upper half spaces
with $\alpha \in (1, 2)$
enjoy  $\UHK(\phi)$ with $\phi(r)=r^\alpha$
but have finite lifetime;  see \cite[Theorem 1.2]{CT}.
 We also note that $\RVD$ are only used in the proofs of $\UHKD(\phi)\Longrightarrow \FK(\phi)$
 and $\J_{\phi,\ge}\Longrightarrow \FK(\phi)$.

\medskip

We emphasize again that in our main results above,
the underlying metric measure space $(M, d,\mu)$ is only assumed to satisfy the general VD and RVD.
We do not assume the uniform comparability of volume of balls; that is, we do not assume
the existence of a non-decreasing function $V$ on $[0, \infty)$ with $V(0)=0$ so that 
 $\mu(B(x,r))\asymp V(r)$ for all $x\in M$ and $r>0$.
Neither do we assume
$M$ to be connected nor $(M, d)$ to be geodesic.

\medskip

As mentioned earlier, parabolic
Harnack inequality is equivalent to the two-sided Aronson type heat kernel estimates
for diffusion processes.
In subsequent papers \cite{CKW,CKW2}, we study stability of parabolic Harnack inequality
and elliptic Harnack inequality respectively
for symmetric jump processes on metric measure spaces.

Let $Z:=\{V_s,X_s\}_{s\ge0}$ be the   space-time process where $V_s=V_0-s$ corresponding to $X$. The filtration generated by $Z$ satisfying the usual conditions will be denoted by $\{\widetilde{\mathcal{F}}_s;s\ge0\}$. The law of the space-time process $s\mapsto Z_s$ starting from $(t,x)$ will be denoted by $\bP^{(t,x)}$. For every open subset $D$ of $[0,\infty)\times M$, define
$\tau_D=\inf\{s>0:Z_s\notin D\}.$

\begin{definition} \rm \begin{description}
\item{(i)} We say that a Borel measurable function $u(t,x)$ on
$[0,\infty)\times M$ is \emph{parabolic} (or \emph{caloric}) on
$D=(a,b)\times B(x_0,r)$ for the process $X$ if there is a properly
exceptional set $\mathcal{N}_u$ associated with the process $X$ so that for every
relatively compact open subset $U$ of $D$,
$u(t,x)=\bE^{(t,x)}u(Z_{\tau_{U}})$ for every $(t,x)\in
U\cap([0,\infty)\times (M\backslash \mathcal{N}_u)).$

\item{(ii)}
We say that the \emph{parabolic Harnack inequality} ($\PHI(\phi)$) holds for the process $X$, if there exist constants $0<C_1<C_2<C_3<C_4$,  $C_5>1$ and  $C_6>0$ such that for every $x_0 \in M $, $t_0\ge 0$, $R>0$ and for
every non-negative function $u=u(t,x)$ on $[0,\infty)\times M$ that is parabolic on cylinder $Q(t_0, x_0,\phi(C_4R),C_5R):=(t_0, t_0+\phi(C_4R))\times B(x_0,C_5R)$,
\be\label{e:phidef}
  \esssup_{Q_- }u\le C_6 \,\essinf_{Q_+}u,
 \ee where $Q_-:=(t_0+\phi(C_1R),t_0+\phi(C_2R))\times B(x_0,R)$ and $Q_+:=(t_0+\phi(C_3R), t_0+\phi(C_4R))\times B(x_0,R)$.
\end{description}
\end{definition}

We note that the above $\PHI(\phi)$ is called a weak parabolic Harnack inequality in \cite{BGK2},
in the sense that \eqref{e:phidef} holds for some $C_1, \cdots, C_5$.
It is called a parabolic Harnack inequality in \cite{BGK2}
if \eqref{e:phidef} holds for any choice of positive constants $C_1, \cdots, C_5$
with $C_6=C_6(C_1, \dots, C_5)<\infty$. Since our underlying metric measure space may not be geodesic,
one can not expect to deduce  parabolic Harnack inequality from weak  parabolic Harnack inequality.

As a consequence of Theorem \ref{T:main} and various
equivalent characterizations of
parabolic Harnack
inequality established in \cite{CKW}, we have  the following.

\begin{theorem}\label{C:1.25} Suppose that the metric measure space  $(M, d,  \mu)$ satisfies $\VD$ and $\RVD$, and $\phi$ satisfies \eqref{polycon}. Then
 $$ \HK (\phi)\Longleftrightarrow\PHI(\phi)+ \J_{\phi,\ge}.$$
 \end{theorem}

Thus for symmetric jump processes, parabolic Harnack inequality
$\PHI (\phi)$ is strictly weaker than $\HK (\phi)$.
This fact was proved
for symmetric
jump processes on graphs with $V(x,r)\asymp r^d$ and $\phi(r)=r^\alpha$ for all $x\in M$, $r>0$ and some
$d\ge 1$, $\alpha\in (0,2)$ in
\cite[Theorem 1.5]{BBK2}.

\medskip

Some of the main results of this paper were presented
at the 38th Conference on Stochastic Processes and their Applications
held at the University of
 Oxford, UK from July 13-17,  2015 and at the International Conference
 on Stochastic Analysis and Related Topics held at Wuhan University, China
 from August 3-8, 2015.  While we were at the final stage of finalizing this paper, we received
a copy  of \cite{MS1,MS2} from M. Murugan.
Stability of discrete-time long range random walks of stable-like jumps
on infinite connected locally finite graphs  is studied in \cite{MS2}.
Their results are quite similar to ours when specialized to the case
of $\phi (r)=r^\alpha$ but the techniques and the settings are somewhat different.
They work on discrete-time random walks on infinite connected locally finite
graphs equipped with graph distance,
while we work on
continuous-time symmetric jump processes on
general metric measure space  and with much
more general jumping mechanisms. Moreover, it is assumed in \cite{MS2} that there is a constant $c\geq 1$ so that
$c^{-1}\leq \mu(\{x\})\leq c$ for every $x\in M$ and the $d$-set condition that there are constants $C\geq 1$ and $d_f>0$
so that  $ C^{-1}r^{d_f} \leq V(x, r) \leq C r^{d_f}$ for every $x\in M$ and $r\geq 1$,
while we only assume general VD and RVD.
 Technically, their approach is
 to generalize the so-called Davies' method (to obtain
the off-diagonal heat kernel upper bound from the on-diagonal upper bound) to be applicable
when $\alpha>2$ under the assumption of cut-off Sobolev inequalities.
Quite recently, we also learned from A. Grigor'yan \cite{GHH} that they are also
working on the same topic of this paper on metric measure spaces with the $d$-set condition
and the conservativeness assumption on $(\sE, \sF)$.
Their results are also quite similar to ours, again specialized to the case of $\phi (r)=r^\alpha$, but
the techniques are also somewhat different.
Their approach \cite{GHH} is
 to deduce a kind of weak Harnack inequalities
first from $\J_{\phi}$ and $\CSJ(\phi)$,
which they call generalized capacity condition.
They then obtain uniform H\"older continuity of harmonic functions,
which plays the key role for
 them to obtain the near-diagonal lower heat kernel bound that corresponds to \eqref{NDLdf1-ww}.
 As we see below, our approach is different from theirs.
We emphasize here that in this paper we do not assume a priori that $(\sE, \sF)$ is conservative.

\medskip

The rest of the paper is organized as follows.
In the next section, we present some preliminary results about $\J_{\phi,\le}$ and $\CSJ(\phi)$. In particular, in Proposition \ref{L:cswk} we show that the leading  constant in $\CSJ(\phi)$ is self-improving. Sections 3, 4 and 5 are devoted to the proofs of  (1) $\Longrightarrow$ (3),  (4) $\Longrightarrow$ (2) and (2) $\Longrightarrow$ (1) in Theorems \ref{T:main} and \ref{T:main-1}, respectively. Among them, Section 4 is the most difficult part, where in Subsection \ref{s:cac} we establish the Caccioppoli inequality
and the $L^p$-mean value inequality
for subharmonic functions associated with symmetric jump processes, and in Subsection \ref{FKEPHIj} Meyer's decomposition is realized for jump processes in the
$\VD$ setting.
Both subsections are of interest in their own.
In Section 6, some examples are
given to illustrate
the applications of our results, and a counterexample is also given to indicate that $\CSJ(\phi)$ is necessary for $\HK(\phi)$ in general setting.
For reader's convenience,  some known facts used in this paper are streamlined and collected in
Subsections 7.1-7.4
of the Appendix.
In connection with the implication of $(3) \Longrightarrow (1)$ in Theorem \ref{T:main-1},
we show in Subsection \ref{S:7.5} that
{ $\SCSJ(\phi)+ \J_{\phi,\le}\Longrightarrow
(\sE, \sF) $  is conservative; in other words $\FK (\phi)$ is not needed for establishing the conservativeness of $(\sE, \sF)$.
 We remark that,
 in order to increase the readability of the paper,
we have tried to make
the paper as self-contained as possible. Figure \ref{diagfig} illustrates
implications of various conditions and flow of our proofs.

\begin{figure}[t]
\centerline{\epsfig{file=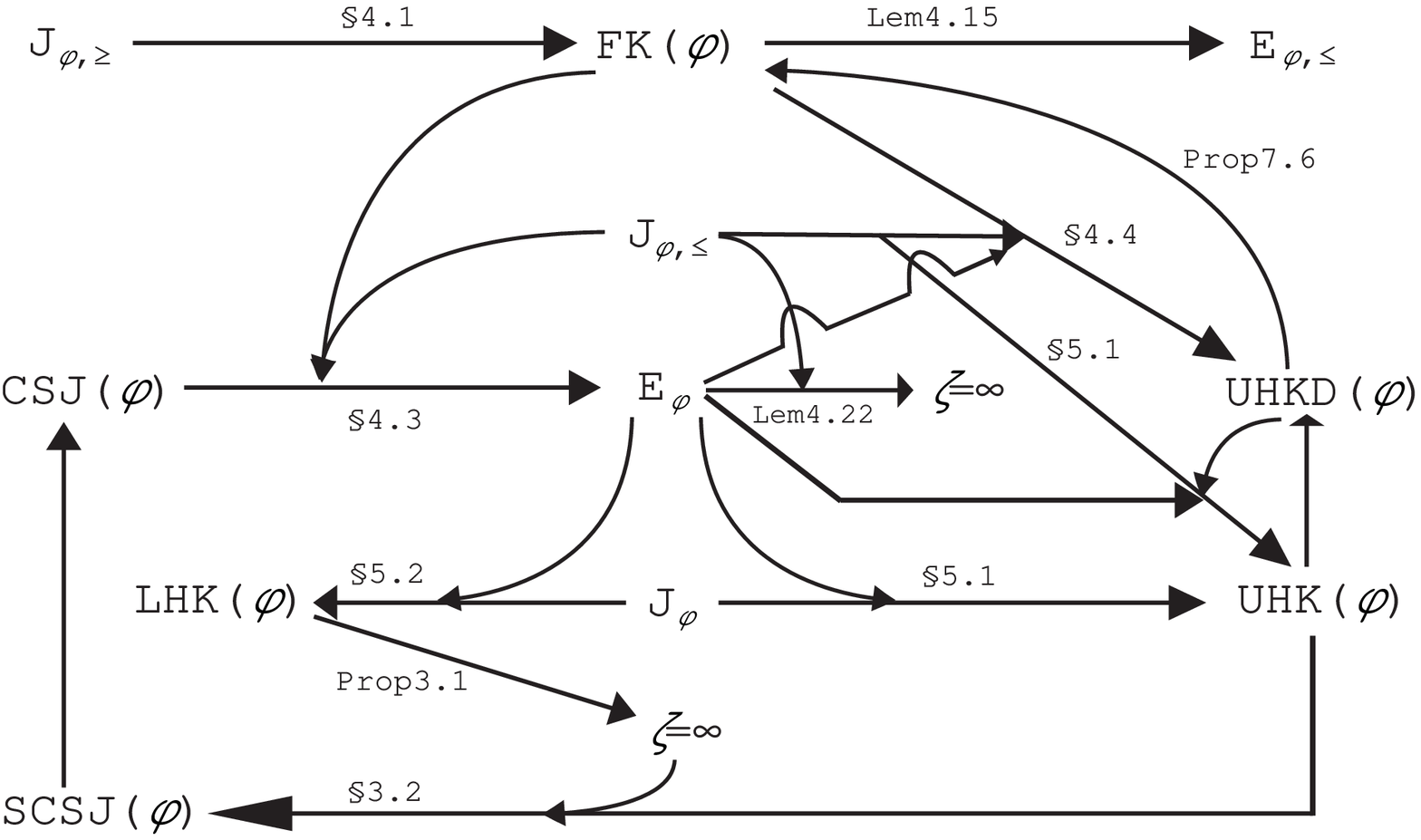, height=3in}}
\caption{diagram}\label{diagfig}
\end{figure}

Throughout this paper, we will use $c$, with or without subscripts,
to denote strictly positive finite constants whose values are
insignificant and may change from line to line.
For $p\in [1, \infty]$, we will use $\| f\|_p$ to denote the $L^p$-norm in $L^p(M;\mu)$.
For $B=B(x_0, r)$ and $a>0$, we use $a B$ to denote
the ball $B(x_0, ar)$,
and $\bar{B}:=\{x\in M:d(x,x_0)\le r\}$. For any subset $D$ of $M$,
$D^c$ denotes the complement of $D$
in $M$.

\section{Preliminaries}\label{section2}

For basic properties and definitions related to Dirichlet forms,
such as the relation between regular Dirichlet forms and Hunt
processes, associated semigroups, resolvents, capacity and
quasi-continuity, we refer the reader to \cite{CF, FOT}.

We begin with the following estimate, which is essentially given in \cite[Lemma 2.1]{CK2}.

 \begin{lemma}\label{intelem} Assume that $\VD$ and
 \eqref{polycon} hold. Then there exists a constant $c_0>0$ such that
\be \label{e:2.1}
\int_{B(x,r)^c}\frac1{V(x,d (x, y))\,\phi (d(x,y))}\,\mu (dy) \le \frac{c_0}{\phi (r)}
 \quad\mbox{for every }
x\in M \hbox{ and } r>0.
\ee
Thus if, in addition,   $\J_{\phi,\le}$ holds,
then there exists a constant $c_1>0$ such that
\[
 \int_{B(x,r)^c}
J(x,y)\,\mu (dy)\le \frac{c_1}{\phi (r)}
\quad\mbox{for every }
x\in M \hbox{ and } r>0.
\]
\end{lemma}

 \begin{proof}
For completeness, we present a proof here.
By $\J_{\phi,\le}$ and $\VD$, we have for every $x\in M$ and $r>0$,
\begin{align*}
& \int_{B(x,r)^c}\frac1{ V(x,d (x, y))\,\phi (d(x,y))}\,\mu (dy)\\
&= \sum_{i=0}^\infty\int_{B(x,2^{i+1}r)\backslash B(x,2^ir)}\frac1{V(x,d (x, y))\,\phi (d(x,y))}\,\mu (dy)\\
&\le  \sum_{i=0}^\infty \frac1{ V (x,2^ir)\,\phi (2^ir)}
   V (x,2^{i+1}r)\\
&\le c_2 \sum_{i=0}^\infty\frac{1} {\phi (2^ir)}
\le  \frac{c_3}{\phi (r)} \sum_{i=0}^\infty2^{-i\beta_1} \le  \frac{c_4}{\phi (r)},\end{align*}
where the lower bound in (\ref{polycon}) is used in the second to the last inequality.
\qed\end{proof}

Fix $\rho>0$ and define  a bilinear form $(\sE^{(\rho)}, \sF)$ by
\be\label{eq:rhoEdef}
\sE^{(\rho)}(u,v)=\int (u(x)-u(y))(v(x)-v(y)){\bf 1}_{\{d(x,y)\le \rho\}}\, J(dx,dy).
\ee
Clearly, the form $\sE^{(\rho)}(u,v)$ is well defined for $u,v\in \sF$, and $\sE^{(\rho)}(u,u)\le \sE(u,u)$ for all $u\in \sF$. Assume that $\VD$, \eqref{polycon}  and $\J_{\phi,\le}$ hold. Then  we have by  Lemma \ref{intelem}
  that for all $u\in \sF$,
\begin{equation}\label{dfcomp}\begin{split}
   \sE(u,u)-\sE^{(\rho)}(u,u)&= \int(u(x)-u(y))^2{\bf 1}_{\{d(x,y)>\rho\}}\,J(dx,dy)\\
&\le 4\int_Mu^2(x)\,\mu(dx)\int_{B(x,\rho)^c}J(x,y)\,\mu(dy)\le \frac{c_0\|u\|_{2}^2 }{\phi(\rho)}.
\end{split}\end{equation}
Thus $\sE_1 (u, u)$ is equivalent to $\sE^{(\rho)}_1(u,u):=
\sE^{(\rho)}(u,u)+ \|u\|_2^2$ for every $u\in \sF$. Hence $(\sE^{(\rho)},
\sF)$ is a regular Dirichlet form on $L^2(M; \mu)$.
Throughout this paper, we call $(\sE^{(\rho)}, \sF)$  $\rho$-truncated
Dirichlet form. The Hunt process associated with $(\sE^{(\rho)}, \sF)$ can be identified in distribution with
the Hunt process of the original Dirichlet form $(\sE, \sF)$ by removing those jumps of size larger than
 $\rho$.

 Assume that $\J_{\phi,\le}$ holds, and in particular \eqref{e:1.2} holds.
 Define
 $J(x,dy)=J(x,y)\,\mu(dy)$.
   Let $J^{(\rho)}(dx,dy)={\bf 1}_{\{d(x,y)\le \rho\}} J(dx,dy)$, $J^{(\rho)}(x,dy)={\bf 1}_{\{d(x,y)\le \rho\}} J(x,dy)$, and  $\Gamma^{(\rho)}(f,g)$ be the
    carr\'e du champ of the $\rho$-truncated Dirichlet form $(\sE^{(\rho)}, \sF)$; namely,
$$ \sE^{(\rho)}(f,g)=\int_M \,\mu(dx)\int_M(f(x)-f(y))(g(x)-g(y))\,J^{(\rho)}(x,dy)
=:\int_M d\Gamma^{(\rho)}(f,g). $$

We now define variants of $\CSJ(\phi)$.

\begin{definition} \rm
Let $\phi$ be an increasing function on $\bR_+$ with $\phi(0)=0$, and $C_0\in (0, 1]$.  For any $x_0\in M$ and $0<r\le R$, set
$U=B(x_0,R+r)\setminus B(x_0,R)$, $U^*=B(x_0,R+(1+C_0)r)\setminus B(x_0,R-C_0r)$ and ${U^*}'=B(x_0,R+2r)\setminus B(x_0,R-r)$.
\begin{description}
\item{(i)}  We say that condition $\CSJ^{(\rho)}(\phi)$ holds if the following holds: there exist constants $C_0\in (0,1]$ and $C_1, C_2>0$ such that for every
$0<r\le R$, almost all $x_0\in M$ and any $f\in \sF$, there exists
a cut-off function $\vp\in \sF_b$ for $B(x_0,R) \subset B(x_0,R+r)$ so that the following holds for all $\rho>0$:
\be \label{e:csj2} \begin{split}
 \int_{B(x_0,R+(1+C_0)r)} f^2 \, d\Gamma^{(\rho)}(\vp,\vp)
\le &C_1 \int_{U\times U^*}(f(x)-f(y))^2\,J^{(\rho)}(dx,dy) \\
&+ \frac{C_2}{\phi(r\wedge \rho)}  \int_{B(x_0,R+(1+C_0)r)} f^2  \,d\mu.\end{split}
\ee

\item{(ii)}  We say that condition $\CSAJ(\phi)$ holds if there exist constants $C_0\in (0,1]$ and $C_1, C_2>0$
such that for every
$0<r\le R$, almost all $x_0\in M$ and any $f\in \sF$, there exists
a cut-off function $\vp\in \sF_b$ for $B(x_0,R) \subset B(x_0,R+r)$ so that the following holds for all $\rho>0$:
\be \label{e:csaj1} \begin{split}
 \int_{U^*} f^2 \, d\Gamma (\vp,\vp)
\le &C_1 \int_{U\times U^*}(f(x)-f(y))^2\,J(dx,dy)+ \frac{C_2}{\phi(r)}  \int_{U^*} f^2  \,d\mu.\end{split}
\ee

\item{(iii)} We say that condition $\CSAJ^{(\rho)}(\phi)$ holds if the following holds: there exist constants $C_0\in (0,1]$ and $C_1, C_2>0$
such that for every
$0<r\le R$, almost all $x_0\in M$ and any $f\in \sF$, there exists
a cut-off function $\vp\in \sF_b$ for $B(x_0,R) \subset B(x_0,R+r)$ so that the following holds for all $\rho>0$:
$$
 \int_{U^*} f^2 \, d\Gamma^{(\rho)}(\vp,\vp)
\le c_1 \int_{U\times U^*}(f(x)-f(y))^2\,J^{(\rho)}(dx,dy) + \frac{C_2}{\phi(r\wedge \rho)}  \int_{U^*} f^2  \,d\mu.
$$

\item{(iv)} We say that condition $\CSJ^{(\rho)}(\phi)_+$
holds if the following holds: for any $\varepsilon>0$,
there exists a constant $ c_1(\varepsilon)>0$
such that for every $0<r\le R$, almost all $x_0 \in M$ and any $f\in \sF$, there exists
a cut-off function $\vp\in \sF_b$ for $B(x_0,R) \subset B(x_0,R+r)$ so that the following holds for all $\rho>0$:
\be \label{e:csj3} \begin{split}
 \int_{B(x_0,R+2r)} f^2 \, d\Gamma^{(\rho)}(\vp,\vp)
\le&  \varepsilon \int_{U\times {U^*}'} \vp^2(x)(f(x)-f(y))^2\, J^{(\rho)}(dx,dy) \\
&
+ \frac{c_1 (\eps) }{\phi(r\wedge \rho)}  \int_{B(x_0,R+2r)} f^2 \, d\mu.\end{split}
\ee

\item{(v)} We say that condition $\CSAJ^{(\rho)}(\phi)_+$ holds if the following holds: for any $\varepsilon>0$,
there exists a constant $ c_1(\varepsilon)>0$
such that for every $0<r\le R$, almost all $x_0 \in M$ and any $f\in \sF$, there exists
a cut-off function $\vp\in \sF_b$ for $B(x_0,R) \subset B(x_0,R+r)$ so that the following holds for all $\rho>0$:
\begin{align*}
 \int_{{U^*}'} f^2 \, d\Gamma^{(\rho)}(\vp,\vp)
\le & \varepsilon \int_{U\times {U^*}'} \vp^2(x)\, (f(x)-f(y))^2\, J^{(\rho)}(dx,dy) + \frac{c_1 (\eps)}{\phi(r\wedge \rho)}  \int_{{U^*}'} f^2 \, d\mu.
\end{align*}
\end{description}
\end{definition}

For open subsets $A$ and $B$ of $M$ with $A\subset B$, and for any $\rho>0$, define
\[
\mbox{{\rm Cap}}^{(\rho)} (A,B)=\inf\{\sE^{(\rho)}(\vp,\vp): \vp\in \sF, ~\vp|_A=1,~ \vp|_{B^c}=0\}.
\]

\begin{proposition}\label{CSJ-equi}
Let $\phi$ be an increasing function on $\bR_+$.
Assume that $\VD$, \eqref{polycon} and $\J_{\phi,\le}$ hold. The following hold.
\begin{description}
\item{\rm (1)} $\CSJ(\phi)$ is equivalent to $\CSJ^{(\rho)}(\phi)$.

\item{\rm (2)} $\CSJ(\phi)$ is implied by $\CSAJ(\phi)$.

\item{\rm (3)}  $\CSAJ(\phi)$ is equivalent to $\CSAJ^{(\rho)}(\phi)$.

\item{\rm (4)}  If $\CSJ^{(\rho)}(\phi)$ $($resp.\ $\CSAJ^{(\rho)}(\phi)$$)$ holds for some $C_0\in(0,1]$,
then for any $C_0'\in[C_0,1]$, there exist constants $C_1,C_2>0$ $($where $C_2$ depends on $C_0'$$)$ such that
$\CSJ^{(\rho)}(\phi)$ $($resp.\ $\CSAJ^{(\rho)}(\phi)$$)$ holds for $C_0'$.

\item{\rm  (5)} If $\CSJ(\phi)$ holds, then there is a constant $c_0>0$ such that for  every
$0<r\le R$, $\rho>0$ and almost all $x\in M$,
$$
\mbox{\rm Cap} ^{(\rho)} (B(x,R),B(x,R+r))\le c_0\frac{V(x,R+r)}{\phi(r\wedge \rho)}.
$$
In particular, we have
\be\label{eq:fneoobo3}
\mbox{\rm Cap} (B(x,R),B(x,R+r))\le c_0\frac{V(x,R+r)}{\phi(r)}.
\ee
\end{description}
\end{proposition}

\begin{proof} (1)
Letting $\rho\to\infty$, we see that \eqref{e:csj2} implies \eqref{e:csj1}. Now, let $\rho>0$ and assume that \eqref{e:csj1} holds. Then
there exist constants $C_0\in (0,1]$ and $C_1,C_2>0$ such that for every $0<r\le R$, almost all $x_0\in M$ and any $f\in \sF$, there exists a cut-off function $\vp\in \sF_b$ for $B(x_0,R)\subset B(x_0, R+r)$ such that
 \begin{align*}
&\int_{{B(x_0,R+(1+C_0)r)}} f^2 \, d\Gamma^{(\rho)}(\vp,\vp) \\
&\le \int_{B(x_0,R+(1+C_0)r)} f^2 \, d\Gamma(\vp,\vp)\\
& \le  C_1 \int_{U\times U^*}(f(x)-f(y))^2\,J(dx,dy)+ \frac{C_2}{\phi(r)}  \int_{B(x_0,R+(1+C_0)r)} f^2  \,d\mu\\
& \le C_1 \int_{U\times U^*}(f(x)-f(y))^2\,J^{(\rho)}(dx,dy)
+ 2 C_1 \int_{U\times U^*}(f^2(x)+f^2(y)){\bf1}_{\{d(x,y)>\rho\}}\,J(dx,dy)
 \\
& +\frac{C_2}{\phi(r)}  \int_{B(x_0,R+(1+C_0)r)} f^2  \,d\mu\\
&  \le C_1 \int_{U\times U^*}(f(x)-f(y))^2\,J^{(\rho)}(dx,dy)
  +\frac{C_3}{\phi(r\wedge \rho)}  \int_{B(x_0,R+(1+C_0)r)} f^2  \,d\mu,
\end{align*}
where Lemma \ref{intelem} is used in the last inequality.

(2) Fix $x_0\in M$, $0<r\le R$ and $C_0\in (0,1]$. Let $\varphi\in \sF_b$ be a cut-off function for $B(x_0,R) \subset B(x_0,R+r)$. Since $\vp(x)=1$ on $x\in B(x_0,R)$,  we have for $f \in \sF$,
\begin{align*}
 \int_{B(x_0,R-C_0r)} f^2 \, d\Gamma(\vp,\vp)
 &= \int_{B(x_0,R-C_0r)} f^2(x)\,\mu(dx)\int_M(1-\vp(y))^2J(x,y)\,\mu(dy)\\
 &\le \int_{B(x_0,R-C_0r)} f^2(x)\,\mu(dx)\int_{B(x_0,R)^c}J(x,y)\,\mu(dy)\\
 &\le \int_{B(x_0,R-C_0r)} f^2(x) \,\mu(dx)\int_{B(x,C_0r)^c} J(x,y)\,\mu(dy)\\
 &\le \frac {c_1}{\phi(C_0r)}\int_{B(x_0,R-C_0r)} f^2\,d\mu \\
& \le  \frac {c_2}{\phi(r)}\int_{B(x_0,R-C_0r)} f^2\,d\mu,
\end{align*}
where we used Lemma \ref{intelem} and \eqref{polycon} in the last two inequalities.
This together with \eqref{e:csaj1} gives us the desired conclusion.

(3) This can be proved in the same way as (1).

(4) This is easy. Indeed, for $x_0\in M$, $0<r\le R$, $C_0\in (0,1]$ and $C_0'\in [C_0,1]$, set $D_1=B(x_0,R+(1+C_0')r)\setminus B(x_0,R+(1+C_0)r)$ and $D_2=B(x_0,R-C_0r)\setminus B(x_0,R-C_0'r)$. Let $\varphi\in \sF_b$ be a cut-off function for $B(x_0,R) \subset B(x_0,R+r)$.  Then  for any $f\in \sF$ and $\rho>0$,
\begin{align*}
\int_{D_1} f^2 \, d\Gamma^{(\rho)}(\vp,\vp)&
=\int_{D_1} f^2(x) \,\mu(dx)\int_{B(x_0,R+r)}\vp^2(y)J^{(\rho)}(x,y)\,\mu(dy)\\
&\le \int_{D_1} f^2(x) \,\mu(dx)\int_{B(x,C_0r)^c}J(x,y)\,\mu(dy) \\
&\le \frac {c_1}{\phi(r)}\int_{D_1} f^2 \,d\mu,
\end{align*}
where Lemma \ref{intelem} and \eqref{polycon} are used in the last inequality. Similarly, for any $f\in \sF$ and $\rho>0$,
$$\int_{D_2} f^2 \, d\Gamma^{(\rho)}(\vp,\vp)\le \frac {c_2}{\phi(r)}\int_{D_2} f^2 \,d\mu.$$
From both inequalities above we can get the desired assertion for $C_0'\ge C_0$.

(5) In view of (1) and (4),   $\CSJ^{(\rho)}(\phi)$ holds for every $\rho >0$ and we
can and do take
$C_0=1$ in \eqref{e:csj1}. Fix $x_0\in M$ and
write $B_s:=B(x_0,s)$ for $s\ge 0$. Let
$f\in \sF$ with $0\le f\le 1$ such that $f|_{B_{R+2r}}=1$ and $f|_{B_{R+3r}^c}=0$. For any $\rho>0$ and $0<r\le R$, let $\vp\in \sF_b$ be the cut-off function for $B_R\subset B_{R+r}$ associated with $f$ in $\CSJ^{(\rho)}(\phi)$. Then
\begin{align*}
\mbox{{\rm Cap}}^{(\rho)} (B_R, B_{R+r})\le& \int_{B_{R+2r}}\,d\Gamma^{(\rho)}(\vp,\vp)+\int_{B_{R+2r}^c}\,d\Gamma^{(\rho)}(\vp,\vp)\\
=&\int_{B_{R+2r}}f^2\,d\Gamma^{(\rho)}(\vp,\vp)+\int_{B_{R+2r}^c}\,d\Gamma^{(\rho)}(\vp,\vp)\\
\le &c_1\int_{(B_{R+r}\setminus B_R)\times (B_{R+2r}\backslash B_{R-r})}(f(x)-f(y))^2\,J^{(\rho)}(dx,dy)\\
&+\frac{c_2}{\phi(r\wedge \rho)}\int_{B_{R+2r}}f^2\,d\mu+\int_{B_{R+2r}^c}\,\mu(dx)\int_{B_{R+r}} \vp^2(y)J(x,y)\,\mu(dy)\\
\le&\frac{c_2\mu(B_{R+2r})}{\phi(r\wedge \rho)}+\frac{c_3\mu(B_{R+r})}{\phi(r)}\\
\le &\frac{c_4\mu(B_{R+r})}{\phi(r\wedge \rho)},
\end{align*}
where we used $\CSJ^{(\rho)}(\phi)$ in the second inequality and Lemma
\ref{intelem} with $\VD$ in the third inequality.

Now let $f_\rho$ be
the potential whose $\sE^{(\rho)}$-norm gives the capacity. Then the
Ces\`{a}ro mean of a subsequence of $f_\rho$ converges in
$\sE_1$-norm, say to $f$, and $\sE(f, f)$ is no less than the
capacity corresponding to $\rho =\infty$.  So \eqref{eq:fneoobo3} is
proved.\qed\end{proof}

We next show that the leading  constant in $\CSJ^{(\rho)}(\phi)$ (resp. $\CSAJ^{(\rho)}(\phi)$) is self-improving
in the following sense.

\begin{proposition} \label{L:cswk}
Suppose that $\VD$, \eqref{polycon} and $\J_{\phi,\le}$ hold. Then
the following hold.
\begin{description}
\item{\rm(1)} $\CSJ^{(\rho)}(\phi)$ is equivalent to $\CSJ^{(\rho)}(\phi)_+$.
\item{\rm(2)} $\CSAJ^{(\rho)}(\phi)$ is  equivalent to
$\CSAJ^{(\rho)}(\phi)_+$.
\end{description}
\end{proposition}

\begin{proof} We only prove (1), since (2) can be verified similarly. It is clear that $\CSJ^{(\rho)}(\phi)_+$
implies
$\CSJ^{(\rho)}(\phi)$. Below, we assume that $\CSJ^{(\rho)}(\phi)$ holds.

Fix $x_0\in M$, $0<r\le R$ and $f\in \sF$. For $s >0$, set $B_{s}=B(x_0,s)$.
The goal is to construct a cut-off function $\vp\in \sF_b$ for  $B_R
\subset B_{R+r}$
so that \eqref{e:csj3} holds.
Without loss of generality, in the following we may and do assume that $\int_{B_{R+2r}} f^2\,d\mu>0$; otherwise, \eqref{e:csj3} holds trivially.

For $\lam>0$ whose exact value to be determined later,
let
$$ s_n = c_0 r e^{- n \lam/(2\beta_2)}, $$
where $c_0:=c_0(\lam)$ is chosen so that $\sum_{n=1}^\infty s_n =r $ and $\beta_2$
is given in \eqref{polycon}.
Set $r_0=0$ and
$$ r_n = \sum_{k=1}^n s_k,\quad n\ge 1.$$
Clearly, $R < R+r_1 < R+r_2 < \dots < R+r$.
For any $n\ge 0$, define $U_n:= B_{R+r_{n+1}} \setminus B_{R+r_n}$, and
$U_n^*:=B_{R+r_{n+1}+s_{n+1} } \setminus  B_{R+r_n-s_{n+1}}$.
Let $\theta>0$,
whose value also to be determined later,
 and define $f_\theta:=|f|+\theta$.
By $\CSJ^{(\rho)}(\phi)$ (with $C_0=1$; see Proposition \ref{CSJ-equi}\,(4)), there exists a cut-off function $\vp_n$ for
$B_{R+r_n} \subset B_{R+r_{n+1}}$
such that
\be \label{e:cswk}\begin{split}
 \int_{B_{R+r_{n+1}+s_{n+1}}}   f_\theta^2 \, d\Gamma^{(\rho)}(\vp_n,\vp_n)
\le &C_1 \int_{U_n\times U_{n}^*} (f_\theta(x)-f_\theta(y))^2\, J^{(\rho)}(dx,dy)  \\
&+ \frac{C_2 }{\phi(s_{n+1}\wedge \rho)} \int_{B_{R+r_{n+1}+s_{n+1}}}   f_\theta^2 \,d\mu,\end{split}
\ee
where $C_1,C_2$ are positive constants independent of $f_\theta$ and $\vp_n$.
Here, we mention that since $(\sE, \sF)$ is a regular Dirichlet form on $L^2(M,\mu)$, $f_\theta\in \sF_{loc}$, and so, by Remark \ref{rek:scj}(v), $\CSJ^{(\rho)}(\phi)$ can apply to $f_\theta.$

Let $b_n = e^{-n \lam}$ and define
\be \label{phi}
 \vp = \sum_{n=1}^\infty (b_{n-1} - b_n) \vp_n.
\ee
 Then  $\vp$ is a cut-off
function for $B_R \subset B_{R+r}$, because $\vp=1$ on $B_R$ and  $\vp=0$ on $B_{R+r}^c$.
On $U_n$ we have
$ \vp = (b_{n-1}-b_n) \vp_n + b_n$,
so that $b_n \le \vp \le b_{n-1}$ on $U_n$. In particular, on $U_n$
\be \label{e:vpne}
 b_{n-1}-b_n  \le  \frac{ \vp (b_{n-1}-b_n)}{b_n} = (e^\lam -1)\vp.
\ee

Below, we verify that the function $\vp$ defined by \eqref{phi} satisfies \eqref{e:csj3} and $\vp\in \sF_b$.
For this, we will make a non-trivial and substantial modification
of the proof of \cite[Lemma 5.1]{AB}.
Set $$F_{n,m}(x,y)=
f^2_\theta(x)(\vp_n(x)-\vp_n(y))(\vp_m(x)-\vp_m(y))$$ for any $n,m\ge 1$.  Then
\begin{align*}
 ~ \int_{B_{R+2r}} f_\theta^2 \, d\Gamma^{(\rho)}(\vp,\vp)
 =&   \int_{B_{R+2r}} f_\theta^2(x)\int_M\Big(\sum_{n=1}^\infty(b_{n-1}-b_n)(\vp_n(x)-\vp_n(y))\Big)^2\,J^{(\rho)}(dx,dy)\\
\le &\int_{B_{R+2r}} \int_M\bigg[2\sum_{n=3}^\infty\sum_{m=1}^{n-2}
(b_{n-1}-b_n)(b_{m-1}-b_m)F_{n,m}(x,y)\\
&\qquad\qquad\quad+2\sum_{n=2}^\infty(b_{n-1}-b_n)(b_{n-2}-b_{n-1})F_{n,n-1}(x,y)\\
&\qquad\qquad\quad+\sum_{n=1}^\infty(b_{n-1}-b_n)^2F_{n,n}(x,y)\bigg]\,J^{(\rho)}(dx,dy)\\
=&:I_1+I_2+I_3.
\end{align*}
For $n\ge m+2$, since $F_{n,m}(x,y)=0$ for $x,y\in B_{R+r_n}$ or $x,y\notin B_{R+r_{m+1}}$, we can deduce that
$F_{n,m}(x,y)\ne 0$ only if $x\in B_{R+r_{m+1}}, y\notin B_{R+r_n}$ or $x\notin B_{R+r_n}, y\in B_{R+r_{m+1}}$. Since
$|F_{n,m}(x,y)|\le f_\theta^2(x)$, using Lemma \ref{intelem}, we have
\begin{equation}\label{eq:noe20ed}\begin{split}
 \int_{B_{R+2r}}\int_M & F_{n,m}(x,y)\,J^{(\rho)}(dx,dy)\\
=&  \int_{B_{R+2r} \cap B_{R+r_{m+1}}}  \int_{B_{R+r_n}^c}\cdots+ \int_{B_{R+2r} \cap B_{R+r_n}^c}\int_{B_{R+r_{m+1}}}
\cdots\\
\le & \frac c{\phi(\sum_{k=m+2}^ns_k)}\int_{B_{R+2r}}f_\theta^2(x)\,\mu(dx)\\
\le & \frac c{\phi(s_{m+2})}\int_{B_{R+2r}} f_\theta^2(x)\,\mu(dx).
\end{split}
\end{equation}
Note that, according to \eqref{polycon}, we have
$$ \frac{ \phi(r)}{\phi( s_{k+2})}
 \le c'\Big(\frac{r}{c_0(\lam) r e^{- (k+2) \lam /(2\beta_2)}}\Big)^{\beta_2}
 = c'\frac{ e^{\lam} e^{k\lam/2} } {c_0(\lam)^{\beta_2}}
= \frac{c' { e^\lam }(e^{\lam}-1)^{1/2} }{ c_0(\lam)^{\beta_2}(b_{k-1}-b_k )^{1/2}}.
$$
Therefore,
\be\label{eq:noe210}
  (b_{k-1}-b_k)^{1/2} \phi(s_{k+2})^{-1} \le  c_1(\lam) \phi(r)^{-1}.
\ee
This together with \eqref{eq:noe20ed} implies
\begin{align*}
I_1 &\le 2\sum_{n=3}^\infty\sum_{m=1}^{n-2}(b_{n-1}-b_n)(b_{m-1}-b_m)
\frac c{\phi(s_{m+2})}\int_{B_{R+2r}}f_\theta^2(x)\,\mu(dx)\\
&\le \sum_{n=3}^\infty\sum_{m=1}^{n-2}(b_{n-1}-b_n)(b_{m-1}-b_m)^{1/2}
\frac {c_2(\lam)}{\phi(r)}\int_{B_{R+2r}}f_\theta^2(x)\,\mu(dx)\\
&\le \frac {c_3(\lam)}{\phi(r)}\int_{B_{R+2r}}f_\theta^2(x)\,\mu(dx),
\end{align*}
because $\sum_{m=1}^\infty(b_{m-1}-b_m)^{1/2}=c_4(\lam)$ and
$\sum_{n=1}^\infty(b_{n-1}-b_n)=1$. For $I_2$, by the Cauchy-Schwarz inequality, we have
\begin{align*}
I_2&\le 2\sum_{n=2}^\infty
\Big(\int_{B_{R+2r}}\int_M(b_{n-1}-b_n)^2F_{n,n}(x,y)^2\,J^{(\rho)}(dx,dy)\Big)^{1/2}\\
&\qquad \qquad\times \Big(\int_{B_{R+2r}}\int_M(b_{n-2}-b_{n-1})^2F_{n-1,n-1}(x,y)^2\,J^{(\rho)}(dx,dy)\Big)^{1/2}\\
&\le 2\,I_3,
\end{align*}
where we used $2(ab)^{1/2}\le a+b$ for $a,b\ge 0$ in the last inequality. For $I_3$,
\begin{align*}
&\int_{B_{R+2r}}\int_M F_{n,n}(x,y)\,J^{(\rho)}(dx,dy)\\
&=\Big(\int_{B_{R+r_{n+1}+s_{n+1}}} \int_M+\int_{B_{R+2r} \setminus B_{R+r_{n+1}+s_{n+1}} }
\int_M\Big) F_{n,n}(x,y)\,J^{(\rho)}(dx,dy)\\
&\le \int_{B_{R+r_{n+1}+s_{n+1}}}  \int_M F_{n,n}(x,y)\,J^{(\rho)}(dx,dy)+\frac c{\phi(s_{n+1})}
   \int_{B_{R+2r} } f_\theta^2(x)\,\mu(dx)\\
&\le  C_1 \int_{U_n\times U_{n}^*}(f_\theta(x)-f_\theta(y))^2\,J^{(\rho)}(dx,dy)
  + \frac {c+C_2}{\phi(s_{n+1}\wedge \rho)} \int_{B_{R+2r}}  f_\theta^2 (x)\, \mu (dx),
\end{align*}
where we used Lemma \ref{intelem} in the second line and \eqref{e:cswk} in the last line.
Using \eqref{e:vpne} and \eqref{eq:noe210}, and noting that $s_{k+1}\ge s_{k+2}$ and
$\sum_{m=1}^\infty(b_{m-1}-b_m)^{3/2}+\sum_{m=1}^\infty(b_{m-1}-b_m)^{2}=c_5(\lam)$, we have
\[
I_3\le C_3(e^\lam-1)^2 \int_{U\times {U^*}'} (f_\theta(x)-f_\theta(y))^2\,J^{(\rho)}(dx,dy) + \frac{c_6(\lam)}{\phi(r\wedge \rho)}
 \int_{B_{R+2r}} f_\theta^2\, d\mu ,
\]
where we used the facts that $\{U_n; n\geq 1\}$ are disjoint, $\bigcup_{n=1}^\infty  U_n=U$, and $U_{n}^*\subset {U^*}'$ for all $n\ge1$.
For any $\varepsilon>0$, we now choose $\lam$ so that $3C_3 (e^\lam-1)^2 =\varepsilon$, and obtain \eqref{e:csj3}
for $f_\theta$, i.e.,\
\begin{equation}\label{osos}\begin{split}\int_{B_{R+2r}} f_\theta^2 \, d\Gamma^{(\rho)}(\vp,\vp)\le& \varepsilon \int_{U\times {U^*}'} (f_\theta(x)-f_\theta(y))^2\,J^{(\rho)}(dx,dy) \\
&+ \frac{C_4(\varepsilon)}{\phi(r\wedge \rho)}
 \int_{B_{R+2r}} f_\theta^2(x)\, \mu (dx),\end{split}\end{equation} where the positive constant $C_4(\varepsilon)$ is independent of $\theta$.
It is clear that the left hand side of \eqref{osos} is bigger than $\int_{B_{R+2r}} f^2 \, d\Gamma^{(\rho)}(\vp,\vp)$. On the other hand, since for any $x,y\in M$ and $\theta>0$,
$|f_\theta(x)-f_\theta(y)|\le \big| |f|(x)-|f|(y)\big|\le |f(x)-f(y)|,$ it holds that
$$\int_{U\times {U^*}'} (f_\theta(x)-f_\theta(y))^2\,J^{(\rho)}(dx,dy)\le \int_{U\times {U^*}'} (f(x)-f(y))^2\,J^{(\rho)}(dx,dy).$$
Note that $$\int_{B_{R+2r}} f_\theta^2\, d\mu \le 2\left(\int_{B_{R+2r}} f^2\, d\mu +\theta^2 \mu(B_{R+2r})\right).$$ Then, by choosing $$\theta= \left(\frac{\int_{B_{R+2r}} f^2\, d\mu}{\mu(B_{R+2r})}\right)^{1/2}>0,$$ we have $$\int_{B_{R+2r}} f_\theta^2\, d\mu \le 4\int_{B_{R+2r}} f^2\, d\mu.$$ Hence, for this choice of $\theta$, we know that the left hand side of \eqref{osos} is smaller than
$$\varepsilon \int_{U\times {U^*}'} (f(x)-f(y))^2\,J^{(\rho)}(dx,dy) + \frac{4C_4(\varepsilon)}{\phi(r\wedge \rho)}
 \int_{B_{R+2r}} f^2(x)\, \mu (dx).$$ Combining both estimates above, we prove that \eqref{e:csj3} holds
for $f$.

Next, we prove that $\vp\in \sF_b$. Let $\vp^{(i)}=\sum_{n=1}^i(b_{n-1}-b_n)\vp_n$ for $i\ge 1$. It is clear that $\vp^{(i)}\in \sF_b$ and $\vp^{(i)}\to \vp$ as $i\to\infty$. So in order to prove $\vp\in \sF_b$, it suffices to verify that
\begin{equation}\label{proofclose}\lim_{i,j\to\infty}\sE(\vp^{(i)}-\vp^{(j)}, \vp^{(i)}-\vp^{(j)})=0.\end{equation}
Indeed, for any $i>j$, we can follow the arguments above and
obtain that
\begin{align*}
&\int_{B_{R+2r}}
\,d\Gamma(\vp^{(i)}-\vp^{(j)},\vp^{(i)}-\vp^{(j)})\\
&\le \theta^{-2}\int_{B_{R+2r}} f_\theta^2
\,d\Gamma(\vp^{(i)}-\vp^{(j)},\vp^{(i)}-\vp^{(j)})\\
&\le \theta^{-2} e^{-j\lambda}\left(c_7(\lambda)\int_{U\times
{U^*}'}(f_\theta(x)-f_\theta(y))^2\,J(dx,dy)+\frac{c_8(\lambda)}{\phi(r)}\int_{B_{R+2r}}f_\theta^2(x)\,\mu(dx)\right).
\end{align*}
 On the other hand, by Lemma \ref{intelem} and the fact that ${\rm supp}\, (\vp^{(i)}-\vp^{(j)}) \subset B_{R+r}$,
\begin{align*}\int_{B_{R+2r}^c} \,d\Gamma(\vp^{(i)}-\vp^{(j)},\vp^{(i)}-\vp^{(j)})\le &\left(\sum_{n=j+1}^i (b_{n-1}-b_n)\right)^2
\int_{B_{R+2r}^c} \int_{B_{R+r}}  J(x,y)\,\mu(dy)\,\mu(dx)\\
\le&  e^{-j\lambda}\frac{c_9(\lambda)}{\phi(r)}\mu(B_{R+r}).
\end{align*}
Combining with both inequalities above,
we obtain \eqref{proofclose}.    \qed
\end{proof}

As a direct consequence of Proposition \ref{CSJ-equi}(1) and Proposition \ref{L:cswk}(1), we have the following corollary.
\begin{corollary} \label{C:cswk-1}
Suppose that $\VD$, \eqref{polycon}, $\J_{\phi,\le}$ and $\CSJ(\phi)$ hold.
 Then  there exists a constant $c_1>0$
such that for every $0<r\le R$, almost all $x_0 \in M$ and any $f\in \sF$, there exists
a cut-off function $\vp\in \sF_b$ for $B(x_0,R) \subset B(x_0,R+r)$ so that the following holds for all $\rho\in (0,\infty]$:
\be \label{e:csa2} \begin{split}
 \int_{B(x_0,R+2r)} f^2 \, d\Gamma^{(\rho)}(\vp,\vp)&
\le \frac 18 \int_{U\times {U^*}'} \vp^2(x)(f(x)-f(y))^2\,J^{(\rho)}(dx,dy)\\
&\quad + \frac{c_1}{\phi(r\wedge \rho)}  \int_{B(x_0,R+2r)} f^2 \, d\mu,
\end{split}\ee where $U=B(x_0,R+r)\setminus B(x_0,R)$ and ${U^*}'=B(x_0,R+2r)\setminus B(x_0,R-r)$.\end{corollary}

\begin{remark}\label{R:csj} \rm According to all the arguments above, we can easily obtain that Propositions \ref{CSJ-equi}, \ref{L:cswk} and Corollary  \ref{C:cswk-1} with small modifications (i.e. the cut-off function $\varphi \in \sF_b$ can be chosen to be independent of $f\in \sF$) hold for $\SCSJ(\phi)$.   \end{remark}

We close this subsection by the following statement.

\begin{lemma}\label{Conserv} Assume that
$\VD$, \eqref{polycon} and
$\UHK(\phi)$ hold and that $(\sE, \sF)$ is conservative.
Then $\EP_{\phi,\le}$ holds.
\end{lemma}

\begin{proof}
We first verify that there is a constant $c_1>0$ such that for each $t,r>0$ and for almost all $x\in M$,
$$
\int_{B(x,r)^c} p(t, x,y)\,\mu(dy)\le \frac{c_1 t}{\phi(r)}.
$$
 Indeed, we only need to consider the case that $\phi(r)>t$; otherwise, the inequality above holds trivially with $c_1=1$.
 According to $\UHK(\phi)$, $\VD$ and \eqref{polycon}, for any $t,r>0$ with $\phi(r)>t$ and almost all $x\in M$,
\begin{align*}
\int_{B(x,r)^c}p(t, x,y)\,\mu(dy)&=\sum_{i=0}^\infty\int_{B(x,2^{i+1}r)\setminus B(x,2^ir)}p(t, x,y)\,\mu(dy)\\
&\le \sum_{i=0}^\infty \frac{c_2t V(x,2^{i+1}r)}{V(x,2^{i}r)\phi(2^ir)}\le \frac{c_3t}{\phi(r)}\sum_{i=0}^\infty 2^{-i\beta_1}\le \frac{c_4t}{\phi(r)}.
\end{align*}

Now, since $(\sE, \sF)$ is conservative, by the strong Markov property, for any each $t,r>0$ and for almost all $x\in M$,
\begin{align*}\bP^x(\tau_{B(x,r)}\le t)&=\bP^x(\tau_{B(x,r)}\le t, X_{2t}\in B(x,r/2)^c)+\bP^x(\tau_{B(x,r)}\le t, X_{2t}\in B(x,r/2))\\
&\le\bP^x( X_{2t}\in B(x,r/2)^c)+ \sup_{z\notin B(x,r)^c, s\le t}\bP^z( X_{2t-s}\in B(z,r/2)^c)\\
&\le \frac{c_5t}{\phi(r)},\end{align*} which yields
$\EP_{\phi,\le}$. (Note that the conservativeness of $(\sE, \sF)$ is used in the equality above. Indeed, without the conservativeness,
there must be an extra term $\bP^x(\tau_{B(x,r)}\le t, \zeta\le 2t)$
in the right hand side of the above equality, where $\zeta$ is the
lifetime of $X$.)\qed\end{proof}

\section{Implications of heat kernel estimates} \label{section3}

In this section, we will prove (1) $\Longrightarrow$ (3) in Theorems \ref{T:main} and \ref{T:main-1}.
We point out that, under $\VD$, $\RVD$ and \eqref{polycon},
$\UHK(\phi)\Longrightarrow\FK(\phi)$ is given in Proposition \ref{pi-e-pre}  in the Appendix.

\subsection{$\UHK(\phi) +
(\sE, \sF) \hbox{ is conservative}
\Longrightarrow \J_{\phi,\le}$,  \ and $\HK(\phi) \Longrightarrow \J_\phi$}

We first show the following, where, for future reference,
it is formulated  for a general
Hunt process $Y$ that admits no killings inside.

\begin{proposition} \label{P:3.1}
Suppose that $Y=\{ Y_t, t\ge0, \bP^x, x\in E\} $ is an arbitrary
Hunt process on a locally compact separable metric space $E$
that admits no killings inside $E$.
Denote its lifetime by $\zeta$.

\begin{description}
\item{\rm (1)}
If there is a constant $c_0>0$ so that
\begin{equation}\label{e:3.1}
\bP^x (\zeta=\infty ) \geq c_0 \quad \hbox{for every } x\in E,
\end{equation}
then $\bP^x (\zeta=\infty)=1$  for every $x\in E$.

\item{\rm (2)}
Suppose that $\VD$ holds, the heat kernel $p(t,x,y)$ of the process
$Y$ exists, and there exist constants $\eps\in (0,1)$ and $c_1>0$
such that for any $x\in E$ and  $t>0$, \be\label{NDLdf1-ww}
 p(t, x ,y)\ge \frac{c_1}{V(x, \phi^{-1}(t))} \quad \hbox{for }
 y\in B(x,\eps\phi^{-1}(t)),
 \ee
where $\phi: \bR_+\to \bR_+$ is a strictly increasing continuous
function  with $\phi (0)=0$. Then $\bP^x (\zeta  = \infty)=1$  for
every $x\in E$. In particular,   $\LHK(\phi)$ implies  $\zeta
=\infty$ a.s.
\end{description}
\end{proposition}

\begin{proof}  (1) Let $\{\FF^Y_t; t\geq 0\}$ be the minimal augmented filtration generated by
the Hunt process $Y$, and set $u(x):=\bP^x (\zeta =\infty)$. Then we have $u(x) \geq c_0>0$ for $x\in E$.
 Note that $$ u(Y_t)={\bf
1}_{\{\zeta >t\} } u(Y_t)=
\bE^x \left[ {\bf 1}_{\{\zeta = \infty\} }| \FF^Y_t \right]
$$
 is a bounded martingale with $\lim_{t\to \infty} u(Y_t)
= {\bf 1}_{\{\zeta = \infty\} }$. Let $\{K_j; j\geq 1\}$ be an
increasing sequence of   compact sets so that $\cup_{j=1}^\infty
K_j=E$ and define $\tau_j = \inf\{t\geq 0:  Y_t \notin K_j\}$.
Since the Hunt process $Y$ admits no killings inside $E$, we have
$\tau_j<\zeta$
a.s. for every $j\geq 1$.
Clearly $\lim_{j\to \infty}
\tau_j=\zeta$. By the optional stopping theorem,  we have for $x\in
E$,
\begin{align*}
u(x) & =  \lim_{j\to \infty} \bE^x u(Y_{\tau_j}) = \bE^x \left[
\lim_{j\to \infty} u(Y_{\tau_j} )\right]  \\
& =   \bE^x  \left[ \lim_{j\to
\infty} u(Y_{\tau_j} ) {\bf 1}_{\{\zeta < \infty\}}+\lim_{t\to
\infty} u(Y_{t})
{\bf 1}_{\{\zeta = \infty\}}\right] \\
 &\geq    c_0 \bP^x (\zeta <\infty) + \bP^x (\zeta =\infty ) \\
& = c_0 \bP^x (\zeta <\infty) +u(x).
\end{align*}
It follows that $\bP^x (\zeta <\infty)=0$ for every $x\in E$.

(2) By \eqref{NDLdf1-ww} and
the equivalent characterization \eqref{e:vd2} of VD,
we have for every $x\in E$ and $t>0$,
$$
\bP^x (\zeta >t ) \geq  \int_{B(x,
\eps\phi^{-1}(t))}p(t, x,y)\,\mu(dy) \geq \int_{B(x,
\eps\phi^{-1}(t))} \frac{c_1}{V(x, \phi^{-1}(t)) }\,\mu(dy) \ge c_2>0.
$$
Passing $t\to \infty$, we get $\bP^x (\zeta =\infty )\geq c_2$
for every $x\in E$. The conclusion now follows immediately from (1).
\qed
\end{proof}

\begin{remark}\label{R:3.2}  \rm
\begin{itemize}
\item[(i)] The condition that $Y$ admits no killings inside $E$
is needed for Proposition \ref{P:3.1} to hold.  That is, condition \eqref{e:3.1} alone does not  guarantee
$Y$ is conservative. Here is a counterexample. Let $Y$ be
the process obtained from a Brownian motion $W=\{W_t\}$ in $\bR^3$ killed according to the potential
$q(x):={\bf 1}_{B(0, 1)} (x)$. That is, for $f\geq 0$ on $\bR^3$,
\begin{equation}\label{e:3.3}
 \bE^x [ f(Y_t)] = \bE^x \left[ f(W_t) \exp \left(-\int_0^t
 {\bf 1}_{B(0, 1)}
  (W_s)\,ds \right) \right].
\end{equation}
Denote by $\zeta$  the lifetime of $Y$. We claim that \eqref{e:3.1} holds for $Y$.
Indeed, for three-dimensional Brownian motion $W$, we have
$$
\inf_{x\in \bR^3: |x|\geq 2} \bP^x \left(\sigma^W_{B(0, 1)} =\infty \right) =1-\sup_{x\in \bR^3: |x|\geq 2} \bP^x \left( \sigma^W_{B(0, 1)} <\infty \right)=1-\sup_{x\in \bR^3: |x|\geq 2}\frac1{|x|}=\frac{1}{2},
$$
where $\sigma^W_{B(0, 1)}=\inf\{t\geq 0: W_t \in B(0, 1)\}$.
Clearly for $x\in B(0, 2)^c$,
\begin{equation}\label{e:3.4}
 \bP^x (\zeta =\infty) \geq \bP^x \left(\sigma^W_{B(0, 1)} =\infty \right) \geq \frac{1}{2}.
 \end{equation}
On the other hand, if we use $p(t, x, y)$ and $p^0(t, x, y)$ to denote the transition density function of $Y$ and $W$
with respect to the Lebesgue measure
on $\bR^3$  respectively,
then we have by \eqref{e:3.3} that
$$
e^{-t} p^0(t, x, y) \leq p(t, x, y) \leq p^0(t, x, y) \quad \hbox{for } t>0 \hbox{ and } x, y \in \bR^3.
$$
Hence there is a constant $c_1\in (0, 1)$ so that
$$
\bP^x  \left( Y_1 \in \bR^3 \setminus B(0, 2) \right)  \geq c_1 \quad \hbox{for every } x\in B(0, 1).
$$
Using the Markov property of $Y$ at time 1, we have from \eqref{e:3.4} that
$\bP^x (\zeta =\infty ) \geq
c_1/2$
for every $x\in B(0, 1)$. This establishes \eqref{e:3.1} with
$c_0=c_1/2$.
However $\bP^x (\zeta <\infty) >0$ for every $x\in \bR^3$.

\medskip

\item[(ii)] In the setting of this paper, $X$ is
the symmetric Hunt process associated with the regular Dirichlet form $(\sE, \sF)$
given by \eqref{e:1.1} that has no killing term. So $X$ always admits no killings inside $M$.
\end{itemize}
\end{remark}

The next proposition in particular
shows that
$\UHK (\phi)$ implies \eqref{e:1.2}.

\begin{proposition}\label{l:jk}
Under $\VD$ and \eqref{polycon},
$$
\UHK(\phi)
\hbox{ and } (\sE, \sF) \hbox{ is conservative}
\Longrightarrow \J_{\phi,\le} ,
$$
and  $$\HK(\phi)  \Longrightarrow \J_\phi.$$\end{proposition}

\begin{proof}  The proof is easy and standard, and we only consider $\HK(\phi)\Longrightarrow \J_\phi$ for simplicity.
Consider the form $\sE^{(t)}(f,g):=\langle f-P_{t}f,g\rangle /t$. Since $(\sE, \sF)$ is
conservative by Proposition \ref{P:3.1}(2), we can write
\begin{equation*}
\sE^{(t)}(f,g)=\frac{1}{2t}\int_{M}\int_{M}(f(x)-f(y))(g(x)-g(y))p(t,x,y)\,\mu
(dx)\,\mu (dy).
\end{equation*}
It is well known that $\lim_{t\rightarrow 0}\sE^{(t)}(f,g)=
\sE(f,g)$ for all $f,g\in \sF$. Let $A$, $B$ be disjoint compact sets, and
take $f,g\in \sF$ such that $\mathrm{supp}\,f\subset A$ and $\mathrm{supp}\, g\subset B$. Then
\begin{equation*}
\sE^{(t)}(f,g)=-\frac{1}{t}\int_{A}\int_{B}f(x)g(y)p(t,x,y)\,\mu (dy)\,\mu (dx)
\overset{t\rightarrow 0}{\longrightarrow }-\int_{A}\int_{B}f(x)g(y)\,J(dx,dy).
\end{equation*}
Using $\HK(\phi)$, we obtain
\begin{equation*}
\int_{A}\int_{B}f(x)g(y)\,J(dx,dy)\asymp\int_{A}\int_{B}\frac{f(x)g(y)}{
V(x,d(x,y))\phi(d(x,y))}\,\mu (dy)\,\mu (dx),
\end{equation*}
for all $f,g\in \sF$ such that $\mathrm{supp}\,f\subset A$ and $\mathrm{supp}\,g\subset B$.
Since $A$, $B$ are arbitrary disjoint compact sets,  it follows that
$J(dx,dy)$ is absolutely continuous w.r.t. $\mu (dx)\,\mu (dy)$, and
$\J_\phi$ holds.  \qed
\end{proof}

\subsection{$\UHK(\phi)
\hbox{ and } (\sE, \sF) \hbox{  is conservative}
\Longrightarrow\SCSJ(\phi)$}

In this subsection, we give the proof that $\UHK(\phi)$
and the conservativeness of $(\sE, \sF)$ imply $\SCSJ(\phi)$.
For $D \subset M$ and $\lam>0$, define
$$ G^{D}_\lam f(x) = \bE^x \int_0^{\tau_D} e^{-\lam t} f(X_t)\, dt,\quad x\in M_0. $$

\begin{lemma}\label{L:taub1}
Suppose that $\VD$, \eqref{polycon} and $\UHK(\phi)$
hold, and that $(\sE, \sF)$ is conservative.
Let $x_0 \in M$, $0<r\le R$, and define
\begin{align*}
D_0 &= B(x_0,R+ 9r/10 )\setminus \ol B(x_0, R + r/10),\\
D_1 &= B(x_0,R+ 4r/5 )\setminus\ol B(x_0, R + r/5),\\
D_2 &= B(x_0,R+ 3r/5 )\setminus\ol B(x_0, R + 2r/5).
\end{align*}
Let $\lam = \phi(r)^{-1}$, and set
$h = G^{D_0}_\lam {\bf1}_{D_1}$.
Then $h \in \sF_{D_0}$ and $h(x)\le \phi(r)$ for all $x\in M_0$. Moreover, there exists a constant $c_1>0$,
independent of $x_0$, $r$ and $R$, so that $ h(x) \ge c_1\phi(r)$  for all $x \in D_2\cap M_0$.
\end{lemma}

\begin{proof}
That $h \in \sF_{D_0}$ follows by \cite[Theorem 4.4.1]{FOT}.
The definition of $h$ implies that $h(x)=0$ for $x \not\in \ol D_0$,
and the upper bound on $h$ is elementary, since $h \le G^{M}_\lam {\bf1} = \lam^{-1}=\phi(r)$.

By Lemma \ref{Conserv}, we can choose a constant $\delta_{1/2}>0$ such that for all $r>0$ and all $x\in M_0$,
$$\bP^x(\tau_{B(x,r)}\le \delta_{1/2} \phi(r))\le \frac{1}{2}.$$
For any $x \in D_2\cap M_0$, $B_1: =B(x,r/5) \subset D_1$.
Hence
\begin{align*}
  h(x) &= \bE^x \int_0^{\tau_{D_0}} e^{-\lam t}{\bf1}_{D_1}(X_t)\, dt\\
  & \geq \bE^x \left[ \int_0^{\tau_{B_1}} e^{-\lam t}{\bf1}_{B_1}(X_t)\, dt;  \
 \tau_{B_1} > \delta_{1/2} \phi(r/5) \right] \\
 &\ge \bP^x ( \tau_{B_1} > \delta_{1/2} \phi(r/5)) \left[ \int_0^{\delta_{1/2} \phi(r/5)}e^{-\lam t} \, dt\right]
   \geq c_1 \phi (r),
\end{align*}
where we used \eqref{polycon} in the last inequality.
  \qed\end{proof}

We also need the following property for non-local Dirichlet forms.

\begin{lemma}\label{gamma}
For each $f,g\in \sF_b$, $\eta>0$ and any subset $D \subset M$,
\be\label{eq:enowec}\begin{split}
&(1-\eta^{-1})\int_{D\times D} f^2(x)(g(x)-g(y))^2\,J(dx,dy)\\
&\le \int_{D\times D}(g(x)f^2(x)-g(y)f^2(y))(g(x)-g(y))\,J(dx,dy)\\
 &\quad +\eta\int_{D\times D} g^2(x)(f(x)-f(y))^2\,J(dx,dy)
\end{split}\ee
\end{lemma}
\begin{proof}
For any $f,g\in \sF_b$, we can easily get that
\be \label{eq:niobefw}\begin{split}
&\int_{D\times D}f^2(x)(g(x)-g(y))^2\,J(dx,dy)\\
&=\int_{D\times D}(g(x)f^2(x)-g(y)f^2(y))(g(x)-g(y))\,J(dx,dy)\\
&\quad  -\frac 12\int _{D\times D}(f^2(x)-f^2(y))(g^2(x)-g^2(y))\,J(dx,dy).
\end{split}\ee
 Then  according to the Cauchy-Schwarz inequality, for any $\eta>0$,
\begin{align*}
&\Big|\int _{D\times D}(f^2(x)-f^2(y))(g^2(x)-g^2(y))\,J(dx,dy)\Big|\\
&\le \left( \int_{D\times D}\eta(g(x)+g(y))^2(f(x)-f(y))^2\,\,J(dx,dy) \right)^{1/2} \\
&\quad\times \left( \int_{D\times D}\eta^{-1}(f(x)+f(y))^2(g(x)-g(y))^2\,J(dx,dy) \right)^{1/2} \\
&\le \left( \int_{D\times D}4\eta g^2(x)(f(x)-f(y))^2\,J(dx,dy)\right)^{1/2} \\
&\quad\times \left( \int_{D\times D}4\eta^{-1} f^2(x)(g(x)-g(y))^2\,J(dx,dy)\right)^{1/2} \\
&\le 2\eta\int_{D\times D} g^2(x)(f(x)-f(y))^2\,J(dx,dy)\\
&\quad +2\eta^{-1}\int_{D\times D} f^2(x)(g(x)-g(y))^2\,J(dx,dy),
\end{align*} where we have used the fact $ab\le \frac{1}{2}(a^2+b^2)$ for all $a,b\ge 0$ in the last inequality.
Plugging this into \eqref{eq:niobefw}, we obtain \eqref{eq:enowec}.
\qed\end{proof}

\begin{proposition} \label{T:gdCS}
Suppose that $\VD$, \eqref{polycon} and $\UHK(\phi)$
hold, and $(\sE, \sF)$ is conservative.
Then $\SCSJ(\phi)$ holds.
\end{proposition}

\begin{proof}  By the dominated convergence theorem, we only need to verify that $\SCSJ(\phi)$ holds for any $f\in \sF_b$.
For any $x_0\in M$ and $s>0$, let $B_s=B(x_0,s)$. For $0<r\le R$, let
$U=B_{R+r}\setminus B_R$ and
$U^*=B_{R+3r/2}\setminus B_{R-r/2}$. Let $D_i$  be those as in Lemma \ref{L:taub1}, and $\lambda=\phi(\lam)^{-1}$.
For $x\in M_0$, set
\begin{align*}
 g(x) &= \frac{ G^{D_0}_\lam {\bf1}_{D_1} (x)}{c^*\phi(r)},  \\
 \vp(x) &=
\begin{cases}
 1 \wedge g(x) & \text{ if } x \in B_{R+r/2}^c\cap M_0, \\
 \quad 1  & \text{ if } x \in B_{R+r/2}\cap M_0, \\
\end{cases}
\end{align*}where $c^*$ is the constant $c_1$ in Lemma \ref{L:taub1}.
 Then  by Lemma \ref{L:taub1}, $\vp =0 $ on $B_{R+r}^c\cap M_0$, and $\vp = 1$ on $B_R\cap M_0$.

We first claim
\be\label{eq:jooe}
\int_{U^*} f^2 \,d\Gam(\vp, \vp)\le \int_{U^*} f^2 \,d\Gam(g, g)+\frac {c_1}{\phi(r)}\int_{U^*}f^2\,d\mu, \quad f \in \sF_b.
\ee
Indeed, by decomposing the regions of integrals, we have
\begin{align*}
\int_{U^*} f^2\, d\Gam(\vp, \vp) &=\int_{B_{R+r/2}\setminus B_{R-r/2}}\int_{B_{R+r}\setminus B_{R+r/2}} +
\int_{B_{R+3r/2}\setminus B_{R+r/2}}\int_{B_{R+r/2}}\\
&\quad+\int_{B_{R+3r/2}\setminus B_{R+r/2}}\int_{B_{R+r}\setminus B_{R+r/2}}
+\int_{B_{R+3r/2}\setminus B_{R-r/2}}\int_{B_{R+r}^c}\\
&=: I_1+I_2+I_3+I_4,
\end{align*}
where the first integral of each term in the right hand side is with respect to $x$. Here we used the fact
\[
\int_{B_{R+r/2}\setminus B_{R-r/2}}f^2(x)\,\mu(dx)\int_{B_{R+3r/5}}(\vp(x)-\vp(y))^2J(x,y)\,\mu(dy)=0,
\]
because $\vp(x)=\vp(y)=1$ when $x,y\in B_{R+3r/5}\cap M_0$.
By Lemma \ref{intelem} and \eqref{polycon}, we have
\begin{align*}
I_1&=\int_{B_{R+r/2}\setminus B_{R-r/2}}f^2(x)\,\mu(dx)\int_{B_{R+r}\setminus B_{R+3r/5}}(1-\varphi(y))^2 J(x,y)\,\mu(dy)\\
&\le \frac {c_1}{\phi(r)}\int_{B_{R+r/2}\setminus B_{R-r/2}}f^2\,d\mu.
\end{align*}
Similarly,
\begin{align*}
I_2&=\int_{B_{R+3r/2}\setminus B_{ R+3r/5}}f^2(x)(\varphi(x)-1)^2\,\mu(dx)\int_{B_{R+r/2}} J(x,y)\,\mu(dy)\\
&\le \frac {c_2}{\phi(r)}\int_{B_{R+3r/2}\setminus B_{R+3r/5}}f^2\,d\mu,\\
I_4&=\int_{B_{R+9r/10}\setminus B_{R-r/2}}f^2(x)\vp^2(x)\,\mu(dx)\int_{ B_{R+r}^c} J(x,y)\,\mu(dy)\\
&\le \frac {c_3}{\phi(r)}\int_{B_{R+9r/10}\setminus B_{R-r/2}}f^2\,d\mu.
\end{align*}
Finally, we have
\begin{align*}
I_3&=\int_{B_{R+3r/2}\setminus B_{R+r/2}}f^2(x)\,\mu(dx)\int_{B_{R+r}\setminus B_{R+r/2}}(\varphi (x)-\varphi(y))^2 J(x,y)\,\mu(dy)\\
&\le \int_{B_{R+3r/2}\setminus B_{R+r/2}}f^2(x)\,\mu(dx)\int_{B_{R+r}\setminus B_{R+r/2}}(g(x)-g(y))^2 J(x,y)\,\mu(dy)\\
&\le \int_{U^*} f^2\, d\Gam(g,g),
\end{align*}
so that \eqref{eq:jooe} is proved.

Next, using Lemma \ref{intelem} and \eqref{eq:enowec} with $\eta=2$, we have for any $ f\in \sF_b$,
\be \begin{split}
\int_{U^*} f^2\, d\Gam(g, g)&
\le \int_{U^*\times U^*} f^2(x)(g(x)-g(y))^2\,J(dx,dy)\\
&\quad + \int_{U^*\times {U^*}^c} f^2(x)g^2(x)\,J(dx,dy)\\
&\le 2\int_{U^*\times U^*} (f^2(x)g(x)-f^2(y)g(y))(g(x)-g(y)) \,J(dx,dy)\\
&\quad + 4\int_{U^*\times U^*} g^2(x)(f(x)-f(y))^2\,J(dx,dy)
 + \frac{c_4}{\phi(r)}\int_U f^2\,d\mu,
\label{e:int1.1} \end{split}
\ee
where in the last inequality we have used the fact that $g$ is zero outside $U$.

 With $\lam = \phi(r)^{-1}$, we have for any $f\in \sF_b$,
\begin{equation}\label{e:cr1.1}\begin{split}
&\int_{U^*\times U^*} (f^2(x)g(x)-f^2(y)g(y))(g(x)-g(y)) \,J(dx,dy)\\
&\le \int_{(U^*\times U^* )\cup ({U^*}^c\times U^* )\cup (U^*\times {U^*}^c)} (f^2(x)g(x)-f^2(y)g(y))(g(x)-g(y)) \,J(dx,dy) \\
 &=\int_M d\Gam(f^2 g,g)= \sE(f^2 g,g)
 \le \sE_\lam(f^2 g,g) \\
&= (c^* \phi(r))^{-1} \sE_\lam(f^2 g,G^{D_0}_\lam {\bf1}_{D_1}) \\
&=  (c^* \phi(r))^{-1} \langle f^2 g,{\bf1}_{D_1} \rangle\\
&\le (c^* \phi(r))^{-1}  \int_U f^2 g\, d\mu .
\end{split}\end{equation}
Here we used \cite[Theorem 4.4.1]{FOT}
and the fact that $f^2 g \in \sF_{D_0}$ to obtain the third equality.
Plugging \eqref{e:cr1.1} into \eqref{e:int1.1}, and using the facts that
$g \le c_5$ and $g$ is zero outside $U$, we obtain
\begin{align*}
\int_{U^*} f^2 \, d\Gamma(g,g)
&\le 4\int_{U^*\times U^*} g^2(x)(f(x)-f(y))^2\,J(dx,dy)
  \\
 &\quad+\frac{2}{ c^* \phi(r)}
 \int_U f^2 g\, d\mu +  \frac{c_4}{\phi(r)}\int_U f^2\,d\mu \\
&\le 4c_5^2 \int_{U\times U^*} (f(x)-f(y))^2\,J(dx,dy)
+ \Big(\frac{2 c_5}{c^*}+c_4\Big) \frac{1}{\phi(r)}
\int_U f^2\, d\mu.
\end{align*}
This and \eqref{eq:jooe} imply $\CSAJ(\phi)$ for any $f\in \sF_b$ with the strong form (i.e. the cut-off function is independent of $f\in \sF_b$) with $C_0=\frac{1}{2}$. Therefore, the desired assertion follows from Proposition \ref{CSJ-equi}(2) and Remark \ref{R:csj}. \qed\end{proof}

As mentioned in the beginning of this section, $\UHK (\phi)$ implies $\FK (\phi)$   by Proposition
\ref{pi-e-pre} under $\VD$, $\RVD$ and \eqref{polycon}. This completes the proof of (1) $\Longrightarrow$ (3)
part in   Theorems \ref{T:main} and \ref{T:main-1}. Note also that (3) $\Longrightarrow$ (4)
part in   Theorems \ref{T:main} and \ref{T:main-1} holds trivially.

\section{Implications of $\CSJ (\phi)$  and  $\J_{\phi, \geq}$}\label{section4}

In this section, we will prove (4) $\Longrightarrow$ (2) in Theorems \ref{T:main} and \ref{T:main-1}.

\subsection{$\J_{\phi,\ge}\Longrightarrow \FK(\phi)$}\label{jFKnow}
We first prove that under $\VD$ and \eqref{polycon},
$\J_{\phi,\ge}$ implies the local Nash inequality introduced by Kigami (\cite{Ki}).
Note that 
when volume of balls is uniformly comparable, 
the following lemma was proved in \cite[Theorem 3.1]{CK2}.
The proof below is similar to that of \cite[Theorem 3.1]{CK2}.

\begin{lemma}\label{T:3.2}
Under $\VD$, \eqref{polycon} and $\J_{\phi,\ge}$, there is a constant $c_0>0$
such that
for any $s>0$,
\[
\|u\|_2^2\le c_0\Big(\frac {\|u\|_1^2}{\inf_{z\in {\rm supp}\,u}V(z,s)}+\phi(s)\sE(u,u)\Big),
\quad \forall u\in \sF\cap L^1(M;\mu).
\]

\end{lemma}
\begin{proof} For any $u\in \sF\cap L^1(M;\mu)$ and $s>0$, define
$$ u_s(x) := \frac1{V(x, s)} \int_{B(x, s)} u(z)\, \mu (dz) \quad \hbox{for }
   x\in M.
$$
For $A\subset M$ and $s>0$, denote $A^s:=\{z\in M: d(z,A)<s\}$. Using \eqref{eq:vdeno}, we have
$$ \| u_s \|_\infty \leq \frac{c_1 \| u\|_1}{\inf_{z_0\in ({\rm supp}\,u)^s}V(z_0,s)}
\leq \frac{c_1' \| u\|_1}{\inf_{z\in {\rm supp}\,u}V(z,2s)}
\leq \frac{c_1' \| u\|_1}{\inf_{z\in {\rm supp}\,u}V(z,s)}
$$
and  \begin{align*}\| u_s \|_1 \le &\int_{({\rm supp}\,u)^s} \frac{1}{V(x,s)}\,\mu(dx) \int_{B(x,s)} |u(z)|\,\mu(dz)\\
=&\int_{{\rm supp}\,u} |u(z)|\,\mu(dz)\int_{ ({\rm supp}\,u)^s \cap
B(z,s)}\frac{1}{V(x,s)}\,\mu(dx)
\\
\le& \int_{{\rm supp}\,u} |u(z)|\,\frac{V(z,s)}{\inf_{x\in B(z,s)} V(x,s)}\,\mu(dz)
\leq c_2 \| u \|_1,
\end{align*}
where in the last inequality we used the fact (due to \eqref{eq:vdeno} again) that for any $x\in B(z,s)$,
$$\frac{V(z,s)}{V(x,s)}\le \widetilde C_\mu\left(\frac{d(x,z)}{s}+1\right)^{d_2}\le 2^{d_2}\widetilde C_\mu.$$
In particular,
\begin{equation*}\label{eqn:3.3}
\| u_s\|_2^2 \leq \| u_s\|_\infty \| u_s\|_1 \leq \frac{c_3 \| u\|_1^2}{\inf_{z\in {\rm supp}\,u}V(z,s)}.
\end{equation*}
Therefore, for $u\in \sF\cap L^1(M;\mu)$, by $\J_{\phi,\ge}$,
\begin{equation*}\begin{split}
\| u \|_2^2 &\leq2 \|u-u_s\|_2^2 +2 \| u_s \|_2^2 \\
&\leq 2 \int_M \left( \frac1{V(x, s)} \int_{B(x, s)}
  (u(x)-u(y))^2 \mu (dy) \right)   \mu (dx) +
  \frac{2c_3 \| u\|_1^2}{\inf_{z\in {\rm supp}\,u}V(z,s)}\\
&\leq  c_4  \int_M \bigg( \frac1{V(x,s)} \int_{B(x, s)}
  (u(x)-u(y))^2 J (x, y) \, \phi (s) V(x, s)\, \mu (dy) \bigg) \,  \mu (dx)\\
  &\quad
  + \frac{2c_3 \| u\|_1^2}{\inf_{z\in {\rm supp}\,u}V(z,s)}\\
&\leq c_5  \, \phi (s)   \int_M  \int_{B(x, s)}
  (u(x)-u(y))^2 J (x, y)\, \mu (dy) \, \mu (dx)\\
  &\quad+\frac{2c_3 \| u\|_1^2}{\inf_{z\in {\rm supp}\,u}V(z,s)}  \\
&\leq  c_6 \left( \phi (s) \sE (u, u) +\frac{\| u\|_1^2}{\inf_{z\in {\rm supp}\,u}V(z,s)}  \right).
\end{split}
\end{equation*}
We thus obtain the desired inequality. \qed \end{proof}

We then conclude by  Proposition \ref{P:3.3-2} that $\J_{\phi,\ge}\Longrightarrow \FK(\phi)$
under $\VD$, $\RVD$ and \eqref{polycon}.

\medskip

By Proposition \ref{P:reg} in Appendix (see also \cite[Theorem 3.1]{BBCK}
and \cite[Section 2.2]{GT}),
it follows that there is a properly
exceptional set ${\cal N}$ so that the Hunt process
$\{X_t\}$ has a transition density function $p(t,x, y)$ for every $x,
y\in M\setminus {\cal N}$.

 \subsection{Caccioppoli and $L^1$-mean value inequalities} \label{s:cac}

 In this subsection, we establish mean value inequalities for subharmonic functions.
 Though in this paper,  we only need mean value inequalities for the $\rho$-truncated Dirichlet form $(\sE^{(\rho)},\sF)$,
 we choose to  first
 establish these inequalities for subharmonic functions of the original Dirichlet form $(\sE,\sF)$
 and then indicate how these proofs can be modified to establish similar inequalities for subharmonic functions
 of the $\rho$-truncated Dirichlet form $(\sE^{(\rho)},\sF)$.
 There are several reasons for doing so:
 (i) the mean value inequalities for the original Dirichlet form $(\sE,\sF)$ will be used as one of the key tools in the study of
 the stability of parabolic Harnack inequality in our subsequent paper \cite{CKW};
 (ii) since the proofs share many common parts and ideas in the truncated and non-truncated settings,
 it is more efficient to do
 it in this way;  (iii) although they share
 many common ideas in these two settings, there are also some differences; see the paragraph
   proceeding the statement of Proposition \ref{P:mvi2}, by putting together in one place
  clearly reveals differences and difficulties in the setting of jump processes as for the diffusion case.

\medskip

We first need to
introduce the analytic characterization of subharmonic functions
and to extend the definition of bilinear form $\sE$. Let $D$ be an open subset of $M$.
Recall that a function $f$ is said to be locally in $\sF_{D}$, denoted as $f\in \sF_{D}^{loc}$, if for every relatively compact subset $U$ of ${D}$, there is a function $g\in \sF_{D}$ such that $f=g$ $\mu$-a.e. on $U$.

The next lemma is proved in  \cite[Lemma 2.6]{Chen}.

\begin{lemma}  \label{L:4.2}
Let ${D}$ be an open subset of $M$. Suppose $u$ is a function in $\sF_{D}^{loc}$ that is locally bounded on ${D}$
and satisfies that
\begin{equation}\label{con-1}
\int_{U\times V^c} |u(y)|\,J(dx,dy)<\infty
\end{equation} for any relatively compact open sets $U$ and $V$ of $M$ with $\bar{U}\subset V \subset \bar{V} \subset {D}$.
 Then  for every $v\in C_c({D})\cap \sF$, the expression
$$
\int (u(x)-u(y))(v(x)-v(y))\,J(dx,dy)
$$ is well defined and finite; it will still be denoted as $\sE(u,v)$.
\end{lemma}

As noted in \cite[(2.3)]{Chen}, since $(\sE,\sF)$ is a regular Dirichlet form on $L^2(M; \mu)$, for any relatively compact open sets $U$ and $V$ with $\bar{U}\subset V$, there is a function $\psi\in \sF\cap C_c(M)$ such that $\psi=1$ on $U$ and $\psi=0$ on $V^c$. Consequently,
$$\int_{U\times V^c}\,J(dx,dy)=\int_{U\times V^c}(\psi(x)-\psi(y))^2\,J(dx,dy)\le \sE(\psi,\psi)<\infty,$$
so each bounded function $u$ satisfies \eqref{con-1}.

\begin{definition}\rm
Let ${D}$ be an open subset of $M$.
\begin{description}
\item{\rm (i)} We say that a nearly Borel measurable function $u$ on $M$
is \emph{$\sE$-subharmonic}  (resp. \emph{$\sE$-harmonic, $\sE$-superharmonic})
in ${D}$ if $u\in\sF_{D}^{loc}$, satisfies condition \eqref{con-1} and
\begin{equation*}\label{an-har}
\sE(u,\varphi)\le 0 \quad (\textrm{resp.}\ =0, \ge0)
\end{equation*}
for any $0\le\varphi\in\sF_{D}.$

\item{\rm (ii)} A nearly Borel measurable function $u$ on $M$ is said to be \emph{subharmonic}  (resp. \emph{harmonic, superharmonic})
in ${D}$ (with respect to the process $X$)
if for any relatively compact subset $U\subset D$,
$t\mapsto   u (X_{t\wedge \tau_U}) $ is a uniformly integrable submartingale
(resp. martingale, supermartingale) under $\bP^x$ for q.e. $x\in M$.
\end{description}
\end{definition}

The following result is established in \cite[Theorem 2.11  and Lemma 2.3]{Chen} first for harmonic functions,
and then extended in \cite[Theorem 2.9]{ChK} to subharmonic functions.

\begin{theorem}\label{equ-har}Let ${D}$ be an open subset of $M$, and let $u$ be a bounded function.
Then $u$ is $\sE$-harmonic $($resp.  $\sE$-subharmonic$)$ in ${D}$ if and only if $u$ is  harmonic
 $($resp. subharmonic$)$
 in ${D}$.
 \end{theorem}

To establish the Caccioppoli inequality, we also need the following definition.

\begin{definition} \rm
For a Borel measurable function $u$ on $M$, we define its
\emph{nonlocal tail}   in the ball
$B(x_0,r)$ by
\begin{equation} \label{def-T}
\T\, (u; x_0,r)=\phi(r)\int_{B(x_0,r)^c}\frac{|u(z)|}{V(x_0,d(x_0,z))\phi(d(x_0,z))}\,\mu(dz).
\end{equation}
\end{definition}

Suppose that $\VD$ and \eqref{polycon} hold. Observe  that  in view of \eqref{e:2.1},  $\T\, (u; x_0,r)$ is finite if $u$ is bounded.
 Note also
that $\T\, (u; x_0,r)$ is finite by the H\"{o}lder inequality and \eqref{e:2.1} whenever
$u\in L^p(M;\mu)$ for any $p\in[1,\infty)$ and $r>0$. As mentioned in \cite{CKP},
the key-point in the present nonlocal setting is how to manage the
nonlocal tail.

We first show that $\CSJ(\phi)$ enables us to prove
a Caccioppoli
inequality for $\sE$-subharmonic functions. Note that the Caccioppoli
inequality below is different from that in \cite[Lemma 1.4]{CKP}, since our argument is heavily based on $\CSJ(\phi)$.

\begin{lemma} \label{L:ci}
{\bf (Caccioppoli inequality)}\, For $x_0 \in M$ and $s>0$, let $B_s=
B(x_0, s)$. Suppose that $\VD$, \eqref{polycon}, $\CSJ(\phi)$ and
$\J_{\phi,\le}$ hold. For $0<r <  R$, let $u$
be an
$\sE$-subharmonic function on $B_{R+r}$ for the Dirichlet
form $(\sE, \sF)$, and $v = (u-\theta)^+$ for $\theta>0$. Also,
let $\vp$ be the cut-off function for $B_{R-r} \subset B_{R}$ associated with $v$ in
$\CSJ(\phi)$. Then there exists a constant $c>0$ independent of
$x_0, R, r$ and $\theta$ such that
\begin{equation}\label{e:cacc}\begin{split}
\int_{B_{R+r}}\,d\Gamma(v\vp,v\vp)\le  \frac{c }{\phi(r)} \left[ 1+\frac{1}{\theta}\left(1+\frac{R}{r}\right)^{d_2+\beta_2-\beta_1}
\T\,(u; x_0, R+r)\right]\int_{B_{R+r}} u^2\,d\mu.
\end{split}\end{equation}
\end{lemma}

\begin{proof} Since $u$ is $\sE$-subharmonic on $B_{R+r}$ for the Dirichlet
form $(\sE, \sF)$ and $\vp^2v\in \sF_{B_{R}}$,
we have $u\in \FF_{B_{R+r}}^{loc}$ and
\begin{equation}\begin{split}\label{e:ci1}0\ge \sE(u,\vp^2v)=&\int_{B_{R+r}\times B_{R+r}} (u(x)-u(y))(\vp^2(x)v(x)-\vp^2(y)v(y))\,J(dx,dy)\\
&+2\int_{B_{R+r}\times B_{R+r}^c} (u(x)-u(y)) \vp^2(x)v(x)\,J(dx,dy)\\
=&:I_1+2I_2.\end{split}\end{equation}
For $I_1$, we may and do assume without loss of generality that $u(x)\ge u(y)$; otherwise just exchange the roles of $x$ and $y$ below. We have
\begin{align*}
&(u(x)-u(y))(\vp^2 (x)  v(x)-\vp^2 (y) v(y))\\
 &= (u(x)-u(y)) \vp^2(x)  (v(x)-v(y)) + (u(x)-u(y))  (\vp^2 (x) -\vp^2 (y)) v(y)\\
&\geq \vp^2 (x)  (v(x)-v(y)) ^2 +  (v(x)-v(y))  (\vp^2 (x) -\vp^2 (y)) v(y)\\
&\ge \vp^2(x)(v(x)-v(y))^2 -\frac{1}{8} (\vp(x)+\vp(y))^2(v(x)-v(y))^2- 2 v^2(y)(\vp(x)-\vp(y))^2 \\
&\ge \frac34 \vp^2(x)(v(x)-v(y))^2 -\frac14 \vp^2(y)(v(x)-v(y))^2- 2 v^2(y)(\vp(x)-\vp(y))^2.
\end{align*}
where the first inequality follows from the facts that for any $x,y\in M$, $u(x)-u(y)\geq v(x)-v(y)$ and
$(u(x)-u(y))v(y)= (v(x)-v(y))  v(y)$,  while in the second and third equalities  we used the
facts that $ ab\geq  - \frac{1}{8}a^2- 2b^2$ and $(a+b)^2 \leq 2a^2+2b^2$, respectively, for all $a,b\in \bR$.
This together with the symmetry of $J(dx,dy)$  yields that
$$
I_1 \ge \frac{1}{2}\int_{B_{R+r}\times B_{R+r}} \vp^2(x)(v(x)-v(y))^2 \, J(dx,dy)
 -2\int_{B_{R+r}\times B_{R+r}} v^2(x)(\vp(x)-\vp(y))^2\, J(dx,dy).
$$

For $I_2$, note that
\begin{align*}
(u(x)-u(y)) \vp^2(x)v(x)
 = &((u(x)-\theta) -(u(y)-\theta)) \vp^2(x)v(x)\\
  \geq &(v(x)-v(y)) \vp^2(x)  v (x)  \geq - v(x) v(y).
 \end{align*}
 Note also that $v\le vu/\theta\le u^2/\theta$.
Hence we have
\begin{align*}
I_2=& \int_{B_{R}\times B_{R+r}^c}(u(x)-u(y)) \vp^2(x)v(x)\, J(dx,dy)\\
\ge &-\int_{B_{R}}v\,d\mu\left[\sup_{x\in B_{R}}\int_{B_{R+r}^c} v(y)\,J(x,dy)\right]\\
\ge&-\frac{1}{\theta}
\int_{B_{R}}u^2
\,d\mu\left[\sup_{x\in B_{R}}\int_{B_{R+r}^c} v(y)\, J(x,dy)\right]\\
\ge&-\frac{c_1 }{\theta\phi(r)}
\bigg[ \left(1+\frac{R}{r}\right)^{d_2+\beta_2-\beta_1} \phi(R+r)
\int_{B_{R+r}^c}\frac{|u(y)|}{V(x_0,d(x_0,y))\phi(d(x_0,y))}\,\mu(dy)
\bigg]
  \int_{B_{R}} u^2\,d\mu\\
=&-\frac{c_1 }{\theta\phi(r)}
\bigg[ \left(1+\frac{R}{r}\right)^{d_2+\beta_2-\beta_1}\!\! \T\,(u; x_0,
R+r)\bigg]\int_{B_{R}} u^2\,d\mu ,
\end{align*}
where  the last inequality follows from the fact  that $v\le |u|$, $\J_{\phi,\le}$
as well as \eqref{eq:vdeno} and \eqref{polycon} which imply that
 for any $x\in B_{R}$ and $y\in B_{R+r}^c$,
 $$
  \frac{V(x_0,d(x_0,y))\phi(d(x_0,y))}{V(x,d(x,y))\phi(d(x,y))}\le c'\Big(1+\frac{d(x_0,x)}{d(x,y)}\Big)^{d_2+\beta_2}
 \le c''\Big(1+\frac{R}{r}\Big)^{d_2+\beta_2}
 $$
 and
 $$\frac{\phi(r)}{\phi(R+r)}\le c'''\Big(1+\frac{R}{r}\Big)^{-\beta_1}.
 $$

Putting the estimates for $I_1$ and $I_2$ above into \eqref{e:ci1}, we arrive at
\begin{equation}\begin{split}\label{e:ci2}0
&\le 4\int_{B_{R+r}\times B_{R+r}} v^2(x)(\vp(x)-\vp(y))^2\, J(dx,dy)\\
&\quad
 - \int_{B_{R}\times B_{R+r}} \vp^2(x)(v(x)-v(y))^2 \, J(dx,dy)\\
&\quad + \frac{c_2 }{\theta\phi(r)}
\bigg[ \left(1+\frac{R}{r}\right)^{d_2+\beta_2-\beta_1} \T\,(u; x_0,
R+r)\bigg] \int_{B_{R}}
u^2\,d\mu\\
&\le 4\int_{B_{R+r}} v^2\, d\Gamma(\vp,\vp)
 -\int_{B_{R}\times B_{R+r}} \vp^2(x)(v(x)-v(y))^2 \, J(dx,dy)\\
&\quad + \frac{c_2 }{\theta\phi(r)}
\bigg[ \left(1+\frac{R}{r}\right)^{d_2+\beta_2-\beta_1} \T\,(u; x_0,
R+r)\bigg] \int_{B_{R}} u^2\,d\mu.
\end{split}\end{equation}

On the other hand, using the inequality $(a+b)^{2}\le 2(a^2+b^2)$
for all $a,b\in\bR$ and Lemma \ref{intelem}, we have
\begin{equation}\label{e:cac33}\begin{split}
\int_{B_{R+r}}\,&d\Gamma(v\vp,v\vp)\\
=& \,
\int_{B_{R+r}\times M}(v(x)\vp(x)-v(y)\vp(y))^2\,J(dx,dy)\\
\le &\int_{B_{R+r}\times B_{R+r}}\big(v(x)(\vp(x)-\vp(y))+\vp(y)(v(x)-v(y))\big)^2\,J(dx,dy)\\
& \, +\int_{B_{R}} v^2(x)\vp^2(x) \int_{B_{R+r}^c}\,J(dx,dy) \\
\le & \, 2\bigg[
\int_{B_{R+r}\times B_{R+r}}v^2(x)(\vp(x)\!-\!\vp(y))^2\,J(dx,dy)\\
&\quad +\int_{B_{R+r}\times
B_{R+r}}\vp^2(x)(v(x)\!-\!v(y))^2\,J(dx,dy)\bigg]+\frac{c_3}{\phi(r)}\int_{B_{R}}v^2\,d\mu\\
\le & \, 2\int_{B_{R+r}}v^2 \,d\Gamma(\vp,\vp)
+2\int_{B_{R}\times B_{R+r}}\vp^2(x)(v(x)\!-\!v(y))^2\,J(dx,dy)\\
&\,
+\frac{c_3}{\phi(r)}\int_{B_{R}}u^2\,d\mu.
\end{split}\end{equation}
Combining \eqref{e:ci2} with \eqref{e:cac33}, we have for $a>0$,
\begin{equation}\label{eq:nefow}\begin{split}
&\, a\int_{B_{R+r}}\,d\Gamma(v\vp,v\vp)\\
&\le (2a+4)\int_{B_{R+r}}v^2\,d\Gamma(\vp,\vp)
 +(2a-1)\int_{B_{R}\times B_{R+r}} \vp^2(x)(v(x)-v(y))^2\,J(dx,dy)\\
&\quad+\frac{c_4(1+a) }{\phi(r)}
\bigg[ 1+\frac{1}{\theta}\left(1+\frac{R}{r}\right)^{d_2+\beta_2-\beta_1}
\T\,(u; x_0, R+r)\bigg] \int_{B_{R}} u^2\,d\mu.
\end{split}\end{equation}

 Next by using \eqref{e:csa2} for $v$ with $\rho=\infty$,  we have
\begin{equation}\label{rem-csj}
\int_{B_{R+r}}v^2\,d\Gamma(\vp,\vp) \le \frac 18 \int_{B_{R} \times
B_{R+r}}\vp^2(x)(v(x)-v(y))^2\,J(dx,dy) + \frac{c_0}{\phi(r)} \int_{B_{R+r}}
v^2  \,d\mu.
\end{equation}
Plugging this into \eqref{eq:nefow} with $a=2/9$ (so that $(4+2a)/8+(2a-1)=0$), we obtain
\begin{align*}
\frac 29\int_{B_{R+r}}\,d\Gamma(v\vp,v\vp)\le  \frac{c_5 }{\phi(r)}
\bigg[1+\frac{1}{\theta}\left(1+\frac{R}{r}\right)^{d_2+\beta_2-\beta_1}
\T\,(u; x_0, R+r)\bigg]\int_{B_{R+r}} u^2\,d\mu,
\end{align*} which proves the desired assertion.
\qed\end{proof}

\begin{remark} {\rm In order to obtain \eqref{e:cacc} we need that the constant in the first term on the right hand side of \eqref{e:csa2} was less than $1/4$. On the other hand, we note that
\eqref{rem-csj} is weaker than \eqref{e:csa2} yielded by
$\CSJ(\phi)$, which can strengthen the first term in the right hand
side of \eqref{rem-csj} into $$\frac 18 \int_{U \times
U^*}\vp^2(x)(v(x)-v(y))^2\,J(dx,dy)$$ with $U=B_{R}\setminus B_{R-r}$ and $U^*=B_{R+r}\setminus
B_{R-2r}$. } \end{remark}

The key step in the proof of the mean value inequality is the
following comparison over balls. For a ball $B=B(x_0,r) \subset M $
and a function $w$ on $B$, write
$$ I(w,B) = \int_B w^2 \, d\mu. $$
The following lemma can be proved similarly to
that of \cite[Lemma 3.5]{AB} (see also \cite[Lemma 3.2]{Gr1}) with very minor corrections
due to $B_{R+r}$ instead of $B_{R}$.
For completeness, we give the proof below.

\begin{lemma}\label{oppo}  For $x_0 \in M$ and $s>0$, let $B_s=
B(x_0, s)$. Suppose $\VD$, \eqref{polycon}, $\FK(\phi)$,
$\CSJ(\phi)$ and $\J_{\phi,\le}$ hold. For $R, {r_1},{r_2}>0$  with
 ${r_1}\in [\frac12 R,R]$ and ${r_1}+{r_2}\le R$, let $u$ be an $\sE$-subharmonic
function on $B_R$ for the Dirichlet form $(\sE, \sF)$, and $v =
(u-\theta)_+$ for some $\theta>0$. Set $ I_0 = I(u,B_{{r_1}+{r_2}})$ and $ I_1
= I(v,B_{{r_1}})$. We have
\be \label{e:IIi}
 I_1\le
\frac{ c_1 }{ \theta^{2\nu} V(x_0,R)^\nu} I_0^{1+\nu}
\left(1+\frac{{r_1}}{{r_2}}\right)^{\beta_2} \left[ 1+
\left(1+\frac{{r_1}}{{r_2}}\right)^{d_2+\beta_2-\beta_1}
\frac{\T\,(u; x_0,R/2)}{\theta} \right],
\ee
where $\nu$ is the constant appearing in
the $\FK(\phi)$ inequality \eqref{e:fki}, $d_2$ is the constant in
\eqref{e:vd2} from  $\VD$, and $c_1$ is a constant independent of $\theta, x_0, R, {r_1}$
and ${r_2}$.
\end{lemma}

 \begin{proof} Set
$$ D =\{ x \in B_{{r_1}+{r_2}/2}: u(x) > \theta \}. $$
Let $\vp$ be a cut-off function for $B_{{r_1}} \subset B_{{r_1}+{r_2}/2}$ associated with $v$ in $\CSJ(\phi)$.

As in \cite{Gr1} the proof uses the following five inequalities:
\begin{align}
\label{e:in15}
 \int_{B_{{r_1}+{r_2}/2}} u^2\,  d\mu
&\le I_0,\\
\label{e:in25}
\int_{B_{{r_1}+{r_2}}} d\Gam( v\vp, v \vp) &\le \frac{c_0 }{\phi({r_2})} \left[ 1+\frac{1}{\theta}\left(1+\frac{{r_1}}{{r_2}}\right)^{d_2+\beta_2-\beta_1}  \T\,(u; x_0,R/2)\right]  I_0 , \\
\label{e:in35}
 2\int_{D} d\Gam( v\vp, v \vp) &\ge \lam_1(D) \int_{D} v^2 \vp^2\, d\mu,\\
\label{e:in55}
 \lam_1(D) &\ge C \mu(B_{{r_1}+{r_2}})^\nu \phi({r_1}+{r_2})^{-1} \mu(D)^{-\nu}, \\
\label{e:in45}
 \mu(D) &\le \theta^{-2} \int_{B_{{r_1}+{r_2}/2}} u^2 \, d\mu.
\end{align}
Of these, \eqref{e:in15} holds trivially. The inequality \eqref{e:in25} follows  immediately from \eqref{e:cacc}
since, by $\VD$ and \eqref{polycon},
 $$
 \T\,(u; x_0,{r_1}+{r_2})\le c_1\T\, (u;x_0, R/2).
 $$
 Inequality \eqref{e:in35} is immediate from the variational definition \eqref{e:lam1} of $\lam_1(D)$ and the facts that $v\vp\in \sF_{D}$ and  $$2\int_{D} d\Gam( v\vp, v \vp) \ge \sE(v\vp, v\vp).$$ Indeed, since $v\vp=0$ on $D^c$, we have \begin{align*}\sE(v\vp,v\vp)=&\bigg(\int_{D\times D }+\int_{D\times D^c}+\int_{D^c\times D}+\int_{D^c\times D^c}\bigg)\big(v(x)\vp(x)-v(y)\vp(y)\big)^2\,J(dx,dy)\\
=&\bigg(\int_{D\times D }+\int_{D\times D^c}+\int_{D^c\times D}\bigg)\big(v(x)\vp(x)-v(y)\vp(y)\big)^2\,J(dx,dy)\\
\le& \bigg(\int_{D\times M }+\int_{M\times D}\bigg)\big(v(x)\vp(x)-v(y)\vp(y)\big)^2\,J(dx,dy)\\
=&2\int_{D\times M }\big(v(x)\vp(x)-v(y)\vp(y)\big)^2\,J(dx,dy)\\
=&2\int_{D} d\Gam( v\vp, v \vp), \end{align*} where the third equality follows from the symmetry of $J(dx,dy)$.
\eqref{e:in55} follows from the Faber-Krahn inequality \eqref{e:fki}, $\VD$ and \eqref{polycon}. \eqref{e:in45} is just Markov's inequality.

Putting \eqref{e:in15} into \eqref{e:in45}, we get
\be\label{eq:nofbiw}
\mu(D) \le I_0/\theta^2.
\ee
By $\VD$, \eqref{polycon}, \eqref{e:in35}, \eqref{e:in55} and \eqref{eq:nofbiw}, we have
\begin{align*}
\int_{D} d\Gam( v\vp, v \vp)&\ge \frac{C\mu
(B_{{r_1}+{r_2}})^\nu}{\phi({r_1}+{r_2})\mu (D)^\nu}
\int_{D} v^2 \vp^2\, d\mu\\
&= \frac{C\mu (B_{{r_1}+{r_2}})^\nu}{\phi({r_1}+{r_2})\mu (D)^\nu}
\int_{B_{{r_1}+{r_2}/2}} v^2 \vp^2\, d\mu\\
 &\ge
\frac{C'V(x_0,R)^\nu\theta^{2\nu}}{\phi({r_1})I_0^\nu}\int_{B_{{r_1}+{r_2}/2}} v^2 \vp^2\, d\mu\\
&\ge
\frac{C''V(x_0,R)^\nu\theta^{2\nu}}{\phi({r_1})I_0^\nu}\int_{B_{{r_1}}}
v^2\, d\mu\\
&=\frac{C''V(x_0,R)^\nu\theta^{2\nu}}{\phi({r_1})I_0^\nu} I_1,
\end{align*}
where in the last inequality we used the fact $\vp=1$ on $B_{{r_1}}$.
Combining the inequality above with \eqref{e:in25} and \eqref{polycon}, we obtain the desired estimate
\eqref{e:IIi}.
\qed \end{proof}

We need the following elementary iteration lemma,
see, e.g., \cite[Lemma 7.1]{Giu}.

\begin{lemma}\label{L:it}
Let $\beta>0$ and let $\{A_j\}$ be a sequence of real positive numbers such that
$$
A_{j+1}\le c_0b^jA_j^{1+\beta}$$ with $c_0>0$ and $b>1$. If
$$A_0\le c_0^{-1/\beta}b^{-1/\beta^2},
$$
 then we have
  \begin{equation}\label{e:it}
  A_j\le b^{-j/\beta}A_0,
  \end{equation}
  which in particular yields $\lim_{j\to\infty}{ A_j}=0.$
  \end{lemma}

\begin{proof}We proceed by induction. The inequality \eqref{e:it} is obviously true for $j=0$. Assume now that holds for $j$. We have
$$
A_{j+1}\le c_0 b^j b^{-j (1+\beta)/\beta}A_0^{1+\beta}=(c_0b^{1/\beta} A_0^\beta) b^{-(j+1)/\beta}A_0\le b^{-(j+1)/\beta}A_0,
$$
so  \eqref{e:it}  holds for $j+1$. \qed
\end{proof}

\begin{proposition} \label{P:mvi2g}
{\bf($L^2$-mean value inequality)}\, Let $x_0\in M$ and $R>0$.
Assume $\VD$, \eqref{polycon}, $\FK(\phi)$, $\CSJ(\phi)$ and
$\J_{\phi,\le}$ hold, and let $u$
be a  bounded
$\sE$-subharmonic function
in $B(x_0,R)$.  Then  for any $\delta>0$,
\begin{equation} \label{e:mvi2g}
 \esssup_{B(x_0,R/2)} u
 \le  c_1\left[ \left(\frac{(1+\delta^{-1})^{1/\nu}}{V(x_0,R)}\int_{B(x_0,{R})} u^2\,d\mu \right)^{1/2}+\delta\T\, (u; x_0,R/2)
  \right] ,
  \end{equation}
 where $\nu$ is the constant appearing in the $\FK(\phi)$
inequality \eqref{e:fki}, and $c_1>0$ is a constant independent of
$x_0$, $R$ and $\delta$.

In particular, there is a constant $c>0$ independent of
$x_0$ and $R$ so that
\be \label{e:mvi2g-1}
 \esssup_{B(x_0,R/2)} u
 \le  c\left[  \left(\frac{1}{V(x_0,R)}\int_{B(x_0,{R})} u^2\,d\mu \right)^{1/2}+ \T\, (u; x_0,R/2)\right].
\ee
\end{proposition}

\begin{proof}  We first set up some notations. For $i\ge0$ and $\theta>0$, let
$r_i=\frac{1}{2}(1+2^{-i}) R$ and $\theta_i=(1-2^{-i})\theta$. For
any $x_0\in M$ and $s>0$, let $B_s=B(x_0,s)$. Define
$$
I_i=\int_{B_{r_i}}(u-\theta_i)_+^2\,d\mu, \quad i\ge 0.
$$
By \cite[Corollary 2.10(iv)]{ChK},
for any
$i\ge0$, $(u-\theta_i)_+$ is an $\sE$-subharmonic function
for the Dirichlet form $(\sE, \sF)$ on $B_R$. Then, thanks to Lemma
\ref{oppo}, by \eqref{e:IIi} applied to the function $(u-\theta_i)$
in $B_{r_{i+1}} \subset B_{r_i}$,
\begin{equation*}\begin{split}
\label{eq:nooed} I_{i+1}&=\int_{B_{r_{i+1}}}(u-\theta_{i+1})_+^2\,d\mu=\int_{B_{r_{i+1}}}\big[(u-\theta_{i})-(\theta_{i+1}-\theta_i)\big]_+^2\,d\mu\\
&\le \frac{c_2}{(\theta_{i+1}-\theta_i)^{2\nu} V(x_0,R)^\nu}I_i^{1+\nu}
\left(\frac{r_{i}}{r_{i}-r_{i+1}}\right)^{\beta_2} \\
&\quad\times \left[ 1+ \left(\frac{r_{i}}
{r_{i}-r_{i+1}}\right)^{d_2+\beta_2-\beta_1}\frac{\T\,(u; x_0,R/2)}{(\theta_{i+1}-\theta_i)}\right] \\
&\le \frac{c_32^{(1+2\nu+d_2+2\beta_2-\beta_1)i}}{\theta^{2\nu} V(x_0,R)^\nu}
 I_i^{1+\nu}\left[1+
\frac{\T\,(u; x_0,R/2)}{\theta } \right].
\end{split}
\end{equation*}

In the following, we take
$$
\theta=\delta{\T\,(u; x_0,R/2)}+\sqrt{ \frac{ c_* I_0}{V(x_0,R)}},
\quad \delta>0,
$$
where $c_*=[(1+\delta^{-1})c_3]^{1/\nu}2^{(1+2\nu+d_2+2\beta_2-\beta_1)/\nu^2}.$
 It holds that
$$
I_0\le \left[ \frac{c_3} {\theta^{2\nu}
V(x_0,R)^\nu}\Big(1+
\frac{\T\,(u; x_0,R/2)}{\theta } \Big)\right]^{-1/\nu}
2^{-(1+2\nu+d_2+2\beta_2-\beta_1)/\nu^2}.
$$
Indeed, since
$r\mapsto \left( \frac{1}{r^{2\nu}}\right)^{-1/\nu}$ and
$r\mapsto \left( 1+\frac{c}{r}\right)^{-1/\nu}$ 
(with $c>0)$  
are both increasing functions
 on $(0,\infty)$,
the right hand side of
the above inequality is larger than
\begin{align*}
&\left( \frac{c_3} {(c_*{I_0}/{V(x_0,R)})^{\nu}
V(x_0,R)^\nu}\Big(1+
\frac{\T\,(u; x_0,R/2)}{\delta{\T\,(u; x_0,R/2)} } \Big)\right)^{-1/\nu}
2^{-(1+2\nu+d_2+2\beta_2-\beta_1)/\nu^2}\\
&=\left(c_3(1+\delta^{-1})/(c_*I_0)^{\nu}\right)^{-1/\nu}2^{-(1+2\nu+d_2+2\beta_2-\beta_1)/\nu^2}=I_0.\end{align*}
Then by Lemma
\ref{L:it}, we have $I_{i} \to 0$ as $i \to \infty$. Hence
$$
 \int_{B_{R/2}} (u- \theta)_+^2\,d\mu \le \inf_i I_i =0,
$$
which implies that
$$
 \esssup_{B_{R/2}} u \le \theta\le c_4\left[  \left( \frac{(1+\delta^{-1})^{1/\nu}I_0}{V(x_0,R)}\right)^{1/2}+\delta\T\,(u; x_0,R/2)
 \right].
$$
 This proves \eqref{e:mvi2g}.\qed
\end{proof}
\sms

In the following, we consider $L^2$ and $L^1$ mean value
inequalities for $\sE$-subharmonic functions for truncated Dirichlet
forms.
In the truncated situation we can no longer
 use the nonlocal tail of subharmonic functions, and the remedy is to enlarge the integral regions
 in the right hand side of the mean value inequalities.
These mean value inequalities will be used in the next subsection to consider the stability of heat kernel.

\begin{proposition} \label{P:mvi2}{\bf($L^2$-mean value inequality for $\rho$-truncated Dirichlet forms)}
Assume $\VD$, \eqref{polycon}, $\FK(\phi)$,
$\CSJ(\phi)$ and $\J_{\phi,\le}$ hold.
There are positive constants $c_1, c_2>0$ so that for $x_0\in M$, $\rho,
R>0$, and for any bounded $\sE^{(\rho)}$-subharmonic function $u$ on $B(x_0,R)$ for the $\rho$-truncated
Dirichlet form $(\sE^{(\rho)}, \sF)$, we have
\be \label{e:mvi3}
 \esssup_{B(x_0,R/2)} u^2  \le  \frac{c_1
 }{V(x_0,R)}\left(1+\frac{\rho}{R}\right)^{d_2/\nu} \Big(1+\frac{R}{\rho}\Big)^{\beta_2/\nu}
    \int_{B(x_0,{R+\rho})} u^2\,d\mu.
\ee
 Here, $\nu$ is the constant in $\FK(\phi)$, $d_2$ and  $\beta_2$ are the exponents in \eqref{e:vd2} from $\VD$
  and \eqref{polycon} respectively.
  \end{proposition}

\begin{proof} The proof is mainly based on that of Proposition \ref{P:mvi2g}. For simplicity, we only present the main different steps, and the details are left to
the interested readers.

First, we apply the argument in the proof of Lemma \ref{L:ci} to the $\rho$-truncated
Dirichlet form $(\sE^{(\rho)}, \sF)$. In this truncated setting,  we estimate the term $I_2$ in \eqref{e:ci1} as follows.
\begin{align*}
I_2=& \int_{B_{R}\times {B^c_{R+r}}}(u(x)-u(y)) \vp^2(x)v(x)\, J^{(\rho)}(dx,dy)\\
\ge &-\int_{B_{R}}v(x)\,\mu(dx)\left[\sup_{x\in B_{R}}\int_{{B^c_{R+r}}} v(y)J^{(\rho)}(x,y)\,\mu(dy)\right]\\
\ge&-\frac{1}{\theta}\int_{B_{R}}  u^2(x)\,\mu(dx)\left[\sup_{x\in B_{R}}\int_{{B^c_{R+r}}} v(y)J^{(\rho)}(x,y)\,\mu(dy)\right]\\
\ge&-\frac{1}{\theta}\int_{B_{R}} u^2(x)\,\mu(dx)
\left[\frac{c_1}{\phi(r)}\bigg( \sup_{x\in B_{R}} \frac{1}{V(x,r)}\bigg)\int_{B_{R+\rho}}v(y)\,\mu(dy)\right] \\
\ge&-\frac{c_2 }{\phi(r)}
\bigg[ \left(\frac{R+\rho}{r}\right)^{d_2} \frac{1}{\theta
V(x_0,R+\rho)} \int_{B_{R+\rho}}|u|(y)\,\mu(dy)\bigg] \int_{B_{R}}
u^2(x)\,\mu(dx),
\end{align*}
 where in the second and third inequality we have used  the fact that $v\le vu/\theta\le u^2/\theta$
and the condition $\J_{\phi,\le}$ respectively,  while   the last inequality follows from that for any $x\in B_{R}$,
 $$ \frac{V(x,r)}{V(x_0,R+\rho)}\ge \frac{V(x,r)}{V(x,2R+\rho)} \ge c'\Big(\frac{R+\rho}{r}\Big)^{-d_2},$$ thanks to
 $\VD$.

On the other hand, we do  the upper estimate for $\int_{B_{R+r}}\,d\Gamma(v\vp,v\vp)$ just as \eqref{e:cac33},
but using $\rho$-truncated Dirichlet form $(\sE^{(\rho)}, \sF)$ instead. Indeed, we have \begin{align*}
\int_{B_{R+r}}\,d\Gamma(v\vp,v\vp)
\le&\int_{B_{R+r}\times M}(v(x)\vp(x)-v(y)\vp(y))^2\,J^{(\rho)}(dx,dy)
  \\ &
 +2\int_{B_{R}}v^2(x)\vp^2(x)\int_{d(x,y)\ge \rho}\,J(dx,dy)\\
&+2\int_Mv^2(y)\vp^2(y)\int_{d(x,y)\ge \rho}\,J(dx,dy)\\
\le &\int_{B_{R+r}\times B_{R+r}}\big(v(x)(\vp(x)-\vp(y))+\vp(y)(v(x)-v(y))\big)^2\,J^{(\rho)}(dx,dy)\\
&+\int_{B_{R}} v^2(x)\vp^2(x) \int_{B^c_{R+r}}\,J^{(\rho)}(dx,dy) +\frac{c'_1}{\phi(\rho)}\int_{B_{R}}v^2\,d\mu \\
\le &2\Big(
\int_{B_{R+r}\times B_{R+r}}v^2(x)(\vp(x)\!-\!\vp(y))^2\,J^{(\rho)}(dx,dy)\\
&\quad +\int_{B_{R+r}\times
B_{R+r}}\vp^2(x)(v(x)\!-\!v(y))^2\,J^{(\rho)}(dx,dy)\Big)+\frac{c''_1}{\phi(\rho\wedge r)}\int_{B_{R}}v^2\,d\mu\\
\le &2\int_{B_{R+r}}v^2 \,d\Gamma^{(\rho)}(\vp,\vp)
+2\int_{B_{R}\times
B_{R+r}}\vp^2(x)(v(x)\!-\!v(y))^2\,J^{(\rho)}(dx,dy)\\
&\quad +\frac{c''_2}{\phi(r)}\left(1+ \frac{r}{\rho}\right)^{\beta_2}\int_{B_{R}}u^2\,d\mu.
\end{align*}

Having both two estimates above at hand, one can change \eqref{e:cacc} in Lemma \ref{L:ci} into
\begin{equation*}\begin{split}
&\int_{B_{R+r}}\,d\Gamma(v\vp,v\vp)\\
&\le  \frac{c }{\phi(r)}
 \bigg[
  1+\left(1+ \frac{r}{\rho}\right)^{\beta_2}
  +\left(\frac{R+\rho}{r}\right)^{d_2}  \frac{1}{\theta V(x_0,R+\rho)} \int_{B_{R+\rho}}u\,d\mu\bigg]
\int_{B_{R+r}} u^2\,d\mu,
\end{split}\end{equation*}
where $c>0$ is a constant independent of
$x_0, R, r, \rho $ and $\theta$. This in turn gives us the following conclusion instead of \eqref{e:IIi} in Lemma \ref{oppo}:
$$
 I_1\le \frac{ c_1 }{ \theta^{2\nu} V(x_0,R)^\nu} I_0^{1+\nu}  \left(
\frac{r_1}{r_2}\right)^{\beta_2}\bigg[ 1+\left(1+\frac{r_2}{\rho}\right)^{\beta_2}+ \left(\frac{r_1+\rho}{r_2}\right)^{d_2}
\frac{1}{\theta V(x_0,R+\rho)}\int_{B_{R+\rho}} |u| \, d\mu \bigg] .
$$

Finally, following the argument of
Proposition \ref{P:mvi2g},  we can obtain that for any bounded $\sE^{(\rho)}$-subharmonic function $u$ associated with the
 $\rho$-truncated
Dirichlet form $(\sE^{(\rho)}, \sF)$ on $B(x_0,R)$, it holds
\be  \label{e:mvi2} \begin{split}
 \esssup_{B(x_0,R/2)} u^2
 & \le  c_0\bigg[  \bigg(\frac{1}{V(x_0, R+\rho)}\int_{B(x_0,{R+\rho})} u \,d\mu\bigg)^2\\
 &\quad\quad\,\, +\left(1+\frac{\rho}{R}\right)^{d_2/\nu} \Big(1+\frac{R}{\rho}\Big)^{\beta_2/\nu}
     \frac{1}{V(x_0,R)}\int_{B(x_0,{R})} u^2\,d\mu\bigg],
     \end{split}
\ee
where  $\nu$ is the constant in $\FK(\phi)$, $d_2$ and  $\beta_2$ are the constants in $\VD$ and \eqref{polycon} respectively, and $c_0>0$ is a constant independent of $x_0$, $\rho$ and
$R$. Hence, the desired assertion \eqref{e:mvi3} immediately follows from \eqref{e:mvi2}.
\qed
\end{proof}
As a consequence of Proposition \ref{P:mvi2}, we have the following $L^1$-mean value inequality for truncated Dirichlet forms.

\begin{corollary} {\bf($L^1$-mean value inequality for $\rho$-truncated Dirichlet forms)}
Assume $\VD$, \eqref{polycon}, $\FK(\phi)$,
$\CSJ(\phi)$ and $\J_{\phi,\le}$ hold.
There are positive constants $c_1, c_2>0$ so that for $x_0\in M$, $\rho,
R>0$, and for any non-negative, bounded and $\sE^{(\rho)}$-subharmonic function $u$ on $B(x_0,R)$ for the $\rho$-truncated
Dirichlet form $(\sE^{(\rho)}, \sF)$, we have
\be \label{e:mvi}
 \esssup_{B(x_0,R/2)} u \le  \frac{c_2}{V(x_0,R)} \Big(1+\frac{\rho}{R}\Big)^{d_2/\nu}\Big(1+\frac{R}{\rho}\Big)^{\beta_2/\nu}
    \int_{B(x_0,R+\rho)} u \,d\mu.
\ee
 Here, $\nu$ is the constant in $\FK(\phi)$, $d_2$ and  $\beta_2$ are the exponents in \eqref{e:vd2} from $\VD$
  and \eqref{polycon} respectively.
\end{corollary}
\begin{proof}
Fix $x_0\in M$ and $R>0$. For any $s>0$, let $B_s=
B(x_0, s)$.
For $n\ge 0$, let $r_n=R2^{-n-1}$.
Note that $\{r_n\}$ is decreasing such that $r_0=R/2$ and $r_\infty=0$,
 and $\{B_{r_n}\}$ is decreasing and $\{B_{R-r_n}\}$ is increasing such that $B_0=B_{R-r_0}=B(x_0, R/2)$ and $B_{R-r_\infty}=B(x_0,R)$.
Take arbitrary point $\xi\in B_{R-r_{n-1}}$;
 then since $r_n=r_{n-1}/2$, we have $B(\xi,r_n)\subset B_{R-r_n}$.
Applying \eqref{e:mvi2} with $x_0=\xi$ and $R=r_n$, we have for almost $\xi$
\be\label{eq:niebi2q1}\begin{split} u(\xi)^2\le
c_1\bigg[& \bigg(\frac{1}{V(\xi,r_n+\rho)}\int_{B(\xi,r_n+\rho)}u\,d\mu\bigg)^2\\
&+{\Big(1+\frac{\rho}{r_n}\Big)^{d_2/\nu}\Big(1+\frac{r_n}{\rho}\Big)^{\beta_2/\nu}\frac{1}{V(\xi,r_n)}\int_{B(\xi,r_n)}u^2\,d\mu}\bigg],\end{split}
\ee where $c_1>0$ does not depend on $\xi$, $r_n$ and $\rho$.

In the following, let
\[
M_n= \esssup_{B_{R-r_n}} u~~~\mbox{and }~~~A=\frac 1{V(x_0,R)}\int_{B(x_0,R+\rho)}u\,d\mu.
\]
Since
$B(\xi,r_n)\subset B_{R-r_n}$, we have
\[
\int_{B(\xi,r_n)}u^2\,d\mu\le M_n V(x_0,R)A.
\]
Note that, by $\VD$,
$$\frac{V(x_0,R)}{V(\xi,r_n+\rho)} \le c'\left(\frac{d(x_0,\xi)+R}{r_n+\rho}\right)^{d_2}\le c' 2^{(n+2)d_2}$$ and
$$\frac{V(x_0,R)}{V(\xi,r_n)}\le c'' 2^{(n+2)d_2}.$$
Plugging these estimates into \eqref{eq:niebi2q1}, we have
\[
u(\xi)^2\le c_22^{2nd_2}A^2+c_3\Big(1+\frac{\rho}{R}\Big)^{d_2/\nu}\Big(1+\frac{R}{\rho}\Big)^{\beta_2/\nu} M_nA2^{nd_2(1+1/\nu)}.
\]
Since $\xi$ is an arbitrary point chosen almost in $B_{R-r_{n-1}}$, we obtain
\be\label{eq:926bif}
M_{n-1}^2\le c_4\Big(1+\frac{\rho}{R}\Big)^{d_2/\nu}\Big(1+\frac{R}{\rho}\Big)^{\beta_2/\nu} (A+e^{nb(1/\nu-1)}M_n)e^{2nb}A,
\ee where $b=d_2\log 2$.

Our goal is to prove $$M_0\le c_0\Big(1+\frac{\rho}{R}\Big)^{d_2/\nu}\Big(1+\frac{R}{\rho}\Big)^{\beta_2/\nu}A$$ for some constant $c_0>0$ independent of $x_0$, $R$ and $\rho$. If $M_0\le A$, then we are done, and so
we only need to consider the case $M_0>A$.  Then  $A<M_0\le e^{nb(1/\nu-1)}M_n$ for all $n\ge 0$,
because $\{M_n\}$ is increasing and, without loss of generality, we may and do assume that
the constant $\nu$ in $\FK (\phi)$ is strictly less than 1.
 Therefore, \eqref{eq:926bif}
implies $$M_{n-1}^2\le 2c_4\Big(1+\frac{\rho}{R}\Big)^{d_2/\nu}\Big(1+\frac{R}{\rho}\Big)^{\beta_2/\nu}e^{nb(1+1/\nu)}M_nA.$$
From here we can argue similarly to \cite[p.\ 689-690]{CG}.
By iterating the inequality above, we have
\[
M_0^{2^n}\le \exp\left(b(1+1/\nu)\sum_{i=1}^n i2^{n-i}\right)\left[ 2c_4\Big(1+\frac{\rho}{R}\Big)^{d_2/\nu}\Big(1+\frac{R}{\rho}\Big)^{\beta_2/\nu}A\right]^{1+2+2^2+\cdots +2^{n-1}}M_n.
\]
So
\begin{align*}
M_0\le & c_5\left[ 2c_4\Big(1+\frac{\rho}{R}\Big)^{d_2/\nu}\Big(1+\frac{R}{\rho}\Big)^{\beta_2/\nu}A\right]^{1-2^{-n}}M_n^{2^{-n}}\\
\le &c_6\left[ \Big(1+\frac{\rho}{R}\Big)^{d_2/\nu}\Big(1+\frac{R}{\rho}\Big)^{\beta_2/\nu}\right] A(M_n/A)^{2^{-n}}.
\end{align*}
Since $u$ is bounded in $B_R$, $M_n\le c_7$ for all $n\ge 0$ and some constant $c_7>0$, so we have
$\lim_{n\to\infty}(M_n/A)^{2^{-n}}=1$. We thus obtain $$M_0\le c_6\Big(1+\frac{\rho}{R}\Big)^{d_2/\nu}\Big(1+\frac{R}{\rho}\Big)^{\beta_2/\nu}A.$$ The proof is complete.
\qed\end{proof}

\subsection{$\FK(\phi)+\J_{\phi,\le}+\CSJ(\phi)\Longrightarrow \E_\phi$} \label{s:hku}

The main result of this subsection is as follows.
\begin{proposition} \label{P:exit}
Assume $\VD$, \eqref{polycon}, $\FK(\phi)$, $\J_{\phi,\le}$ and $\CSJ(\phi)$ hold.  Then  $\E_{\phi}$ holds.
\end{proposition}

In order to prove this, we first show that

\begin{lemma}\label{upper-e} Assume that $\VD$, \eqref{polycon} and $\FK(\phi)$ hold.  Then  $\E_{\phi,\le} $ holds.\end{lemma}

\begin{proof}
By Proposition \ref{P:3.3}, under $\VD$ and \eqref{polycon},
$\FK(\phi)$ implies that there is  a constant $C>0$ such that for
any ball $B:=B(x,r)$ with $x\in M$ and $r>0$,
$$\esssup_{x',y'\in B} p^B(t,x',y')\le \frac{C}{V(x,r)}\left( \frac{\phi(r)}{t}\right)^{1/\nu},$$ where $\nu$ is the constant in $\FK(\phi)$.
 Then  for any $T\in(0,\infty)$ and all $x\in M_0$,
 \begin{align*}
\bE^x\tau_B &=\int_0^\infty P_t^B{\bf 1}_{B}(x)\,dt=\int_0^T  P_t^B{\bf 1}_{B}(x)\,dt+ \int_T^\infty  P_t^B{\bf 1}_{B}(x)\,dt\\
& \le T+\int_T^\infty \int_B p^B(t,x,y)\,\mu(dy)\,dt\\
& \le T+C\int_T^\infty \left( \frac{\phi(r)}{t}\right)^{1/\nu}\,dt
\le T+C_1\phi(r)^{1/\nu}T^{1-1/\nu},
\end{align*}
 where in the last inequality we have used the fact that the constant $\nu$ in $\FK(\phi)$ can be assumed that $\nu\in(0,1)$.
Setting $T=\phi(r)$, we conclude that
$\bE^x\tau_B\le C_2\phi(r).$ This proves $\E_{\phi,\le}$. \qed\end{proof}

Let $\{X_t^{(\rho)}\}$ be the Hunt process associated with the
$\rho$-truncated Dirichlet form $(\sE^{(\rho)}, \sF)$. For $\lambda>0$, let
$\xi_\lambda$ be an exponential distributed random variable with
mean $1/\lambda$, which is independent of the $\rho$-truncated process
$\{X_t^{(\rho)}\}$.

\begin{lemma}\label{lower-e00} Assume that $\VD$, \eqref{polycon}, $\FK(\phi)$,  $\J_{\phi,\le}$  and $\CSJ(\phi)$ hold.
 Then  for any $c_0\in (0,1)$, there exists a constant $c_1>0$ such
that for all $R>0$ and all $x\in M_0$,
\[
\bE^x\Big[ \tau^{(c_0R)}_{B(x,R)}\wedge \xi_{\phi(R)^{-1}}\Big] \ge  c_1\phi(R).
\]
\end{lemma}

\begin{proof} For fixed $c_0\in (0,1)$ and $R>0$, set $\rho=c_0R$. Set $B=B(x,R)$, $\lam=1/\phi(R)$ and
$u_\lam(x)=\bE^x\big(\tau^{(\rho)}_{B}\wedge \xi_\lam\big)$ for $x\in M_0$; here
and in the following we make some abuse of notation and use $\bE^x$
for the expectation of the product measure of the truncated process
$\{X_t^{(\rho)}\}$ and $\xi_\lam$.  Then  for all $x\in M_0$,
\[
u_\lam(x)=\bE^x\left[ \int_0^{\tau^{(\rho)}_{B}\wedge
\xi_\lam}{\bf1}(X_t^{(\rho)})\,dt\right] =
\bE^x\left[ \int_0^{\tau^{(\rho)}_{B}}e^{-\lam
t}{\bf1}(X_t^{(\rho)})\,dt\right] =G_\lam^{(\rho),B}{\bf1}(x),
\]
where $G_\lam^{(\rho),B}$ is the $\lam$-order resolvent for the
truncated process $\{X_t^{(\rho)}\}$ killed on exiting $B$.
Clearly $u_\lambda$ is bounded and  is in $\sF^{(\rho)}_B$. Moreover, $u_\lambda (X^{(\rho)}_{t\wedge \tau_B^{(\rho)}})$
is a bounded supermartingale under $\bP^x$ for every $x\in B\cap M_0$.

Set $u_{\lam,\varepsilon}=u_\lam+\varepsilon$ for any
$\varepsilon>0$.
Note that $u_{\lam,\varepsilon} \in \FF^{loc}$ as  $u_\lam\in \sF_B \subset \sF$.
Since $t\mapsto u_{\lam,\varepsilon} (X^{(\rho)}_{t\wedge \tau_B^{(\rho)}})$
is a bounded supermartingale under $\bP^x$ for every $x\in B\cap M_0$,
we have by Theorem \ref{equ-har} that
$u_{\lam,\varepsilon}$ is
$\sE^{(\rho)}$-superharmonic in $B$.
 By $\J_{\phi,\le}$, $\CSJ(\phi)$ and  Proposition
\ref{CSJ-equi}(5), we can choose a non-negative cut-off function
$\varphi\in \sF^{(\rho)}_{B}$ for $\frac{1+c_0}{2}B\subset B$ such that
$$\sE^{(\rho)}(\varphi,\varphi)\le \frac{c_1\mu(B)}{\phi(R)}$$ and
so
$$
\sE_\lam^{(\rho)}(\varphi,\varphi)=\sE^{(\rho)}(\varphi,\varphi)
+\lam\langle \varphi,\varphi\rangle\le \frac{c_1\mu(B)}{\phi(R)}+
\lam\mu(B)\le\frac{c_2\mu(B)}{\phi(R)}.
$$
Furthermore, choose a continuous function $g$ on $[0,\infty)$ such that $g(0)=0$, $g(t)=\varepsilon^2/t$ for $t\ge \varepsilon$ and
 $|g(t)-g(s)|\le |t-s|$ for all $t, s\ge0$. According to \cite[Theorem 1.4.2 (v) and (iii)]{FOT},
 $u_{\lam,\varepsilon}^{-1}= g(u_{\lam,\varepsilon}) /\eps^2$ is a bounded function in $\sF^{loc}$. It follows then 
$u_{\lam,\varepsilon}^{-1}\vp^2\in \sF_{B} = \sF^{(\rho)}_{{B}}$, 
 since  $\varphi$ 
   is a bounded element in $\sF_{B}$
and balls are relatively compact in $M$.  
We deduce from
$$(u_{\lam,\varepsilon}(x)-u_{\lam,\varepsilon}(y))(u_{\lam,\varepsilon}(x)^{-1}\varphi^2(x)-u_{\lam,\varepsilon}(y)^{-1}\varphi^2(y) )\le (\varphi(x)-\varphi(y))^2$$ that
$$\sE^{(\rho)}_\lam(u_{\lam,\varepsilon}, u_{\lam,\varepsilon}^{-1}\varphi^2)=\sE^{(\rho)}(u_{\lam,\varepsilon},
u_{\lam,\varepsilon}^{-1}\varphi^2)+\lam\langle
u_{\lam,\varepsilon}, u_{\lam,\varepsilon}^{-1}\varphi^2\rangle \le \sE^{(\rho)}(\varphi,\varphi)+\lam\langle
\varphi,\varphi\rangle=\sE^{(\rho)}_\lam(\varphi,\varphi).$$
Therefore,
$$\sE^{(\rho)}_\lam(u_{\lam,\varepsilon}, u_{\lam,\varepsilon}^{-1}\varphi^2)\le \frac{c_2\mu(B)}{\phi(R)}.$$
On the other hand,  noticing again that $u_{\lam,\varepsilon}^{-1}\varphi^2\in \sF^{(\rho)}_{B},$
\begin{align*}
\sE^{(\rho)}_\lam(u_{\lam,\varepsilon}, u_{\lam,\varepsilon}^{-1}\varphi^2)&=\varepsilon\sE^{(\rho)}_\lam(1, u_{\lam,\varepsilon}^{-1}\varphi^2)+\sE^{(\rho)}_\lam(u_\lambda, u_{\lam,\varepsilon}^{-1}\varphi^2)\\
&= \varepsilon\lambda\langle
1,u_{\lam,\varepsilon}^{-1}\varphi^2\rangle+
\langle 1,u_{\lam,\varepsilon}^{-1}\varphi^2\rangle\\
&\ge\langle 1,u_{\lam,\varepsilon}^{-1}\varphi^2\rangle\ge
\int_{\frac{1+c_0}{2}B}u_{\lam,\varepsilon}^{-1}\,d\mu,
\end{align*}
and so
$$\int_{\frac{1+c_0}{2}B}u_{\lam,\varepsilon}^{-1}\,d\mu\le \frac{c_2\mu(B)}{\phi(R)}.$$

Noting that $u_{\lam,\varepsilon}$ is
 bounded and  $\sE^{(\rho)}$-superharmonic  in $B$,
$u_{\lam,\varepsilon}(X^{(\rho)}_{t\wedge \tau_U^{(\rho)}})$ is a uniformly integrable $\bP^x$-supermartingale for any relatively compact open subset $U$ of $B$ and q.e.\ $x\in B$ by Theorem \ref{equ-har}.
Since  $u_{\lam,\varepsilon}\geq \eps$,
we have by  Jensen's inequality applied to the convex function $1/x$ that
$u^{-1}_{\lam,\varepsilon}(X^{(\rho)}_{t\wedge \tau_U^{(\rho)}})$ is a uniformly integrable $\bP^x$-submartingale for any relatively compact open subset $U$ of $B$ and q.e.\ $x\in B$.
Using
 Theorem \ref{equ-har} again, we can conclude that $u_{\lam,\varepsilon}^{-1}$
is an $\sE^{(\rho)}$-subharmonic function in $B$.
 Applying the
$L^1$-mean value inequality \eqref{e:mvi} to
$u_{\lam,\varepsilon}^{-1}$ on $\frac{1-c_0}{2}B$, we get that
$$ \esssup_{ \frac{1-c_0}{4}B}u_{\lam,\varepsilon}^{-1}\le \frac{c_3}{\mu(B)}\int_{\frac{1+c_0}{2}B}u_{\lam,\varepsilon}^{-1}\,d\mu\le \frac{c_4}{\phi(R)}.$$
Whence,
$\essinf_{\frac{1-c_0}{4}B}u_{\lam,\varepsilon}\ge
c_5\phi(R)$.  Letting $\varepsilon\to 0$, we get
$ \essinf_{ \frac{1-c_0}{4}B} u_\lam\ge  c_5\phi(R)$.
This yields the desired estimate.
\qed \end{proof}

 The next lemma is standard.

\begin{lemma}\label{E}
If $\E_{\phi}$ holds, then for all $x\in M_0$ and $r, t>0$,
\be \label{e:4.29}
\bP^x (\tau_{B(x,r)}\le t) \le 1- \frac{c_1\phi(r)}{\phi(2r)}+ \frac{c_2t}{\phi(2r)}.
\ee
In particular, if \eqref{polycon} and $\E_{\phi}$ hold, then
$\EP_{\phi, \le, \varepsilon}$ holds, i.e.\ for any ball
$B:=B(x_0,r)$ with $x_0\in M$ and radius $r>0$, there are constants
$\delta, \varepsilon \in (0,1)$ such that
\begin{equation}\label{l:prob}
\bP^x (\tau_{B}\le t) \le \varepsilon  \quad \textrm{ for all }x\in B(x_0,r/4)\cap M_0
\end{equation}
provided that $t\le \delta \phi(r)$. \end{lemma}

\begin{proof} Suppose that there are constants $c_2\ge c_1>0$ such that for all $x\in M_0$ and $r>0$, $$c_1 \phi(r)\le \bE^x\tau_{B(x,r)} \le c_2 \phi(r) .$$ Since for any $t>0$, $\tau_{B(x,r)}\le t+(\tau_{B(x,r)}-t){\bf 1}_{\{\tau_{B(x,r)}> t\}},$ we have by the Markov property
\begin{align*}
 \bE^x\tau_{B(x,r)}&\le t+ \bE^x\Big[{\bf 1}_{\{\tau_{B(x,r)}>t\}} \bE^{X_t}[ \tau_{B(x,r)}-t ]\Big]
 \le t+ \bP^x(\tau_{B(x,r)}>t) \sup_{z\in B(x,r)} \bE^{z}\tau_{B(x,r)}\\
 &\le t+ \bP^x(\tau_{B(x,r)}>t) \sup_{z\in B(x,r)} \bE^{z}\tau_{B(z,2r)}
 \le t+ c_2\bP^x(\tau_{B(x,r)}>t)  \phi(2r).
\end{align*}
Then for all $x\in M_0$, $c_1\phi(r)\le \bE^x\tau_{B(x,r)}\le  t+  c_2\bP^x(\tau_{B(x,r)}>t) \phi(2r)$,
proving \eqref{e:4.29}.   Since
$$
\bP^x (\tau_{B(x_0,r)}\le t) \le \bP^x(\tau_{B(x,3r/4)}\le t),
\quad x\in B(x_0,r/4)\cap M_0.
$$
inequality \eqref{l:prob} follows from   \eqref{e:4.29}  and \eqref{polycon}.    \qed
\end{proof}

\begin{lemma}\label{lower-e} Assume that $\VD$, \eqref{polycon}, $\FK(\phi)$,  $\J_{\phi,\le}$  and $\CSJ(\phi)$ hold.  Then  there exists a constant $c_1>0$ such that for all $x\in M_0$ and all $R>0$,
$$
\bE^x\tau_{B(x,R)}\ge c_1\phi(R).$$
\end{lemma}
\begin{proof}Let $B=B(x,R)$, $\rho=cR$ for some $c\in (0,1)$ and $\lambda=1/\phi(R)$. Recall that
$\xi_\lambda$ is an exponential distributed random variable with
mean $1/\lambda$, which is independent of the $\rho$-truncated process
$\{X_t^{(\rho)}\}$. Since it is clear that for all $x\in M_0$,
$$
\bE^x\Big[\tau_{B}^{(\rho)}\wedge \xi_\lam\Big]\le \bE^x \xi_\lam =\phi(R),
$$
using Lemma \ref{lower-e00}, we have $$
\bE^x\Big[ \tau_{B}^{(\rho)}\wedge \xi_\lam\Big] \asymp \phi(R).$$
So by an argument similar to that of Lemma \ref{E},  we have for all $x\in M_0$,
\[
\bP^x\big(\tau_{B}^{(\rho)}\wedge \xi_\lam\le t\big)\le
1-c_1+c_2t/\phi(R).
\]
In particular, choosing $c_3>0$
small enough, we have
$$\bP^x(\tau_{B}^{(\rho)} \ge c_3\phi(R))\ge  \bP^x\big(\tau_{B}^{(\rho)}\wedge \xi_\lam \ge c_3\phi(R)\big)\ge c_4>0.$$

Next, let $T_\rho$ be the first time that the size of jump bigger than $\rho$ occurs for the original process $\{X_t\}$, and let $\{X_t^{(\rho)}\}$ be the truncated process associated with $\{X_t\}$.
Then, as in the proof of \cite[Lemma 3.1(a)]{BGK1}, we have
\[
\bP^x(T_\rho>t|\sF_\infty^{X^{(\rho)}})=\exp\left(-\int_0^t {\cal J} (X^{(\rho)}_s)\,ds\right)\ge
e^{-c_5t/\phi(\rho)},
\] where
$$
 {\cal J} (x):=\int_{B(x,\rho)^c}J(x,y)\,\mu(dy)\le c_5/\phi(\rho),
$$
  thanks to Lemma \ref{intelem}.  So
 $$
 \bP^x(T_\rho> c_3\phi(R)|\sF_\infty^{X^{(\rho)}})\ge c_6.
 $$
This implies
\begin{align*}
\bP^x\big(\tau_B^{(\rho)}\wedge T_\rho>
c_3\phi(R)\big)=\bE^x\Big[ {\bf1}_{\{\tau_{B}^{(\rho)}\ge
c_3\phi(R)\}} \bE^x\Big[ {\bf1}_{\{T_\rho>
c_3\phi(R)\}}|\sF_\infty^{X^{(\rho)}}\Big] \Big] \ge c_4c_6>0.
\end{align*}

Note that $\tau_B\ge \tau_B^{(\rho)}\wedge T_\rho$. (In fact, if
$\tau_B^{(\rho)}<T_\rho$, then $\tau_B= \tau_B^{(\rho)};$ if
$\tau_B^{(\rho)}\ge T_\rho$, then, by the fact that the truncated
process $\{X^{(\rho)}_t\}$ coincides with the original $\{X_t\}$
till $T_\rho$, we also have $\tau_B\ge T_\rho$.)  We obtain
$$\bP^x(\tau_B> c_3\phi(R))\ge c_4c_6,$$ and so the desired estimate holds.
\qed\end{proof}

\subsection{$\FK(\phi)+\E_\phi+\J_{\phi,\le}\Longrightarrow \UHKD(\phi)$}\label{FKEPHIj}

If $V(x,r)\asymp r^d$ for each $r>0$ and $x\in M$ with some constant $d>0$, then
$\FK(\phi)\Longrightarrow \UHKD(\phi)$ is well-known; e.g.\ see the remark in the proof of \cite[Theorem 4.2]{GT}. 
However, in the setting when the volume function $V(x, r)$ is not comparable to a non-decreasing function $V(r)$ independent of $x$ as is the case in this paper, 
it is highly non-trivial to establish the on-diagonal upper bound estimate $\UHKD(\phi)$ from $\FK(\phi)$. Below, we will adopt the truncating argument and significantly modify the  iteration techniques in \cite[Proof of Theorem 2.9]{Ki} and \cite[Lemma 5.6]{GH}. Without further mention, throughout the proof we will assume that $\mu$ and $\phi$ satisfy $\VD$ and \eqref{polycon}, respectively.

 Recall that for   $\rho>0$, $(\sE^{(\rho)}, \sF)$ is the $\rho$-truncated Dirichlet form
 defined
as in \eqref{eq:rhoEdef}.
It is clear that for any function $f,g\in \sF$ with ${\rm dist} (\textrm{supp}\, f, \textrm{supp} \, g)>\rho$, $\sE^{(\rho)}(f,g)=0$.
For any non-negative open set ${D} \subset M$, denote by $\{P_t^{D}\}$ and $\{Q_t^{(\rho),{D}}\}$ the semigroups of $(\sE, \sF_{D})$ and $(\sE^{(\rho)}, \sF_{D})$, respectively. We write $\{Q_t^{(\rho),M}\}$ as $\{Q_t^{(\rho)}\}$ for simplicity.

We next give the following preliminary heat kernel estimate.

\begin{lemma}\label{odupper} Suppose that  $\VD$, \eqref{polycon}, $\FK(\phi)$ and $\J_{\phi,\le}$ hold. For any ball $B=B(x,r)$ with $x\in M$ and $r>0$, the semigroup  $\{Q_t^{(\rho),B}\}$ possesses the heat kernel $q^{(\rho),B}(t,x,y)$, which satisfies that there exist constants $C, c_0, \nu>0$ $($independent of $\rho$$)$ such that for all $t>0$ and $x',y'\in B \cap M_0$,
$$ q^{(\rho),B}(t,x',y')\le \frac{C}{V(x,r)} \left(\frac{\phi(r)}{t}\right)^{1/\nu
} \exp\left( \frac{ c_0t }{\phi(\rho)}\right).$$      \end{lemma}
\begin{proof}First, by Proposition \ref{P:3.3}, $\FK(\phi)$ implies that
there exist constants $C_1,\nu>0$ such that
for any ball $B=B(x,r)$,
\[
\frac{V(x,r)^\nu}{\phi(r)}\|u\|_2^{2+2\nu}\le C_1\sE(u,u)\|u\|_1^{2\nu},\quad \forall u\in \sF_B.
\]
According to \eqref{dfcomp}, there is a constant $c_0>0$ such that
$$ \frac{V(x,r)^\nu}{\phi(r)}\|u\|_2^{2+2\nu} \|u\|_1^{-2\nu}\le C_1
\Big(\sE^{(\rho)}(u,u)+\frac{c_0\|u\|^2_{2}}{ \phi(\rho)}\Big)
=:C_1\sE^{(\rho)}_{c_0/\phi(\rho)}(u,u),\quad \forall u\in \sF^{(\rho)}_B.$$
According to Proposition \ref{P:3.3} again (to the Dirichelt form $\sE^{(\rho)}_{c_0/\phi(\rho)}$),
this yields the required assertion.\qed
\end{proof}

Let $\{X_t^{(\rho)}\}$ be the Hunt process associated with
the Dirichlet form $(\sE^{(\rho)}, \sF)$. For any open
subset $D$, let $\tau^{(\rho)}_{D}$ be the first exit time from $D$ by
the Hunt process $\{X_t^{(\rho)}\}$.

\begin{lemma}\label{killing} Suppose that $\VD$, \eqref{polycon}, $\E_\phi$ and $\J_{\phi,\le}$ hold.  Then  there are constants $c_1,c_2>0$ such that for any $r,t,\rho>0$,
$$\bP^x(\tau^{(\rho)}_{B(x,r)}\le t)\le 1- c_1+ \frac{c_2t}{\phi(2r)\wedge \phi(\rho)},\quad x\in M_0.$$ \end{lemma}

\begin{proof} First, by \eqref{polycon}, $\E_\phi$ and Lemma \ref{E}, we know that for all $x\in M_0$ and $r, t>0$,
$$\bP^x (\tau_{B(x,r)}\le t) \le 1- c_1+ \frac{c_2t}{\phi(2r)}.$$
Denote by $B=B(x,r)$ for $x\in M$ and $r>0$. According to Lemma \ref{L:semicomp1}, for all $t>0$ and all $x\in M_0$,
\be\label{eq:compatr}
P_t^B{\bf 1}_B(x)\le Q_t^{(\rho),B}{\bf 1}_B(x)+\frac{ c_3t }{\phi(\rho)}.
\ee
Combining both estimates above with the facts that
$$
1-P_t^B{\bf 1}_B(x)= \bP^x(\tau_{B}\le t),\quad 1-Q_t^{(\rho),B}{\bf 1}_B(x)= \bP^x(\tau^{(\rho)}_{B}\le t),$$
we prove the desired assertion.  \qed \end{proof}

\begin{lemma}\label{killing-1}
Suppose that $\VD$, \eqref{polycon}, $\E_\phi$ and $\J_{\phi,\le}$ hold.  Then  there are constants $\varepsilon\in(0,1)$ and $c>0$ such that for any $r,\lambda,\rho>0$ with $\lambda\ge \frac{c}{\phi(r\wedge \rho)},$
$$
\bE^x [ e^{-\lambda\tau^{(\rho)}_{B(x,r)}}] \le 1-\varepsilon,\quad x\in M_0.
$$
 \end{lemma}

\begin{proof} Denote by $B=B(x,r)$. Using Lemma \ref{killing}, we have for any $t>0$ and all $x\in M_0$,
\begin{align*}
\bE^x \left[ e^{-\lambda\tau^{(\rho)}_B} \right]
&=\bE^x  \left[ e^{-\lambda\tau^{(\rho)}_B}{\bf 1}_{\{\tau^{(\rho)}_B<t\}} \right]
+\bE^x \left[ e^{-\lambda\tau^{(\rho)}_B}{\bf 1}_{\{\tau^{(\rho)}_B\ge t\}} \right] \\
&\le \bP^x(\tau^{(\rho)}_B<t)+e^{-\lambda t}
\le 1- c_1+ \frac{c_2t}{\phi(2r)\wedge \phi(\rho)}+e^{-\lambda t} .
\end{align*}

Next, set $\varepsilon=\frac{c_1}{4}>0.$ Taking $t=c_3\phi(r\wedge \rho)$ for some $c_3>0$ such that $\frac{c_2t}{\phi(2r)}+ \frac{ c_2t }{\phi(\rho)}\le 2\varepsilon$, and $\lambda>0$ such that $e^{-\lambda t}\le \varepsilon$ in the inequality above, we  obtain the desired assertion. \qed
 \end{proof}

The following lemma furthermore improves the estimate established in Lemma \ref{killing-1}.

\begin{lemma}\label{killing-2} Suppose that $\VD$, \eqref{polycon}, $\E_\phi$ and $\J_{\phi,\le}$ hold.  Then there exist constants $C, c_0>0$ such that for all $x\in M_0$ and $R, \rho>0$
\be\label{eq:nobf30ob}
\bE^x\left[ e^{- \frac{c}{\phi(\rho)} \tau^{(\rho)}_{B(x,R)}}\right]
\le C\exp\left(-c_0  R/ \rho \right),
\ee where $c>0$ is the constant in Lemma $\ref{killing-1}$.
 In particular, $(\sE, \sF)$ is conservative.
 \end{lemma}

\begin{proof} We only need to consider the case that $\rho\in(0,R/2)$, since the conclusion holds trivially when $\rho\ge R/2$. For simplicity, we drop the superscript $\rho$ from $\tau^{(\rho)}$. For any $z\in M_0$ and $R>0$, set $\tau=\tau_{B(z,R)}$. For any fixed $0<r<\frac{R}{2}$, set $n=\big[ \frac{R}{2r}\big]$. Let
$  u(x)=\bE^x [ e^{-\lambda \tau}] $ for $x\in M_0$ and $\lam>0$,
and
$m_k=\|u\|_{L^\infty(\overline{B(z,kr)}; \mu)}$, $k=1,2,\cdots,n$.
For any $0<\varepsilon'<\varepsilon$ where $\varepsilon$ is the constant for Lemma \ref{killing-1}, we can choose  $x_k\in \overline{B(z,kr)}\cap M_0$ such that
$(1-\varepsilon')m_k\le u(x_k)\le m_k$. For any $k\le n-1$, $B(x_k,r)\subset B(z,(k+1)r)\subset B(z,R)$.

Next, we consider the following function in $B(x_k,r)\cap M_0$:
$$
v_k(x)=\bE^x [e^{-\lambda \tau_k}],
$$
 where $\tau_k=\tau_{B(x_k,r)}$.
Recall that $\{X^{(\rho)}_t\}$ is the Hunt process
associated with the semigroup $\{Q_t^{(\rho)}\}$. By the strong
Markov property, for any $x\in B(x_k,r)\cap M_0$,
\begin{align*}
u(x) &=\bE^x [e^{-\lambda \tau}] =\bE^x \left[ e^{-\lambda \tau_k}e^{-\lambda(\tau-\tau_k)}\right]\\
&=\bE^x \left[e^{-\lambda \tau_k}\bE^{X^{(\rho)}_{\tau_k}}(e^{-\lambda\tau})\right]
=\bE^x \left[e^{-\lambda\tau_k}u(X^{(\rho)}_{\tau_k}) \right]\\
&\le \bE^x \left[ e^{-\lambda \tau_k} \right] \|u\|_{L^\infty(\overline{B(x_k,r+\rho)}; \mu)}= v_k(x) \|u\|_{L^\infty(\overline{B(x_k,r+\rho)}; \mu)},
\end{align*}
where we have used the fact that $X^{(\rho)}_{\tau_k}\in \overline{B (x_k,r+\rho)}$ in the inequality above.
It follows that for any $0<\rho\le r$,
$$u(x_k)\le v_k(x_k)\|u\|_{L^\infty(\overline{B(x_k,r+\rho)}; \mu)}\le v_k(x_k) m_{k+2},$$ hence
$$(1-\varepsilon')m_k\le  v_k(x_k) m_{k+2}.$$

According to Lemma \ref{killing-1}, if $\lambda\ge \frac{c}{\phi(\rho)}$ and $0<\rho\le r$ (here $c$ is the constant in Lemma \ref{killing-1}), then
$$
(1-\varepsilon')m_k\le (1- \varepsilon)m_{k+2},
$$
 whence it follows by iteration that
$$
u(z)\le m_1\le \left(\frac{1-\epsilon}{1-\varepsilon'} \right)^{n-1} m_{2n-1}\le C\exp\left( -c_0\frac{R}{r} \right),
$$
 where in the last inequality we have used that $n> \frac{R}{2r} -1$, $m_{2n-1}\le 1$ and $c_0:=\frac{1}{2}\log \frac{1-\varepsilon'}{1-\varepsilon}.$ This completes the proof of \eqref{eq:nobf30ob}.

To see that this implies
that $(\sE, \sF)$ is conservative,
take $R\to\infty$ in
\eqref{eq:nobf30ob}. Then one has $\bE^x\left(e^{-
\frac{c}{\phi(\rho)}\zeta^{(\rho)}}\right)=0$ for all $x\in M_0$, where
$\zeta^{(\rho)}$ is the lifetime of  $\{X^{(\rho)}_t\}$. So
we conclude $\zeta^{(\rho)}=\infty$ a.s. This together with Lemma
\ref{intelem} implies that $(\sE, \sF)$ is conservative.
Indeed, the process $\{X_t\}$ can be obtained from $\{X^{(\rho)}_t\}$ through
Meyer's construction as discussed in Section \ref{Sect9.2},
 and therefore the conservativeness of $(\sE, \sF)$ follows immediately from that of $(\sE^{(\rho)}, \sF)$ corresponding to the process $\{X^{(\rho)}_t\}$.
\qed
\end{proof}

Since $\J_{\phi,\ge}$ implies $\FK(\phi)$ under an additional assumption $\RVD$
(see Subsection \ref{jFKnow}) and
$\FK(\phi)+\J_{\phi,\le}+\CSJ(\phi)$ imply
$\E_\phi$ (see Subsection \ref{s:hku}), together with the above lemma, we see that each of
Theorem \ref{T:main} (2), (3), (4) and
Theorem \ref{T:main-1} (2), (3), (4) implies
the conservativeness of $(\sE, \sF)$.\\

As a direct consequence of Lemma \ref{killing-2}, we have the following corollary.

\begin{corollary}\label{killing-3}  Suppose that  $\VD$, \eqref{polycon}, $\E_\phi$ and $\J_{\phi,\le}$ hold.  Then there exist constants $C, c_1,c_2>0$ such that for any $R,\rho>0$ and for all $x\in M_0$,
\begin{equation}\label{esexittime}\bP^x(\tau^{(\rho)}_{B(x,R)}\le t)\le C\exp\left(-c_1 \frac{R}{\rho}+c_2\frac{t}{\phi(\rho)} \right) .\end{equation}
In particular, for any $\varepsilon>0$, there is a constant
$c_0>0$
such that for any ball $B=B(x,R)$ with $x\in M_0$ and $R>0$, and any $\rho>0$ with $\phi(\rho)\ge t$ and $R\ge c_0 \rho$,
$$\bP^z(\tau^{(\rho)}_B\le t)\le \varepsilon\quad\textrm{  for all }z\in B(x,R/2)\cap M_0.$$
\end{corollary}
\begin{proof} Denote by $B=B(x,R)$ for $x\in M$ and $R>0$. Using Lemma \ref{killing-2}, we obtain that, for any $t,\rho>0$ and all $x\in M_0$,
\begin{align*}
\bP^x(\tau^{(\rho)}_B\le t)=&\bP^x(e^{-\frac{c}{\phi(\rho)}\tau^{(\rho)}_B}\ge e^{-c\frac{t}{\phi(\rho)}})
\le   e^{c\frac{t}{\phi(\rho)}} \bE^x(e^{-\frac{c}{\phi(\rho)}\tau^{(\rho)}_B})\\
\le & C\exp\left( -c_1\frac{R}{\rho}+ c\frac{t}{\phi(\rho)}\right).
\end{align*} This proves the first assertion. The second assertion immediately follows from the first one and the fact that $\bP^z(\tau^{(\rho)}_B\le t) \le \bP^z(\tau^{(\rho)}_{B(z,R/2)}\le t)$ for all $z\in B(x,R/2)\cap M_0.$
\qed \end{proof}

Given the above control of the exit time, we now aim to prove $\UHKD(\phi)$.
As the first step, we obtain the on-diagonal upper bound for the heat kernel of $\{Q_t^{(\rho)}\}$.
 The proof is a non-trivial modification of \cite[Lemma 5.6]{GH}. For any open subset ${D}$ of $M$ and any $\rho>0$,
 we define ${D}_\rho=\{x\in M: d(x,{D})<\rho\}.$ Recall that, for $B=B(x_0, r)$ and $a>0$, we use $a B$ to denote
the ball $B(x_0, ar)$.

\begin{proposition}\label{trun-ukdd}  Suppose that $\VD$, \eqref{polycon}, $\FK(\phi)$, $\E_\phi$ and $\J_{\phi, \le}$ hold.  Then  the semigroup $\{Q_t^{(\rho)}\}$ possesses the heat kernel $q^{(\rho)}(t,x,y)$, and there is a constant $C>0$ such that for any $x\in M$ and $\rho, t>0$ with $\phi(\rho)\ge t$,
\be\label{eq:trunwono}
\esssup_{x',y'\in B(x,\rho)}
q^{(\rho)}(t,x',y')\le \frac{C}{V(x,\rho)}\left(\frac{\phi(\rho)}{t}\right)^{1/\nu}.
\ee
\end{proposition}

\begin{proof} Fix $x_0\in M$. For any $t>0$, $R>r+\rho$ and $r\ge \rho$, set $U=B(x_0,r)$ and ${D}=B(x_0,R)$.
Then $\frac{1}{4}U_\rho\subset \frac{1}{2}U$. By Corollary \ref{killing-3}, for any $\varepsilon\in(0,1)$ (which is assumed to be chosen small enough),
 there is  a constant $c_0:=c_0(\eps)>1$ large enough such that for all $\phi(\rho)\ge t$ and $r\ge (c_0-1) \rho$,
\begin{align*}
\esssup_{x\in\frac{1}{4}U_\rho} (1- Q_t^{(\rho),U}{\bf 1}_U(x))\le &\esssup_{x\in\frac{1}{2}U} (1- Q_t^{(\rho),U}{\bf 1}_U(x))\\
 = &\esssup_{x\in\frac{1}{2}U}
\bP^x(\tau^{(\rho)}_U\le t)\le \varepsilon.
\end{align*}
 Then  by \eqref{uvcom} in Lemma \ref{L:com-domain} with $V=\frac{1}{4}U_\rho$, we have for any $t,s>0$, $\phi(\rho)\ge t$ and $r\ge (c_0-1) \rho$,\begin{align*} \esssup_{x,y\in \frac{1}{4}U_\rho}q^{(\rho),{D}}(t+s,x,y)&\le \esssup_{x,y\in U} q^{(\rho),U}(t,x,y) +\varepsilon\, \esssup_{x,y\in U_\rho}q^{(\rho),{D}}(s,x,y)\\
&\le \esssup_{x,y\in U_\rho} q^{(\rho),U_\rho}(t,x,y) +\varepsilon \,\esssup_{x,y\in U_\rho}q^{(\rho),{D}}(s,x,y)
.\end{align*} Furthermore,  due to Lemma \ref{odupper}, there exist constants $c_1$, $\nu>0$ (independent of $c_0$) such that for any $r,\rho, t>0$ with $\phi(\rho)\ge t$ and $r\ge (c_0-1) \rho$,
$$ \esssup_{x,y\in U_\rho} q^{(\rho),U_\rho}(t,x,y)\le \frac{c_1}{V(x_0,r)} \left(\frac{\phi(r)}{t}\right)^{1/\nu
}:=Q_t(r) .$$ According to both inequalities above, we obtain that for any $t,s>0$,  $R>r+\rho$, $\phi(\rho)\ge t$ and $r\ge (c_0-1) \rho$,
\begin{equation}\label{key-1} \esssup_{x,y\in \frac{1}{4}U_\rho}q^{(\rho),{D}}(t+s,x,y)\le  Q_t(r) +\varepsilon \,\esssup_{x,y\in U_\rho}q^{(\rho),{D}}(s,x,y).  \end{equation}

Now, for fixed $t>0$, let $\phi(\rho)\ge t$ and  $$t_k=\frac{1}{2}(1+2^{-k})t, \quad r_k=4^kc_0\rho-\rho,\quad B_k=B(x_0,r_k+\rho)$$ for $k\ge0$. In particular, $t_0=t$, $r_0=(c_0-1)\rho$ and $B_0=B(x_0,c_0\rho)$.

Applying \eqref{key-1} with $r=r_{k+1}$, $s=t_{k+1}$ and $t+s=t_k$ yielding that  \begin{equation}\label{key-2} \esssup_{x,y\in B_k}q^{(\rho),{D}}(t_k,x,y)\le  Q_{2^{-(k+2)}t}(r_{k+1}) +\varepsilon \,\esssup_{x,y\in B_{k+1}}q^{(\rho),{D}}(t_{k+1},x,y), \end{equation} where we have used the facts that $\phi(\rho)\ge t\ge t_k$ and $r_k\ge (c_0-1) \rho$ for all $k\ge 0$. Note that, by \eqref{polycon},
\begin{align*} Q_{2^{-(k+2)}t}(r_{k+1})&= \frac{c_1}{V(x_0,r_{k+1})}\left( \frac{\phi(r_{k+1})}{2^{-(k+2)}t}\right)^{1/\nu}\\
&\le  \frac{c_1}{V(x_0,r_{k})}\left( \frac{\phi(r_{k})}{2^{-(k+1)}t}\right)^{1/\nu} 2^{1/\nu} c'\left(\frac{r_{k+1}}{r_k}\right)^{\beta_2/\nu}\\
&\le L  Q_{2^{-(k+1)}t}(r_{k}),
\end{align*} where $L$ is a constant independent of $c_0$ and $x_0$. Without loss of generality, we may and do assume that $\eps$ is small enough and  $L\ge 2^{1/\nu}$ such that $\varepsilon L\le \frac{1}{2}.$ By this inequality, we can get that
$$ Q_{2^{-(k+2)}t}(r_{k+1})\le L  Q_{2^{-(k+1)}t}(r_{k})\le L^2 Q_{2^{-k}t}(r_{k-1})\le \cdots\le L^{k+1} Q_{t/2}(r_0).$$ Hence, it follows from \eqref{key-2} that
$$\esssup_{x,y\in B_k}q^{(\rho),{D}}(t_k,x,y)\le L^{k+1} Q_{t/2}(r_0)+ \varepsilon\,\esssup_{x,y\in B_{k+1}}q^{(\rho),{D}}(t_{k+1},x,y),$$ which gives by iteration that for any positive integer $n$,
\begin{equation}\begin{split}\label{key-4} \esssup_{x,y\in B_0}q^{(\rho),{D}}(t_0,x,y)\le & L(1+L\eps+(L\eps)^2+\cdots) Q_{t/2}(r_0)\\
&+ \eps^n\,\esssup_{x,y\in B_n}q^{(\rho),{D}}(t_n,x,y)\\
\le& 2L Q_{t/2}(r_0)+\varepsilon^n\, \esssup_{x,y\in B_n}q^{(\rho),{D}}(t_n,x,y),\end{split} \end{equation} as long as $B_n\subset {D}$.

 By Lemma \ref{odupper}, $\VD$ and \eqref{polycon},
 there exists a constant $L_1>0$ (also independent of $c_0$) such that
$$\esssup_{x,y\in B_n}q^{(\rho),B_n}(t_n,x,y)\le c''Q_{t_n}(r_n)\le c'''L_1^n Q_{t}(r_0) .$$ Again, without loss of generality, we may and do assume that $L_1\le L$ and so $0<\varepsilon L_1\le \frac{1}{2}$; otherwise, we replace $L$ with $L+L_1$ below. In particular,
$$\lim_{n\to\infty} \varepsilon^n\, \esssup_{x,y\in B_n}q^{(\rho),B_n}(t_n,x,y)\le c'''  Q_{t}(r_0)\lim_{n\to \infty}(\eps L_1)^n=0.$$ Putting both estimates above into \eqref{key-4} with ${D}=B_n$, we find that
\begin{equation}\label{key-3}\limsup_{n\to\infty}\esssup_{x,y\in B_0}q^{(\rho),B_n}(t,x,y)\le  2L Q_{t/2}((c_0-1)\rho).\end{equation}

 Having \eqref{key-3} at hand, we can follow the argument of
 \cite[Lemma 5.6]{GH} to complete the proof, see \cite[p.\ 540]{GH}.  Indeed,
 the sequence $\{q^{(\rho),B_n}(t,\cdot,\cdot)\}$ increases as $n\to\infty$ and converges almost everywhere on $M\times M$ to a non-negative measurable function $q^{(\rho)}(t,\cdot,\cdot)$;
  see \cite[Theorem 2.12 (b) and (c)]{GT}.
 The function $q^{(\rho)}(t,\cdot,\cdot)$ is finite almost everywhere
since $$\int_{B_n}q^{(\rho),B_n}(t,x,y)\,\mu(dy)\le 1.$$
For any non-negative function $f\in L^2(M; \mu)$, we have by the monotone convergence theorem,
 $$
 \lim_{n\to\infty} \int_{B_n} q^{(\rho),B_n}(t,x,y)f(y)\,\mu(dy)=\int q^{(\rho)}(t,x,y)f(y)\,\mu(dy).
 $$
On the other hand,
$$
\lim_{n\to\infty}\int_{B_n}  q^{(\rho),B_n}(t,x,y)f(y)\,\mu(dy)=\lim_{n\to\infty}Q^{(\rho), B_n}_tf(x)=Q^{(\rho)}_tf(x),
$$
 see \cite[Theorem 2.12(c)]{GT}
 again. Hence, $q^{(\rho)}(t,x,y)$ is the heat kernel of $\{Q^{(\rho)}_t\}$.
 Thus
 it follows from \eqref{key-3} that there exists a constant $C>0$ (independent of $\rho$) such that
\eqref{eq:trunwono} holds for all $x_0\in M$ and $\rho, t>0$ with $\phi(\rho)\ge t$.
\qed\end{proof}

For any $\rho>0$ and $x,y\in M$, set
$$
 J_\rho (x,y):=J(x,y){\bf1}_{\{d(x,y)> \rho\}}.
$$
 Using the Meyer's decomposition and Lemma \ref{meyer}(1),
we have the following estimate
\be\label{eq:MeyerHK}
p(t, x,y)\le q^{(\rho)}(t,x,y)+\bE^x\Big[\int_0^t\int_M J_\rho(Y_s,z)p_{t-s}(z,y)\,\mu(dz)\,ds\Big], \quad x,y\in M_0.
\ee

The following is a key proposition.

\begin{proposition}\label{thm:VD-Meyer}
Suppose that  $\VD$, \eqref{polycon}, $\E_\phi$ and $\J_{\phi,\le}$ hold.  Then  there exists a constant
$c_1>0$ such that the following estimate holds for all $t,\rho>0$ and all $x\in M_0$,
\[
\bE^x\ \left[ \int_0^t\int_M
J_\rho(X^{(\rho)}_s,z)p(t-s, z,y)\,\mu(dz)\right] \le
\frac{c_1t}{V(x,\rho)\phi(\rho)}
\exp\Big(c_1\frac{t}{\phi(\rho)}\Big).
\]
\end{proposition}
\begin{proof} By $\J_{\phi,\le}$,
$J_\rho(x,y)\le \frac{c_1}{V(x,\rho) \phi(\rho)}$  for all $x,y\in M$.
By the fact that $p(t, z,y)=p(t, y,z)$, for all $x\in M_0$,
\begin{align*}
&\bE^x \left[ \int_0^t\int_M  J_\rho(X^{(\rho)}_s,z)p(t-s, z,y)\,\mu(dz)\right] \\
&\le c_1 \bE^x\left[ \int_0^t \frac{1}{V(X^{(\rho)}_s,\rho)\phi(\rho)}\,ds\right] \\
&=c_1\sum_{k=1}^\infty \bE^x \left[\int_0^t \frac{1}{V(X^{(\rho)}_s,\rho)\phi(\rho)}\,ds; \tau_{B(x,k\rho)}^{(\rho)}\ge t>\tau_{B(x,(k-1)\rho)}^{(\rho)} \right] \\
&=: c_1\sum_{k=1}^\infty I_k.
\end{align*}

If $t\le \tau_{B(x,k\rho)}^{(\rho)}$, then $d(X^{(\rho)}_s, x)\le
k\rho$ for all $s\le t$. This along with $\VD$ yields that for all
$k\ge 1$,
\begin{align*}\frac{1}{V(X^{(\rho)}_s,\rho)\phi(\rho)}\le &\frac{c_2 k^{d_2}}{V(X^{(\rho)}_s, 2k\rho)\phi(\rho)}\le \frac{c_2 k^{d_2}}{\inf_{d(z,x)\le k\rho}V(z, 2k\rho)\phi(\rho)}\\
\le &\frac{c_2 k^{d_2}}{V(x, k\rho)\phi(\rho)}\le \frac{c_2 k^{d_2}}{V(x, \rho)\phi(\rho)}.\end{align*} In particular, we have
 $$I_1\le  \frac{c_2 t }{V(x, \rho)\phi(\rho)}.$$ Thus, by Corollary \ref{killing-3}, for all $k\ge 2$,
\begin{align*}I_k\le & \frac{c_3 t k^{d_2}}{V(x, \rho)\phi(\rho)}\bP^x(\tau_{B(x,(k-1)\rho)}^{(\rho)}<t)\\
\le & \frac{c_4 t }{V(x, \rho)\phi(\rho)} e^{c_5\frac{t}{\phi(\rho)}} k^{d_2} e^{-c_6k}\le \frac{c_4 t }{V(x, \rho)\phi(\rho)} e^{c_5\frac{t}{\phi(\rho)}} e^{-c_7k}.\end{align*}
This yields the desired assertion.
\qed\end{proof}

Given all the above estimates, we can obtain the main theorem in this subsection.

\begin{theorem}\label{ehi-ukdd}  Suppose that  $\VD$, \eqref{polycon}, $\FK(\phi)$, $\E_\phi$ and $\J_{\phi, \le}$ hold. Then $\UHKD(\phi)$ is satisfied, i.e. there is a constant $c>0$ such that for all $x\in M_0$ and $t>0$,
$$p(t, x,x)\le \frac{c}{V(x,\phi^{-1}(t))}.$$  \end{theorem}
\begin{proof} For each $t>0$, set $\rho=\phi^{-1}(t)$.
 Then  by Proposition \ref{trun-ukdd}, for all $x\in M_0$,
\[
q^{(\rho)}(t,x,x)\le \frac{c_1}{V(x,\phi^{-1}(t))}.
\]
Using this, \eqref{eq:MeyerHK} and Proposition \ref{thm:VD-Meyer}, for all $x\in M_0$,
we have
\[
p(t, x,x)\le q^{(\rho)}(t,x,x)+\frac{c_2t}{V(x,\rho)\phi(\rho)}
\exp\Big(c_2\frac{t}{\phi(\rho)}\Big)\le \frac{c_3}{V(x,\phi^{-1}(t))},
\]
thanks to $\phi(\rho)=t$, $\VD$ and \eqref{polycon}.
\qed\end{proof}

 \section{Consequences of condition $\J_{\phi}$ and mean exit time condition $\E_{\phi}$}\label{section5}

In this section, we will first prove (2) $\Longrightarrow$ (1) in Theorem \ref{T:main-1} and then prove (2) $\Longrightarrow$ (1) in Theorem \ref{T:main}. Without any mention, throughout the proof we will assume that $\mu$ and $\phi$ satisfy $\VD$, $\RVD$  and \eqref{polycon} respectively. (Indeed, $\RVD$ is only used in the proof of $\J_{\phi,\ge}\Longrightarrow \FK(\phi)$.)
We note that (2) implies the conservativeness of $(\sE, \sF)$
due to Lemma \ref{killing-2}.

Recall again that, for any $\rho>0$,
$(\sE^{(\rho)}, \sF)$ defined in \eqref{eq:rhoEdef} denotes the $\rho$-truncated Dirichlet form obtained by $\rho$-truncation for the jump density of the original Dirichlet form $(\sE, \sF)$.
Let $\{X_t^{(\rho)}\}$ be the Hunt process associated with the
$\rho$-truncated Dirichlet form $(\sE^{(\rho)}, \sF)$. For any open subset
$D\subset M$, let $\tau_D^{(\rho)}$ be the first exit time of the
process $\{X_t^{(\rho)}\}$. For any open subset $D\subset M$ and
$\rho>0$, set $D_\rho=\{x\in M: d(x,D)<\rho\}$.

 \subsection{$\UHKD(\phi)+\J_{\phi,\le}+\E_{\phi} \Longrightarrow\UHK(\phi)$,
 $\J_\phi+ \E_\phi \Longrightarrow \UHK(\phi)$}\label{5-15-1}

We begin with the following improved statement for $\UHKD(\phi)$.
\begin{lemma}\label{P:ondiag} Under $\VD$ and \eqref{polycon}, if
$\UHKD(\phi)$, $\J_{\phi,\le}$ and $\E_{\phi}$ hold, then there is a constant $c>0$ such that for any $t> 0$ and all $x,y\in M_0$,
$$p(t, x,y)\le c\left(\frac{1}{V(x,\phi^{-1}(t))}\wedge \frac{1}{V(y,\phi^{-1}(t))}\right).$$
\end{lemma}
\begin{proof}First, using the first conclusion in Lemma \ref{meyer}(2), Lemma \ref{intelem} and $\UHKD(\phi)$, we can easily see that the
$\rho$-truncated Dirichlet form $(\sE^{(\rho)}, \sF)$ has the heat kernel $q^{(\rho)}(t,x,y)$, and
$$q^{(\rho)}(t,x,x)\le
p(t, x,x)\exp\Big(c_1\frac{t}{\phi(\rho)}\Big)\le
\frac{c_2}{V(x,\phi^{-1}(t))}\exp\Big(c_1\frac{t}{\phi(\rho)}\Big),
$$
for all $t>0$ and all $x\in M_0$, where $c_1, c_2>0$ are independent of $\rho$.  Then  by the symmetry of $q^{(\rho)}(t,x,y)$ and the Cauchy-Schwarz inequality, for all $t>0$ and all $x,y\in M_0$,
\begin{equation}\label{remark-u1} q^{(\rho)}(t,x,y)\le \sqrt{q^{(\rho)}(t,x,x)q^{(\rho)}(t,y,y)}\le \frac{c_2}{\sqrt{V(x,\phi^{-1}(t))V(y,\phi^{-1}(t))}}\exp\Big(c_1\frac{t}{\phi(\rho)}\Big).\end{equation}

Second, let $U$ and $V$ be two open subsets of $M$ such that $U_\rho$ and $V_\rho$ are precompact, and $U\cap V=\emptyset$. According to Lemma \ref{hk-exittime}, for any $t>0$ and all $x\in U\cap M_0$ and $y\in V\cap M_0$,
\begin{align*}
q^{(\rho)}(2t,x,y)\le& \bP^x(\tau^{(\rho)}_U\le t) \esssup_{t\le t'\le 2t}\|q^{(\rho)}(t',\cdot,y)\|_{L^\infty(U_\rho,\mu)}\\
&+\bP^y(\tau^{(\rho)}_V\le t) \esssup_{t\le t'\le 2t}\|q^{(\rho)}(t',\cdot,x)\|_{L^\infty(V_\rho;\mu)} \\
\le& \left(\bP^x(\tau^{(\rho)}_U\le t)+\bP^y(\tau^{(\rho)}_V\le t)\right) \esssup_{x'\in U_\rho, y'\in V_\rho, t\le t'\le 2t}q^{(\rho)} (t',x',y').  \end{align*}
 Then  taking $U=B(x,r)$ and  $V=B(y,r)$ with $r=\frac{1}{4}d(x,y)$ in the inequality above, and using Corollary \ref{killing-3} and \eqref{remark-u1}, we find that for any $t, \rho>0$ and all $x,y\in M_0$,
\begin{align*} q^{(\rho)}(2t,x,y) \le &c_3 \exp\Big(-c_4\frac{r}{\rho} +c_5\frac{ t}{\phi(\rho)}\Big) \esssup_{x'\in B(x,r+\rho), y'\in B(y,r+\rho)}\frac{1}{\sqrt{V(x',\phi^{-1}(t))V(y',\phi^{-1}(t))}}\\
\le & \frac{c_6}{V(x,\phi^{-1}(t))}\left(1+ \frac{r+\rho}{\phi^{-1}(t)}\right)^{d_2} \exp\Big(-c_4\frac{r}{\rho} +c_5\frac{ t}{\phi(\rho)}\Big).\end{align*}
This along with \eqref{eq:MeyerHK} and Proposition \ref{thm:VD-Meyer} yields that  for any $t, \rho>0$ and all $x,y\in M_0$,
$$
p(2t, x,y)\le  c_7\left[ \frac{1}{V(x,\phi^{-1}(t))}\left(1+ \frac{r+\rho}{\phi^{-1}(t)}\right)^{d_2} \exp\Big(-c_4\frac{r}{\rho} \Big)+   \frac{t}{V(x,\rho) \phi(\rho)}\right] \exp\Big(c_8\frac{t}{\phi(\rho)}\Big).
$$
Taking $\rho=\phi^{-1}(t)$ and using the fact that the function $f(r)=(2+r)^{d_2} e^{-c_4r}$ is bounded on $[0,\infty)$, we furthermore get that for all $t>0$ and all $x,y\in M_0$,
$$ p(2t, x,y)\le  \frac{c_9}{V(x,\phi^{-1}(t))},$$ which in turn gives us the desired assertion by the symmetry of $p(t, x,y)$, $\VD$ and \eqref{polycon}.   \qed
\end{proof}

\begin{lemma}\label{P:truncated} Under $\VD$ and \eqref{polycon}, if
$\UHKD(\phi)$, $\J_{\phi,\le}$ and $\E_{\phi}$ hold, then the $\rho$-truncated Dirichlet form $(\sE^{(\rho)}, \sF)$ has the heat kernel $q^{(\rho)}(t,x,y)$, and it satisfies that for any $t> 0$ and all $x,y\in M_0$,
$$q^{(\rho)}(t,x,y)\le c_1\bigg(\frac{1}{V(x,\phi^{-1}(t))}+ \frac{1}{V(y,\phi^{-1}(t))}\bigg) \exp\Big(c_2\frac{t}{\phi(\rho)}-c_3\frac{d(x,y)}{\rho}\Big),$$ where $c_1,c_2,c_3$ are positive constants independent of $\rho$.

Consequently, for any $t> 0$ and all $x,y\in M_0$,
$$q^{(\rho)}(t,x,y)\le \frac{c_4}{V(x,\phi^{-1}(t))}\bigg(1+ \frac{d(x,y)}{\phi^{-1}(t)}\bigg)^{d_2} \exp\Big(c_2\frac{t}{\phi(\rho)}-c_3\frac{d(x,y)}{\rho}\Big).$$
\end{lemma}

\begin{proof} (i) The existence of $q^{(\rho)}(t,x,y)$ has been mentioned in the proof of Lemma \ref{P:ondiag}. Furthermore, according to Lemma \ref{meyer}(2), Lemma \ref{intelem} and Lemma \ref{P:ondiag}, there exist $c_1,c_2>0$ such that for all $t>0$ and all $x,y\in M_0$,
\begin{equation}\label{truncated:diag} q^{(\rho)}(t,x,y)\le c_1\left(\frac{1}{V(x,\phi^{-1}(t))}\wedge \frac{1}{V(y,\phi^{-1}(t))}\right)\exp\Big(c_2\frac{t}{\phi(\rho)}\Big).\end{equation}
Therefore, in order to prove the desired assertion, below we only need to consider the case that $d(x,y)\ge 2\rho$.

By Corollary \ref{killing-3}, for any ball $B(x,r)$, $t>0$ and all $z\in B(x,\rho)\cap M_0$ with $r> \rho$,
\begin{equation}\label{truncated:diag2}\begin{split}Q_t^{(\rho)}{\bf 1}_{B(x,r)^c}(z)\le &\bP^z(\tau^{(\rho)}_{B(x,r)}\le t)\le  \bP^z(\tau^{(\rho)}_{B(z,r-\rho)}\le t)\\
 \le& c_3\exp\left(-c_4\frac{r}{\rho}+c_3\frac{t}{\phi(\rho)}\right),\end{split}\end{equation} where $c_3,c_4>0$ are independent of $\rho$.

(ii) Fix $x_0,y_0\in M$ and $t>0$. Set $r= \frac{1}{2}d(x_0,y_0)$. By the semigroup property, we have that
\begin{align*} q^{(\rho)}(2t,x,y)=&\int_M q^{(\rho)}(t,x,z)q^{(\rho)}(t,z,y)\,\mu(dz)\\
\le&\int_{B(x_0,r)^c}q^{(\rho)}(t,x,z)q^{(\rho)}(t,z,y)\,\mu(dz)+ \int_{B(y_0,r)^c}q^{(\rho)}(t,x,z)q^{(\rho)}(t,z,y)\,\mu(dz).\end{align*}
Using \eqref{truncated:diag} and \eqref{truncated:diag2}, we obtain that
\begin{align*} \int_{B(x_0,r)^c}q^{(\rho)}(t,x,z)q^{(\rho)}(t,z,y)\,\mu(dz)&\le \frac{c_1}{V(y,\phi^{-1}(t))}\exp\Big(c_2\frac{t}{\phi(\rho)}\Big)\int_{B(x_0,r)^c} q^{(\rho)}(t,x,z)\,\mu(dz)\\
&\le\frac{c_5}{V(y,\phi^{-1}(t))}\exp\Big(c_5\frac{t}{\phi(\rho)}-c_4\frac{r}{\rho}\Big) \end{align*} for $\mu$-almost all $x\in B(x_0,\rho)$ and $y\in M$. Similarly,
by the symmetry of $q^{(\rho)}(t,z,y)$,
\begin{align*} \int_{B(y_0,r)^c}q^{(\rho)}(t,x,z)q^{(\rho)}(t,z,y)\,\mu(dz)&\le \frac{c_1}{V(x,\phi^{-1}(t))}\exp\Big(c_2\frac{t}{\phi(\rho)}\Big)\int_{B(y_0,r)^c} q^{(\rho)}(t,z,y)\,\mu(dz)\\
&= \frac{c_1}{V(x,\phi^{-1}(t))}\exp\Big(c_2\frac{t}{\phi(\rho)}\Big)\int_{B(y_0,r)^c} q^{(\rho)}(t,y,z)\,\mu(dz)\\
&\le\frac{c_6}{V(x,\phi^{-1}(t))}\exp\Big(c_6\frac{t}{\phi(\rho)}-c_4\frac{r}{\rho}\Big) \end{align*} for $\mu$-almost all $y\in B(x_0,\rho)$ and $x\in M$. Hence, since $x_0$ and $y_0$ are arbitrary, we get the first required assertion by $\VD$ and \eqref{polycon}.  Then  the second one immediately follows from the first one and $\VD$. \qed
\end{proof}

Now, we can prove the following main result.

\begin{proposition}\label{thm:ujeuhkds}
Under $\VD$ and \eqref{polycon}, if $\UHKD(\phi)$, $\J_{\phi,\le}$ and $\E_{\phi}$ hold, then we have  $\UHK(\phi)$.
\end{proposition}

\begin{proof} (i) We first prove that there are $N\in \bN$ with $N>(\beta_1+d_2)/\beta_1$  and $C_0\ge 1$ such that for each $t,r>0$ and all $x\in M_0$,
\begin{equation}\label{tail}
\int_{B(x,r)^c} p(t, x,y)\,\mu(dy)\le C_0\bigg(\frac{\phi^{-1}(t)}{r}\bigg)^\theta,
\end{equation} where $\theta=\beta_1-(\beta_1+d_2)/N$, and  $d_2$ and $\beta_1$ are constants from $\VD$ and \eqref{polycon} respectively.
Indeed, we only need to consider the case that $r>\phi^{-1}(t)$. For any $\rho, t>0$ and all $x,y\in M_0$, by
\eqref{eq:MeyerHK} and Proposition \ref{thm:VD-Meyer}, we have
\[
p(t, x,y)\le q^{(\rho)}(t,x,y)+\frac{c_1t}{V(x,\rho)\phi(\rho)}\exp\left(\frac{c_2t}{\phi(\rho)}\right),
\] where $c_1,c_2>0$ are constants independent of $\rho$.
Now, for fixed large $N\in \bN$  (which will be specified later), define
\[
\rho_n=2^{n\alpha}r^{1-1/N}\phi^{-1}(t)^{1/N},\quad n\in \bN,
\]
where $\alpha\in (d_2/(d_2+\beta_1)\vee 1/2,1)$. Since $r>\phi^{-1}(t)$ and $2\alpha\ge 1$, we have
\begin{equation}\label{booefb}
\phi^{-1}(t)\le \rho_n\le 2^nr,\quad\frac{2^nr}{\rho_n}\le \frac{\rho_n}{\phi^{-1}(t)}.
\end{equation}
In particular, $t/\phi(\rho_n)\le 1$. Plugging these into Lemma \ref{P:truncated}, we have that there are constants $c_3,c_4>0$ such that for every $t>0$ and all $x,y\in M_0$ with $2^nr\le d(x,y)\le 2^{n+1} r$,
$$q^{(\rho_n)}(t,x,y)\le \frac{c_3}{V(x,\phi^{-1}(t))}\left( \frac{2^nr}{\phi^{-1}(t)}\right)^{d_2}\exp\left(-\frac{c_4 2^nr}{\rho_n}\right).$$
Thus, there are constants $c_5, c_6>0$ such that for every $t>0$ and all $x\in M_0$,
\begin{align*}
\int_{B(x,r)^c}p(t, x,y)\,\mu(dy)
=&\sum_{n=0}^\infty\int_{B(x,2^{n+1}r)\setminus B(x,2^nr)} p(t, x,y)\,\mu(dy)\\
\le &\sum_{n=0}^\infty \frac{c_5}{V(x,\phi^{-1}(t))}\left( \frac{2^nr}{\phi^{-1}(t)}\right)^{d_2}\exp\left(-\frac{c_4 2^nr}{\rho_n}\right)V(x,2^nr)
\\
&+\sum_{n=0}^\infty\frac{c_6tV(x,2^nr)}{V(x,\rho_n)\phi(\rho_n)}\\
=&:I_1+I_2.
\end{align*}
We first estimate $I_2$. Take $N$ large enough so that $\beta_1-(\beta_1+d_2)/N>0$.  Then
using $\VD$, \eqref{polycon} and \eqref{booefb}, we have
\begin{align*}
I_2\le&  c_7\sum_{n=0}^\infty\Big(\frac{\phi^{-1}(t)}{\rho_n}\Big)^{\beta_1}\Big(\frac{2^nr}{\rho_n}\Big)^{d_2}\\
=&c_7\Big(\frac{\phi^{-1}(t)}{r}\Big)^{\beta_1-(\beta_1+d_2)/N}\sum_{n=0}^\infty2^{n(d_2-\alpha(d_2+\beta_1))}\\
\le & c_8\Big(\frac{\phi^{-1}(t)}{r}\Big)^{\beta_1-(\beta_1+d_2)/N},
\end{align*}
where in the last inequality we used the fact $d_2-\alpha(d_2+\beta_1)<0$ due to the choice of $\alpha$.
We next estimate $I_1$. Note that for each $K\in \bN$, there exists a constant $c_K>0$ such that
$e^{-x}\le c_Kx^{-K}$ for all $x\ge 1$. Now choose $K$ large enough so that $K/N>2d_2+\beta_1$ and $(1-\alpha)K>2d_2$.
 Then  using $\VD$, \eqref{polycon} and \eqref{booefb} again, we have
\begin{align*}
I_1\le &\sum_{n=0}^\infty \frac{c_{9,K}}{V(x,\phi^{-1}(t))}\Big(\frac{2^nr}{\phi^{-1}(t)}\Big)^{d_2}\Big(\frac{\rho_n}{2^nr}\Big)^KV(x,2^nr)\\
\le & c_{10,K}\sum_{n=0}^\infty\Big(\frac{2^nr}{\phi^{-1}(t)}\Big)^{2d_2}
\Big(\frac{\phi^{-1}(t)^{1/N}}{2^{n(1-\alpha)}r^{1/N}}\Big)^K\\
=&c_{10,K}\Big(\frac{\phi^{-1}(t)}r\Big)^{K/N-2d_2}\sum_{n=0}^\infty 2^{n(2d_2-(1-\alpha)K)}\\
\le& c_{11,K}\Big(\frac{\phi^{-1}(t)}r\Big)^{K/N-2d_2}\le c_{11,K}\Big(\frac{\phi^{-1}(t)}r\Big)^{\beta_1}.
\end{align*}
Combining with all estimations above, we obtain the desired estimate \eqref{tail}.

(ii) For any ball $B$ with radius $r$, by \eqref{tail}, there is a constant $c_1>0$ such that
\be\label{eq:noohes}
1-P_t^B{\bf 1}_B(x)=\bP^x(\tau_B\le t)\le c_1 \left( \frac{r}{\phi^{-1}(t)}\right)^{-\theta}\textrm{ all }x\in \frac{1}{4}B\cap M_0,
\ee e.g.\ see the proof of Lemma \ref{Conserv}.
(Note that due to Lemma \ref{killing-2},  $(\sE, \sF)$ is conservative.)

Combining \eqref{eq:noohes} with \eqref{eq:compatr}, we find that
\begin{equation}\label{step2}
1-Q_t^{(\rho),B}{\bf 1}_B(x)\le c_2 \left[ \left( \frac{r}{\phi^{-1}(t)}\right)^{-\theta}+ \frac{t}{\phi(\rho)}\right]
\quad\hbox{for  all  } x\in \frac{1}{4}B \cap M_0,
\end{equation}
where $Q_t^{(\rho),B}$ is the semigroup for the
$\rho$-truncated Dirichlet form $(\sE^{(\rho)}, \sF_B)$, and the constant $c_2$ is independent of $\rho.$

Next, we prove the following improvement of estimate in Lemma \ref{P:truncated}: for all $t>0$, $k\ge 1$, and all $x_0,y_0\in M$ with $d(x_0,y_0)>4k\rho$, \begin{equation}\begin{split}\label{niceone} q^{(\rho)}(t,x,y)\le c_3(k)&\left(\frac{1}{V(x,\phi^{-1}(t))}+ \frac{1}{V(y,\phi^{-1}(t))}\right)\exp\Big(c_4\frac{t}{\phi(\rho)}\Big)\left(1+ \frac{\rho}{\phi^{-1}(t)}\right)^{-(k-1)\theta}\end{split}\end{equation} for almost all $x\in B(x_0,\rho)$ and $y\in B(y_0,\rho)$. By \eqref{truncated:diag}, it suffices to consider the case that $\rho\ge \phi^{-1}(t).$ Indeed, fix $k\ge 1$, $t>0$ and $x_0,y_0\in M_0$. Set $r=\frac{1}{2} d(x_0,y_0)>2k\rho$. By \eqref{step2} and Lemma \ref{probability-exittime},
$$
Q^{(\rho)}_t{\bf 1}_{B(x_0,r)^c}(x)\le c_5(k) \left[ \left( \frac{\rho}{\phi^{-1}(t)}\right)^{-\theta}+ \frac{t}{\phi(\rho)}
\right]^{k-1}\quad\hbox{for  almost all } x\in B(x_0,\rho).
$$
It is easy to see that $$\left(\frac{\rho}{\phi^{-1}(t)}\right)^{-\theta}\ge c_3\frac{t}{\phi(\rho)}\quad\textrm{ for all } \rho\ge\phi^{-1}(t),$$ (here $c_3$ is the constant in \eqref{polycon}) and so for almost all $x\in B(x_0,\rho)$,
$$Q^{(\rho)}_t{\bf 1}_{B(x_0,r)^c}(x)\le c_6(k) \left( \frac{\rho}{\phi^{-1}(t)}\right)^{-(k-1)\theta}.$$
  Then  using \eqref{truncated:diag} and the estimate above, we can follow part (ii) in the proof of Lemma \ref{P:truncated} to obtain  \eqref{niceone}.

(iii) Finally we prove the desired upper bound for $p(t, x,y)$. For any fixed $x_0,y_0\in M$, let $r=\frac{1}{2} d(x_0,y_0)$. We only need to show that
$$p(t, x,y)\le \frac{C}{V(x,\phi^{-1}(t))}\bigg(1\wedge \frac{V(x,\phi^{-1}(t))t}{V(x,r)\phi(r)}\bigg)$$ for all $t>0$,
small enough $\rho \in (0, r)$ and almost all $x\in B(y_0, \rho)$ and $y\in B(x_0,\rho)$.
 As before, by Lemma \ref{P:ondiag}, without loss of generality we may and do assume that $r/\phi^{-1}(t)\ge 1$.
Take $k=2+ [(2d_2+\beta_2)/\theta]$
and $\rho=r/(8k)$.
Using \eqref{eq:MeyerHK}, Proposition \ref{thm:VD-Meyer} and \eqref{niceone}, we obtain that for all $t>0$, and almost all $x\in B(x_0,\rho)$ and $y\in B(y_0,\rho)$,
 \begin{align*}
 p(t, x,y) &\le \frac{c_7(k)}{V(x,\phi^{-1}(t))}\left(1+\frac{d(x,y)}{\phi^{-1}(t)} \right)^{d_2}\left( \frac{\rho}{\phi^{-1}(t)}+1\right)^{-(k-1)\theta} +\frac{c_0't}{V(x,\rho)\phi(\rho)}\\
 &\le c_8(k) \left[ \frac{1}{V(x,\phi^{-1}(t))}\left( \frac{r}{\phi^{-1}(t)}\right)^{-(k-1)\theta+d_2}+\frac{t}{V(x,r)\phi(r)}
 \right] \\
 &=   \frac{c_8(k)t}{V(x,r)\phi(r)} \left[ 1+ \frac{V(x,r)}{V(x,\phi^{-1}(t))}\frac{\phi(r)}{t}\left( \frac{r}{\phi^{-1}(t)}\right)^{-(k-1)\theta+d_2}
 \right]\\
 & \le\frac{c_9(k)t}{V(x,r)\phi(r)}\left[ 1+\left( \frac{r}{\phi^{-1}(t)}\right)^{-(k-1)\theta+2d_2+\beta_2}
 \right]\\
 & \le\frac{c_{10}(k)t}{V(x,r)\phi(r)},
 \end{align*} where in the third inequality we used $\VD$ and \eqref{polycon}.
 The proof is complete.
\qed\end{proof}

$\J_{\phi,\ge}\Longrightarrow \FK(\phi)$ has been proved in
Subsection \ref{jFKnow} by the additional assumption $\RVD$, and $\FK(\phi)+\E_\phi+\J_{\phi,\le}\Longrightarrow \UHKD(\phi)$ has been proved in Subsection \ref{FKEPHIj}. Combining these with Proposition \ref{thm:ujeuhkds}, we also obtain
 $\J_\phi+ \E_\phi \Longrightarrow \UHK(\phi)$.

\subsection{$\J_\phi+ \E_\phi \Longrightarrow \LHK(\phi)$}\label{subsection:lower}

\begin{proposition}\label{ejlhk} If $\VD$, \eqref{polycon}, $\E_\phi$ and $\J_{\phi}$ hold, then we have $\LHK(\phi)$.  \end{proposition}
\begin{proof}
 The proof is split into two steps, and the first one is concerned with the near-diagonal lower bound estimate.

(i) The argument for the near-diagonal lower bound estimate is
standard; we present it here for the sake of completeness. It
follows from $\E_\phi$ and Lemma \ref{E} that there exist  constants $c_0\geq 1$ and
$c_1\in (0,1)$ so that for all $x\in M_0$ and
$t,r>0$ with $r\ge c_0 \phi^{-1}(t)$,
$$
\int_{B(x,r)^c} p(t, x,y)\,\mu(dy)\le \bP^x(\tau_{B(x,r)}\le t)\le c_1 .
$$
 This and the conservativeness of $(\sE, \sF)$(which is due to
Lemma \ref{killing-2}) imply that
$$\int_{B(x, c_0\phi^{-1}(t))} p(t, x,y)\,\mu(dy)\ge 1-c_1.$$ By the semigroup property and the Cauchy-Schwarz inequality, we get for all $x\in M_0$
\begin{equation}\label{odle} \begin{split}
p (2t,  x,x)= &\int_Mp(t, x,y)^2\,\mu(dy)\ge \frac{1}{V(x,c_0\phi^{-1}(t))} \bigg(\int_{B(x, c_0\phi^{-1}(t))} p(t, x,y)\,\mu(dy)\bigg)^2\\
 \ge &\frac{c_2}{V(x,\phi^{-1}(t))}.
 \end{split}\end{equation}
  Furthermore, by \eqref{eq:holdboef} below, we can take $\delta>0$ small enough and find that for almost all
  $y\in B(x,\delta\phi^{-1}(t))$,
$$p (2t, x,y)\ge p (2t, x,x)- \frac{c_3}{V(x,\phi^{-1}(t))} \delta^\theta\ge  \frac{c_4}{V(x,\phi^{-1}(t))}.$$ By $\VD$ and \eqref{polycon}, there are constants ${\delta_1},c_5>0$ such that for all $t>0$, almost all $x\in M$ and $y\in B(x, {\delta_1}\phi^{-1}(t))$,
$$p(t, x,y)\ge \frac{c_5}{V(x,\phi^{-1}(t))}.$$

(ii) The argument below is motivated by \cite[Section 4.4]{CZ}. According to the result in Subsection \ref{5-15-1},
Lemma \ref{killing-2} and Lemma \ref{Conserv}, $\UHK(\phi)$ and so $\EP_{\phi,\le}$ holds, i.e.\  for all $x\in M_0$ and $t,r>0$,
$$\bP^x(\tau_{B(x,r)}\le t)\le c_6t/\phi(r).$$
In particular, for any ${\delta_2}\in(0, {\delta_1})$, there is a constant $a\in (0,1/2)$ such that for all $x\in M_0$ and $t>0$,
$$\bP^x( \tau_{B(x,2{\delta_2}\phi^{-1}(t)/3)}\le at )\le c_6at/\phi(2{\delta_2}\phi^{-1}(t)/3 )\le c' a\delta_2^{-\beta_2}\le {1}/{2},$$ where we used \eqref{polycon} in the second inequality.

In the following, we fix $\delta_2\in (0,\delta_1)$. By taking a sufficiently small $a\in(0,1/2)$, below we may and do assume that ${\delta_1}\phi^{-1}((1-a)t)\ge {\delta_2}\phi^{-1}(t)$.
 For $A\subset M$, let $$\sigma_A=\inf\{t>0: X_t\in A\}.$$  Now, by the strong Markov property, for
all $x \in M_0$ and $y\in M$ with $d(x,y)\ge {\delta_1}\phi^{-1}(t),$
\begin{align*} \bP^x(X_{at}\in& B(y, {\delta_1}\phi^{-1}((1-a)t)))\\
\ge &\bP^x(X_{at}\in B(y, {\delta_2}\phi^{-1}(t)))\\ \ge &\bP^x\bigg(\sigma_{B(y,{\delta_2}\phi^{-1}(t)/3)}\le a t; \sup_{s\in[\sigma_{B(y,{\delta_2}\phi^{-1}(t)/3)}, at]}
d(X_s,X_{\sigma_{B(y,{\delta_2}\phi^{-1}(t)/3)}})< 2{\delta_2}\phi^{-1}(t)/3\bigg)\\
\ge& \bP^x(\sigma_{B(y,{\delta_2}\phi^{-1}(t)/3)}\le at)\inf_{z\in B(y, {\delta_2}\phi^{-1}(t)/3)}\bP^z(\tau_{B(z,2{\delta_2}\phi^{-1}(t)/3)}>at)\\
\ge& \frac{1}{2}\bP^x(\sigma_{B(y,{\delta_2}\phi^{-1}(t)/3)}\le at)\\
\ge& \frac{1}{2}\bP^x \Big(X_{(at)\wedge \tau_{B(x,2{\delta_2}\phi^{-1}(t)/3)}}\in B(y, {\delta_2}\phi^{-1}(t)/3)\Big).\end{align*}  For any $x, y\in M$ with $d(x,y)\ge {\delta_1}\phi^{-1}(t)\ge {\delta_2}\phi^{-1}(t),$
$B(y, {\delta_2}\phi^{-1}(t)/3)\subset B(x, 2{\delta_2}\phi^{-1}(t)/3)^c$.  Then  by $\J_{\phi,\ge}$ and Lemma \ref{Levy-sys}, for all $x\in M_0$,
\begin{align*}
&\bP^x \Big(X_{(at)\wedge \tau_{B(x,2{\delta_2}\phi^{-1}(t)/3)}}\in B(y, {\delta_2}\phi^{-1}(t)/3)\Big)\\
&= \bE^x\left[ \sum_{s\le{(at)\wedge \tau_{B(x,2{\delta_2}\phi^{-1}(t)/3)}}}{\bf 1}_{\{X_s\in B(y, {\delta_2}\phi^{-1}(t)/3)\}}  \right]\\
&\ge\bE^x\left[\int_0^{{(at)\wedge \tau_{B(x,2{\delta_2}\phi^{-1}(t)/3)}}}ds \int_{B(y, {\delta_2}\phi^{-1}(t)/3)} J(X_s,u)\,\mu(du)\right]\\
&\ge c_7 \bE^x\left[ \int_0^{{(at)\wedge \tau_{B(x,2{\delta_2}\phi^{-1}(t)/3)}}}ds \int_{B(y, {\delta_2}\phi^{-1}(t)/3)}\frac{1}{V(u,d(X_s,u))\phi(d(X_s,u))}\,\mu(du)\right]\\
&\ge c_8 \bE^x\left[ {{(at)\wedge \tau_{B(x,2{\delta_2}\phi^{-1}(t)/3)}}}\right] V(y, {\delta_2}\phi^{-1}(t)/3)\frac{1}{V(y,d(x,y))\phi(d(x,y))}\\
&\ge c_8 at\bP^x\left[{{\tau_{B(x,2{\delta_2}\phi^{-1}(t)/3)}\ge at}}\right] V(y, {\delta_2}\phi^{-1}(t)/3)\frac{1}{V(y,d(x,y))\phi(d(x,y))}\\
&\ge  \frac{c_9 tV(y,\phi^{-1}(t))}{V(x,d(x,y))\phi(d(x,y))},
\end{align*}
where
in the third inequality we have used the fact that
$$
d(X_s,u)\le d(X_s, x)+d(x,y)+d(y,u)\le d(x,y)+{\delta_2}\phi^{-1}(t)\le 2d(x,y).
$$

Therefore, from (i),  for almost all $x, y\in M$ with $d(x,y)\ge {\delta_1}\phi^{-1}(t),$
\begin{align*}
p(t, x,y) &\ge \int_{B(y, {\delta_1}\phi^{-1}(t))} p(at,x,z) p((1-a)t,z,y)\,\mu(dz)\\
&\ge\inf_{z\in B(y, {\delta_1}\phi^{-1}((1-a)t))} p({(1-a)t},z,y)\int_{B(y, {\delta_1}\phi^{-1}((1-a)t))} p(at,x,z)\,\mu(dz)\\
&\ge\frac{c_{10}}{V(y,\phi^{-1}(t))}\cdot\frac{c_9 tV(y,\phi^{-1}(t))}{V(x,d(x,y))\phi(d(x,y))}\\
&=\frac{c_{11} t}{V(x,d(x,y))\phi(d(x,y))}.
\end{align*}
The proof is complete.  \qed
\end{proof}

\begin{remark}\rm
We emphasis that the on-diagonal lower bound estimate \eqref{odle} is based on $\E_\phi$ only. \end{remark}

The following lemma has been used in the proof above.
\begin{lemma}\label{L:holder1} Under $\VD$, \eqref{polycon}, $\J_\phi$ and $\E_\phi$, the heat kernel $p(t, x,y)$ is H\"{o}lder continuous with respect to $(x,y)$. More explicitly, there exist constants $\theta\in(0,1)$ and $c_3>0$ such that for all $t>0$ and $x,y, z\in M$,
 \be\label{eq:holdboef}
|p(t, x,y)- p(t, x,z)| \le  \frac{c_3}{V(x,\phi^{-1}(t))} \left(\frac{d(y,z)}{\phi^{-1}(t)}\right)^\theta.
\ee
 \end{lemma}

\begin{proof} The proof is essentially the same as that of \cite[Theorem 4.14]{CK1}, and we should highlight a few different steps. Let $Z:=\{V_s,X_s\}_{s\ge0}$ be a space-time process where $V_s=V_0-s$. The filtration generated by $Z$ satisfying the usual conditions will be denoted by $\{\widetilde{\mathcal{F}}_s;s\ge0\}$. The law of the space-time process $s\mapsto Z_s$ starting from $(t,x)$ will be denoted by $\bP^{(t,x)}$. For every open subset $D$ of $[0,\infty)\times M$, define
$\tau_D=\inf\{s>0:Z_s\notin D\}$ and $\sigma_D=\inf\{t>0: Z_t\in D\}$.

According to Subsection \ref{5-15-1}, $\J_\phi+\E_\phi$ imply
$\UHK(\phi)$.  Then  by Lemma
\ref{Conserv}, $\EP_{\phi,\le}$ holds, i.e.\ there is a constant
$c_0\in(0,1)$ such that for all $x\in M_0$ and $r>0$,
\begin{equation}\label{diff}\bP^{(0,x)}(\tau_{B(x,r)}\le c_0 \phi(r))\le 1/2.\end{equation}  Let $Q(t,x,r)=[t,t+c_0\phi(r)]\times B(x,r)$.  Then,
 following the argument of \cite[Lemma 6.2]{CK2} and using the L\'{e}vy system for the process $\{X_t\}$ (see Lemma \ref{Levy-sys}), we can obtain that there is a constant $c_1>0$ such that for
all $ x\in M_0$,
$t,r>0$ and any compact subset $A\subset Q(t,x,r)$
\begin{equation}\label{proof-holder-1}\bP^{(t,x)} (\sigma_A<\tau_{Q(t,x,r)})\ge c_1\frac{m\otimes \mu(A)}{V(x, r)\phi(r)},\end{equation} where $m\otimes \mu$ is a product measure of the Lebesgue measure $m$ on $\bR_+$ and $\mu$ on $M$. Note that
unlike \cite[Lemma 6.2]{CK2}, here \eqref{proof-holder-1} is satisfied
for all $r>0$ not only $r\in(0,1]$, which is due to the fact \eqref{diff} holds for all $r>0$.

Also by the L\'{e}vy system of the process $\{X_t\}$ (see Lemma \ref{Levy-sys}),
we find that there is a constant $c_2>0$ such that for all $ x\in M_0$,
$t,r>0$ and $s\ge 2r$,
\begin{equation}\label{proof-holder-2}\begin{split}\bP^{(t,x)}(X_{\tau_{Q(t,x,r)}}\notin B(x,s))&=\bE^{(t,x)} \int_0^{\tau_{Q(t,x,r)}} \int_{B(x,s)^c} J(X_v,u)\,\mu(du)\, dv\\
&\le \bE^{(t,x)} \int_0^{\tau_{Q(t,x,r)}} \int_{B(X_v,s/2)^c} J(X_v,u)\,\mu(du)\, dv\\
&\le c_2\frac{\phi(r)}{\phi(s)},\end{split}\end{equation} where in the last inequality we have used \eqref{polycon}, Lemma \ref{intelem} and $\E_\phi$.

Having \eqref{proof-holder-1} and \eqref{proof-holder-2} at hand, one can follow the argument of \cite[Theorem 4.14]{CK1} to get that the H\"{o}lder continuity of bounded parabolic functions, and so the desired assertion for the heat kernel $p(t, x,y).$
 \qed
\end{proof}

\begin{remark}\rm
The proof above is based on \eqref{diff}, \eqref{proof-holder-1} and
\eqref{proof-holder-2}. According to Lemma \ref{E}, \eqref{diff} is
a consequence of $\E_\phi$; while, from the argument above,
\eqref{proof-holder-2} can be deduced from $\J_{\phi,\le}$ and
$\E_{\phi,\le}$. \eqref{proof-holder-1} is the so called Krylov type
estimate, which is a key to yield the H\"{o}lder continuity of
bounded parabolic functions, and where $\J_{\phi,\ge}$ is used.
\end{remark}

\section{Applications and Example}\label{Sectin-ex}
\subsection{Applications}
We first give examples of $\phi$ such that condition
\eqref{polycon} is satisfied (see \cite[Example 2.3]{CK2}).
\begin{example}\label{E:2.3} \rm
\begin{description}
\item{$(1)$}
Assume that there exist $0<\beta_1\le \beta_2<\infty$ and a probability measure $\nu$ on $[\beta_1,\beta_2]$ such that
$$\phi(r)=\int_{\beta_1}^{\beta_2} r^\beta\,\nu(d\beta),\quad r\ge 0.$$
Then \eqref{polycon} is satisfied.
Clearly, $\phi$ is a continuous strictly increasing function with $\phi (0)=0$.
Note that some additional restriction of the range of $\beta_2$ should be imposed for the corresponding
Dirichlet form to be regular. $($For instance, $\beta_2<2$ when $M=\bR^n$.$)$
In this case,
\be\label{jwongkd}
J(x, y) \asymp \frac1{V(x,d(x, y)) \int_{\beta_1}^{\beta_2} d(x, y)^\beta\,
 \nu (d \beta )},\quad x,y\in M.
\ee
When $\beta_1=\beta_2=\beta$ (i.e. $\nu (\{\beta\})=1$), a symmetric jump process whose
jump density
is comparable to \eqref{jwongkd} is called a symmetric $\beta$-stable like process.
\item{$(2)$} Similarly,
consider the following increasing function
$$\phi(r)=\left(\int_{\beta_1}^{\beta_2} r^{-\beta}\,\nu(d\beta)\right)^{-1}
~\mbox{ for }~r>0,~~~\phi (0)=0,$$
where $\nu$ is a finite measure on $[\beta_1,\beta_2]\subset (0,\infty)$. Then \eqref{polycon} is satisfied.
Again, $\phi$ is a continuous strictly increasing function, and some
additional restriction of the range of $\beta_2$ should be imposed for the corresponding
Dirichlet form to be regular. In this case,
$$ J(x, y) \asymp  \frac{1}{V(x,d(x, y))}\int_{\beta_1}^{\beta_2}
\frac1{ d(x, y)^\beta } \,\nu (d \beta),\quad x,y\in M.$$
A particular case is when $\nu$ is a discrete measure.
For example, when $\nu (A)=\sum_{i=1}^N\delta_{\alpha_i}(A)$ for some $\alpha_i\in(0,\infty)$ with $1\le i\le N$ and $N\ge 1$,
$$J(x, y)\asymp \sum_{i=1}^N\frac1{V(x,d(x, y)) d(x, y)^{\alpha_i} }.$$
\end{description}
\end{example}

\ \

We now give an important class of examples where $\beta$, $\beta_2$ in \eqref{polycon}
could be strictly larger than $2$, and then discuss the stability of heat kernel estimates.

The first class of examples are given as subordinations of diffusion processes on fractals.
First, let us define the Sierpinski carpet as a typical example of fractals.
Set $E_0=[0,1]^n$. For any $l \in \bN$ with $l\geq 2$, let
$${\cal Q}=\Big\{\Pi_{i=1}^n [(k_i-1)/l,k_i/l]: 1\le k_i\le l,~k_i\in
\bN,~1\le i\le n\Big\}.$$ For any $l\le N \le l^n$, let $F_i$ ($1\le
i\le N$) be orientation preserving affine maps of $E_0$ onto some
element of ${\cal Q}$. (Without loss of generality, let
$F_1(x)=l^{-1}x$ for $x\in E_0$ and assume that the sets $F_i(E_0)$
are distinct.) Set
$I=\{1, \dots,  N\}$ and $E_1=\cup_{i\in I}F_i (E_0)$.
 Then  there exists
a unique non-empty compact set $\hat M \subset E_0$ such that $\hat
M=\cup_{i\in I}F_i(\hat M)$. $\hat M$ is called a Sierpinski carpet
if the following hold:
\begin{itemize}
\item[(SC1)] (Symmetry) $E_1$ is preserved by all the isometries of the unit cube $E_0$.
\item[(SC2)] (Connectedness) $E_1$ is connected.
\item[(SC3)] (Non-diagonality) Let $B$ be a cube in $E_0$ which is the union of
$2^d$ distinct elements of ${\cal Q}$. (So $B$ has side length $2l^{-1}$.) If $\mbox{Int} (E_1 \cap B)\ne \emptyset$, then it is connected.
\item[(SC4)] (Borders included property) $E_1$ contains the set $\{x: 0\leq x_1
\leq 1, x_2=\cdots =x_d=0 \}$.
\end{itemize}
Note that Sierpinski carpets are
infinitely ramified in the sense that $\hat M$ can not be disconnected by
removing a finite number of points. Let
\[E_k:=\bigcup_{i_1,\cdots,i_k\in I}F_{i_1}\circ\cdots\circ F_{i_k}
(E_0),~~~M_{{\rm pre}}:=\bigcup_{k\ge 0}l^kE_k~~\mbox{and}~~M
:=\bigcup_{k\ge 0}l^k\hat M.\]
$M_{{\rm pre}}$ is called a pre-carpet, and $M$ is called an unbounded carpet.
Both Hausdorff dimensions of $\hat M$ and $M$
with respect to the Euclidean metric are $d=\log N/\log l$. Let $\mu$ be the
(normalized) Hausdorff measure on $M$. The following has been proved in \cite{BB1}:

There exists a $\mu$-symmetric conservative diffusion on $M$
that has a symmetric jointly continuous transition density
$\{q(t,x,y): t>0, x,y\in M\}$ with the following estimates for
all $t>0, x,y\in M$:
\begin{align}
c_{1}t^{-\alpha/\beta_*}\exp \bigg(-c_{2}\Big(\frac{|x-y|^{\beta_*}}t\Big)^{\frac 1{\beta_*-1}}\bigg)
&\le  q(t,x,y)\label{frachk1}\\
&\le
c_{3}t^{-\alpha/\beta_*}\exp \bigg(-c_{4}\Big(\frac{|x-y|^{\beta_*}}t\Big)^{\frac 1{\beta_*-1}}\bigg),
\nonumber\end{align}
where $0<\alpha\le n$ and $\beta_*\ge 2$.
In fact, it is known that there exist $\mu$-symmetric diffusion processes with the
above heat kernel estimates
on various fractals including the Sierpinski gaskets and nested fractals, and
typically $\beta_*>2$. For example, for the two-dimensional Sierpinski gasket,
$\alpha=\log 3/\log2$ and $\beta_*=\log5/\log2$ (see \cite{B,K2}
for details).

Next, let us consider a more general situation. Let $(M, d, \mu)$ be a metric measure space
as in the setting of this paper that satisfies $\VD$ and $\RVD$.
Assume that there exists a $\mu$-symmetric
 diffusion process $\{Z_t\}$
 that admits no killings inside $M$, and
 has a symmetric and jointly continuous transition density
$\{q(t,x,y): t>0, x,y\in M\}$ with the following estimates for
all $t>0, x,y\in M$:

\begin{align}
\frac{c_1}{V(x,\Psi^{-1}(t))}\exp \Big(-c_{2}\Big(\frac{\Psi(d(x,y))}t\Big)^{\gamma_1}\Big)
&\le  q(t,x,y)\label{frachk2}\\
&\le
\frac{c_3}{V(x,\Psi^{-1}(t))}\exp \Big(-c_{4}\Big(\frac{\Psi(d(x,y))}t\Big)^{\gamma_2}\Big),
\nonumber\end{align}
where $\Psi: \bR_+\to \bR_+$ is a strictly increasing continuous
function  with $\Psi (0)=0$, $\Psi(1)=1$ and satisfying \eqref{polycon}.
The lower bound in \eqref{frachk2} implies that
$$
q(t, x, y) \geq \frac{c_1 e^{-c_2}}{V(x, \Psi^{-1}(t))} \quad \hbox{for } d (x, y) \leq \Psi^{-1}(t)
$$
and so we conclude by Proposition \ref{P:3.1}(2) that
the process $\{Z_t\}$ has infinite lifetime.

Clearly \eqref{frachk1} is a special case of \eqref{frachk2} with
$V(x,r)\asymp r^\alpha$, $\Psi(s)=s^{\beta_*}$ and $\gamma_1=\gamma_2=1/(\beta_*-1)$.
A typical example that the local and global structures of $\Psi$ differ is a so called
fractal-like manifold. It is a $2$-dimensional Riemannian manifold whose global structure is
like that of the fractal. For example, one can construct it from $M_{{\rm pre}}$ by changing each bond
to a cylinder and smoothing the connection to make it a manifold.
One can naturally construct a Brownian motion on the surfaces of cylinders. Using the stability of heat kernel estimates like \eqref{frachk2}
(see for instance \cite{BBK1} for details),
one can show that any divergence operator ${\mathcal L}=\sum_{i,j=1}^2\frac{\partial}
{\partial x_i}(a_{ij}(x)\frac{\partial}{\partial x_j})$
in local coordinates on such manifolds that satisfies
the uniform elliptic condition obeys \eqref{frachk2} with $\Psi(s)=s^2+s^{\beta_*}$.

We now subordinate the diffusion $\{Z_t\}$ whose heat kernel enjoys
\eqref{frachk2}. Let $\{\xi_t\}$ be a subordinator that is
independent of $\{Z_t\}$; namely, it is an increasing L\'evy process
on $\bR_+$. Let $\bar \phi$ be the Laplace exponent of the
subordinator, i.e. $$\bE[\exp (-\lambda \xi_t)]=\exp
(-t\bar\phi(\lambda)),\quad \lambda, t>0.$$ It is known that $\bar
\phi$ is a Bernstein function, i.e. it is a $C^\infty$ function on
$\bR_+$ and $(-1)^nD^n \bar\phi\le 0$ for all $n\ge 0$. See for
instance \cite{SSV}
for the general theory of subordinations. See
also \cite{BSS,K1,Sto} for subordinations on fractals. By the general
theory, there exist $a,b\ge 0$ and a measure $\mu$ on $\bR_+$ satisfying
$\int_0^\infty (1\wedge t)\,\mu(dt)<\infty$ such that
\be\label{feiohi}
\bar\phi(\lambda)=a+b\lambda+\int_0^\infty(1-e^{-\lambda
t})\,\mu(dt). \ee Below, we assume that $\bar \phi$ is a complete
Bernstein function; namely, the measure $\mu(dt)$ has a completely
monotone density $\mu(t)$, i.e. $(-1)^nD^n \mu\ge 0$ for all $n\ge
0$. Assume further that $\bar\phi$ satisfies \eqref{polycon} with
different $\beta_1,\beta_2$ from those for $\Psi$, and that
furthermore $\beta_1,\beta_2\in (0,1)$.  Then  $a=b=0$ in
\eqref{feiohi} and one can obtain $\mu(t)\asymp \bar\phi(1/t)/t$
(see \cite[Theorem 2.2]{KSV}).

The process $\{X_t\}$ defined by $X_t=Z_{\xi_t}$ for any $t\ge0$ is called a subordinate process.
Let $\{\eta_t (u):t>0, u\ge 0\}$ be the distribution density of $\{\xi_t\}$.
It is known (see for instance \cite{BSS,Sto}) that the L\'evy density $J(\cdot,\cdot)$ and the heat kernel
$p(t,\cdot,\cdot)$ of $X$ are given by
\begin{align}
J(x, y)&=\int_0^\infty q(u,x,y) \mu (u)\, du,\label{eq:jphibiw}\\
p(t,x,y)&=\int_0^{\infty}q(u,x,y)\eta_t(u)\,du\quad \mbox{for all}~~
t>0,~x,y\in M.\label{eq:pqbihk}
\end{align}

Define
\begin{equation}\label{bieibfis}
\phi(r)=\frac 1{\bar\phi(1/\Psi(r))}.
\end{equation}
Then $\phi$ also satisfies \eqref{polycon} (with different $\beta_1,\beta_2$
from those for $\bar\phi$ and $\Psi$).
From now on, we discuss whether $p(t,\cdot,\cdot)$ satisfies $\HK(\phi)$ or not.
The most classical case is
when $(M, d, \mu)$ is the Euclidean space $\bR^d$ equipped with the Lebesgue measure $\mu$,
  $\{Z_t\}$ is Brownian motion on $\bR^d$  (and so
$\beta_*=2$ and $\gamma_1=\gamma_2=1$), and $\bar\phi(t)=t^{\alpha/2}$ with $0<\alpha<2$.
In this case
$\{\xi_t\}$ is an
$\alpha/2$-stable subordinator and the corresponding subordinate process is
the rotationally symmetric $\alpha$-stable process on $\bR^d$.
For a diffusion on a fractal whose heat kernel enjoys \eqref{frachk1} for some $\beta_*>2$, it is proved in \cite[Theorem 3.1]{BSS} that $p(t,\cdot,\cdot)$ satisfies $\HK(\phi)$ with $\phi(r)=r^{\beta_*\alpha/2}$ when $\bar\phi(t)=t^{\alpha/2}$. (Note that $\beta_*\alpha/2>2$
when $\alpha > 4/\beta_*$.)
The proof uses \eqref{eq:pqbihk} and some
estimates of $\eta_t(u)$ such as
\[
\eta_t(u)\le c_5tu^{-1-\alpha/2},\quad t,u>0.
\]

Now let us consider the case
$\Psi(s)=s^{\beta_{*,1}}+s^{\beta_{*,2}}$ with $2\le \beta_{*,1}\le
\beta_{*,2}$ (e.g.\ the fractal-like manifold is a special case in
that $\beta_{*,1}=2$), and
$\bar\phi(t)=t^{\alpha_1/2}+t^{\alpha_2/2}$ for some $0<\alpha_1\le
\alpha_2<2$. For this case, $\{\xi_t\}$ is a sum of independent
$\alpha_1/2$- and $\alpha_2/2$-subordinators, so the
distribution density $\eta_t(u)$ is a convolution of their
distribution densities. Hence we have
\begin{equation}\label{eq:convetai}
\eta_t(u)\le c_6t/(u^{1+\alpha_1/2}\wedge u^{1+\alpha_2/2}).
\end{equation}
By elementary but tedious computations (along similar lines as in the proof of \cite[Theorem 3.1]{BSS}), one can deduce that $p(t,\cdot,\cdot)$ satisfies $\HK(\phi)$ with
\begin{equation} \label{e:6.9}
\phi(r)=r^{ \alpha_2 \beta_{*,1}/2}1_{\{r\le 1\}}+r^{ \alpha_1 \beta_{*,2}/2}1_{\{r>1\}},
\end{equation}
which is (up to constant multiplicative) the same as
\eqref{bieibfis}. In fact, the computation by using
\eqref{eq:pqbihk} also requires various estimates of $\eta_t(u)$,
which are in general rather complicated. An alternative way is to
prove first $\J_\phi$ by using \eqref{eq:jphibiw}, which is easier
since we have $\mu(t)\asymp \bar\phi(1/t)/t$.  Then  we can obtain
$$p(t,x,y)\le\frac{ c_7t}{V(x,d(x,y))\phi (d(x,y))}$$ by plugging
\eqref{eq:convetai} into \eqref{eq:pqbihk}. Integrating this, we
have $\bP^x(X_t\notin B(x,r))\le c_8t/\phi(r)$ for all $x\in M$ and
$r, t>0$. Note that since the diffusion process $\{Z_t\}$ has
infinite lifetime, so does the subordinated process $\{X_t\}$. Then,
following the argument of Lemma \ref{Conserv}, we can get that
$\bP^x( \tau_{B(x,r)}\le t)\le c_9t/\phi(r)$ for all $x\in M$ and
$r, t>0$.
  Consequently, by taking $\eps>0$ sufficiently small, we have
$$ \bP^x ( \tau_{B(x, r)}\geq \phi (\eps r))=1-\bP^x ( \tau_{B(x, r)} <\phi (\eps r))
\geq 1-\frac{c_9 \phi (\eps r)}{\phi(r)} \geq c_{10}>0,
$$
 which implies $\E_{\phi,\ge}$.
Under $\VD$ and $\RVD$, $\J_\phi$
implies $\E_{\phi,\le}$ (which is due to Section \ref{jFKnow} and
Lemma \ref{upper-e}). Therefore,  by Theorem \ref{T:main}, we
conclude that $p(t,\cdot,\cdot)$ satisfies $\HK(\phi)$.

The above argument shows that $\HK(\phi)$ is satisfied for the
subordinated process $\{X_t\}$ when
$\bar\phi(t)=t^{\alpha_1/2}+t^{\alpha_2/2}$. It follows from  our
stability theorem, Theorem \ref{T:main},  that
 for any symmetric pure jump process on the above
mentioned space whose jumping kernel enjoys $\J_\phi$ with $\phi$  given by
\eqref{e:6.9},  it enjoys the two-sided heat kernel estimates
$\HK(\phi)$.

The stability results we discuss above are new in general, especially for high dimensional
Sierpinski carpets. However, if we restrict the framework so that (roughly)
$\alpha<\beta_*$ in \eqref{frachk1} (which is the case for diffusions on the Sierpinski gaskets,
for instance), then the stability for the heat kernel was already established
in \cite{GHL2}. See \cite[Examples 6.16 and 6.20]{GHL2} for related examples.

\subsection{Counterexample}

In this subsection, we show that $\J_{\phi}$ does not imply $\HK(\phi)$ through the following counterexample.

\begin{example}\label{E:82-1}{\bf ($\J_{\phi}$ does not imply $\HK(\phi)$.)}\rm\quad
In \cite{BBK2,CK1}, it is proved in the setting of graphs or $d$-sets that $\J_{\phi}$ is equivalent to $\HK(\phi)$,
when $V(x,r)\asymp r^d$ and $\phi(r)=r^\alpha$ with $0<\alpha<2$. Here, we give an example that
this is not the case in general.

Let $M=\bR^d$,
$
\phi(r)=r^\alpha+r^\beta$ with $ 0<\alpha<2<\beta$, and
\[J(x,y)\asymp\frac 1{|x-y|^d\phi(|x-y|)},\quad x,y\in\bR^d.
\]
Note that $\phi(r)\asymp r^\alpha$ if $r\le 1$, and $\phi(r)\asymp r^\beta$ if $r\ge 1$.
This example clearly satisfies $\J_{\phi}$.
We first prove the following
\begin{equation}\label{main1-5}
p(t,x,y)\leq \left\{\!\!\begin{array}{ll} c_1t^{-d /\alpha},
&\quad  t\in (0,1], \\
c_2t^{-d /2},  &\quad  t\in [1,\infty).
\end{array}\right.\end{equation}
Indeed, for the truncated process $\{X^{(1)}_t\}$ with $$J_0(x,y)=J(x,y){\bf1}_{\{|x-y|\le 1\}}\asymp \frac 1{|x-y|^d\phi(|x-y|)}{\bf1}_{\{|x-y|\le 1\}},$$
it is proved in \cite[Proposition 2.2]{CKK} that \eqref{main1-5} holds.
Since \eqref{main1-5} is equivalent to
\begin{equation}\label{3.ss6}
 \theta (\| u\|_2^2) \leq c_3\, \sE (u, u)\quad
 \mbox{for every $u\in \FF$ with $\| u\|_1=1$},
\end{equation}
where  $\theta (r)=r^{1+\alpha/d}\vee r^{1+2/d}$
(see for instance \cite[theorem II.5]{Cou}),
it follows from the fact $J_0(x,y)\preceq J(x,y)$ that \eqref{3.ss6} and so \eqref{main1-5} hold for the original process $\{X_t\}$.
So if we take $t=c_4(r^\alpha \vee r^2)$ for $c_4>0$ large enough, then for all $x$, $x_0\in \bR^d$ and $r>0$,
\[ \bP^x(X_t\in B(x_0,r))=\int_{B(x_0,r)} p(t,x,z)\,dz\leq c_5(t^{-d/\alpha}\vee t^{-d/2})r^d \leq \tfrac12. \]
This implies
$\bP^x(\tau_{B(x_0,r)}> t)\leq \tfrac12$.
Using the strong Markov property of $X$, we have for all $x, x_0\in \bR^d$,
$\bP^x(\tau_{B(x_0,r)}>k t)\leq 2^{-k}$ and so
$\bE^x \tau_{B(x_0, r)} \leq c_6 t= c_4 c_6(r^\alpha \vee r^2)$.
Thus $\E_\phi$ fails, and so $\HK(\phi)$ does not hold either.
\end{example}

\section{Appendix}
\subsection{The L\'{e}vy system formula}
The following formula is used many times in this paper. See, for example \cite[Appendix A]{CK2}
for the proof.
\begin{lemma}\label{Levy-sys}
Let $f$ be a non-negative measurable function on ${\mathbb R}_+
\times M \times M$ that vanishes along the diagonal. Then for every
$t\geq 0 $, $x\in M_0$ and  stopping time $T$ $($with
respect to the filtration of $\{X_t\}$$)$,
$$
{\mathbb E}^x \left[\sum_{s\le T} f(s,X_{s-}, X_s) \right]={\mathbb
E}^x \left[ \int_0^T  \int_M f(s,X_s, y)\, J(X_s,dy)
\,ds \right].
 $$
\end{lemma}

\subsection{Meyer's decomposition}\label{Sect9.2}
 We use the following construction of Meyer \cite{Me} for jump processes. Assume that $J(x,y)=J'(x,y)+J''(x,y)$ for any $x,y\in M$, and that there exists a constant $C>0$ such that
 $$
 {\cal J} (x)=\int J''(x,y)\,\mu(dy)\le C\quad \textrm{ for all }x\in M.$$
 Note that, by Lemma \ref{intelem} the assumption above holds for $J''(x,y):={\bf1}_{\{d(x,y)\ge r\}}J(x,y)$ with $r>0$, when $\VD$, \eqref{polycon}  and $\J_{\phi,\le}$ are satisfied. Let $\{Y_t\}$ be a process corresponding to the jumping kernel $J'(x,y)$. Then we can construct a process
 $\{X_t\}$ corresponding to the jumping kernel $J(x,y)$ by the following procedure.
 Let $\xi_i$, $i\ge 1$, be i.i.d. exponential random variables of parameter $1$ independent of $\{Y_t\}$. Set
 $$
 H_t=\int_0^t {\cal J} (Y_s)\,ds, \quad T_1=\inf\big\{t\ge0: H_t\ge \xi_1\big\} \quad \textrm{ and } \quad Q(x,y)=\frac{J''(x,y)}{{\cal J}(x)}.
 $$
 We remark that $\{Y_t\}$ is a.s. continuous at $T_1$. We let $X_t=Y_t$ for $0\le t<T_1$, and then define $X_{T_1}$ with law $Q(X_{T_1-}, y)\,\mu(dy)=Q(Y_{T_1}, y)\,\mu(dy)$.  The construction now proceeds in the same way from the new space-time starting point $(T_1, X_{T_1})$. Since ${\cal J}(x)$ is bounded, there can be a.s. only finitely many extra jumps added in any bounded time interval. In \cite{Me} it is proved that the resulting process corresponds to the jumping kernel $J(x,y)$.

In the following, we assume that both $\{X_t\}$ and $\{Y_t\}$ have
transition densities. Denote by $p^X(t, x,y)$ and $p^Y(t,x,y)$ the
transition density of $\{X_t\}$ and $\{Y_t\}$, respectively. The
relation below between $p^X(t, x,y)$ and $p^Y(t,x,y)$ has been shown in
\cite[Lemma 3.1 and (3.5)]{BGK1} and \cite[Lemma 3.6]{BBCK}.
\begin{lemma}\label{meyer} We have the following.
\begin{itemize}
\item[\rm(1)] For almost all $x,y\in M$, $$p^X(t, x,y)\le p^Y(t,x,y)+\bE^x\int_0^t\,ds\int J''(Y_s,z)\,p^X(t-s, z,y)\,\mu(dz).$$
\item[\rm(2)] Let $A\in \sigma(Y_t, 0<t<\infty)$.  Then  for almost all $x\in M$,
\begin{equation}\label{e:note}
\bP^x(A)\le e^{t\,\,\| {\cal J} \|_\infty } \bP^x(A\cap \{X_s=Y_s\textrm{ for all }0\le s\le t\}).
\end{equation}
In particular,
$$
p^Y (t, x,y)\le p^X(t, x,y)e^{t\,\,\| {\cal J} \|_\infty }.
$$
\end{itemize}
\end{lemma}
Note that, by \eqref{e:note}, if the process $\{X_t\}$ has transition density functions, so does $\{Y_t\}$.
\subsection{Some results related to $\FK(\phi)$.}\label{ond-Nash}

The following is a general equivalence of $\FK(\phi)$ for regular Dirichlet forms.

\begin{prop}\label{P:3.3}  Assume that $\VD$ and \eqref{polycon} hold.  Then  the following are
equivalent.
\begin{itemize}
\item[\rm(1)] $\FK(\phi)$.
\item[\rm(2)] $\Nash(\phi)_B$; namely, there exist constants $C_1,\nu>0$ such that
for each $x\in M$ and $r>0$,
\[
\frac{V(x,r)^\nu}{\phi(r)}\|u\|_2^{2+2\nu}\le C_1\sE(u,u)\|u\|_1^{2\nu},\quad u\in \sF_{B(x,r)}.
\]
\item[\rm(3)] There exist constants $C_1,\nu>0$ such that
for any ball $B=B(x,r)$, the Dirichlet heat kernel $p^B(t,\cdot,\cdot)$ exists and satisfies that
\[
\esssup_{y,z\in B} p^B(t,y,z)\le \frac{C_1}{V(x,r)}\Big(\frac{\phi(r)}t\Big)^{1/\nu},
\quad t>0.
\]
\end{itemize}
\end{prop}
\begin{proof} $\rm(1)\Longrightarrow (2)\Longrightarrow (3)$ can be proved similarly to
\cite[Lemmas 5.4 and 5.5]{GH} by choosing $a=CV(x,r)^\nu/\phi (r)$
in the paper.

$\rm (3)\Longrightarrow (1)$ can be proved similarly to
the approach of \cite[p.\ 553]{GH}. Note that \cite{GH} discusses the case $\phi(r)=r^\beta$,
but the generalization to $\phi$ is easy by using \eqref{polycon}.
\qed\end{proof}

Under  $\VD$ and $\RVD$, we have further statements for $\FK(\phi)$.

\begin{prop}\label{P:3.3-2}  Assume that $\VD$, $\RVD$ and \eqref{polycon} hold.
Consider the following inequalities:
\begin{itemize}
\item[$(1)$] $\FK(\phi)$.
\item[$(2)$] There exist constants $c_1,\nu>0$ such that for each $x\in M$ and
$r>0$,
\[
\|u\|_2^{2+2\nu}\le \frac{c_1}{V(x,r)^{\nu}}\|u\|_1^{2\nu}\Big(\|u\|_2^2+\phi(r)\sE(u,u)\Big),\quad
u\in \sF_{B(x,r)}.
\]
\item[$(3)$] $\Nash(\phi)_{loc}$; namely, there exists a constant $c_2>0$ such that for each $s>0$,
\[
\|u\|_2^2\le c_2\Big(\frac {\|u\|_1^2}{\inf_{z\in {\rm supp}\,u}V(z,s)}+\phi(s)\sE(u,u)\Big),
\quad  u\in \sF\cap L^1(M;\mu).
\]
\end{itemize}
 We have $\displaystyle (1)\Longleftrightarrow (2)\Longleftarrow (3)$.
 \end{prop}

\begin{proof}
$(1)\Longleftrightarrow \Nash(\phi)_B$ is in Proposition \ref{P:3.3}.
$(2)\Longleftrightarrow \Nash(\phi)_B$ is given in \cite[Proposition 3.4.1]{BCS}
(they are proved for the case $\phi(t)=t^2$ but the modifications are easy), while $(3)\Longrightarrow (2)$ is given in \cite[Proposition 3.1.4]{BCS}.
We note that in all the proofs above $\RVD$ is used only in $(2)\Longrightarrow \Nash(\phi)_B$, and $(2)\Longleftarrow \Nash(\phi)_B$ holds trivially.
We thus obtain the desired results.
\qed\end{proof}

We now define the weak Poincar\'e inequality which will be used in the subsequent paper
\cite{CKW}.
\begin{definition} {\rm
We say that the {\em weak Poincar\'e inequality}
($\PI(\phi)$)
 holds if there exist constants $C>0 $ and $\kappa\ge1$ such that
for any  ball $B_r=B(x,r)$ with $x\in M$  and for any $f \in \sF_b$,
\begin{equation*}\label{eq:PIn}
\int_{B_r} (f-\ol f_{B_r})^2\, d\mu \le C \phi(r)\int_{{B_{\kappa r}}\times {B_{\kappa r}}} (f(y)-f(x))^2\,J(dx,dy),
\end{equation*}
where $\ol f_{B_r}= \frac{1}{\mu({B_r})}\int_{B_r} f\,d\mu$ is the average value of $f$ on ${B_r}$.} \end{definition}

\begin{proposition}\label{pi-e-pre} Assume that $\VD$ and \eqref{polycon} hold. Then either $\PI(\phi)$ or $\UHKD(\phi)$ implies $\Nash(\phi)_{loc}$.
Consequently, if $\VD$, $\RVD$ and \eqref{polycon} are satisfied,
then either $\PI(\phi)$ or $\UHKD(\phi)$ implies $\FK(\phi)$.
\end{proposition}
\begin{proof} (i)
When $\phi(t)=t^2$, this fact that $\PI(\phi)\Longrightarrow
\Nash(\phi)_{loc}$ is well-known; see for example \cite[Theorem
2.1]{Sa}. Generalization to this setting is a line by line
modification.  Then  the second assertion follows from Proposition
\ref{P:3.3-2}.

(ii) That $\UHKD(\phi)$ implies $\Nash(\phi)_{loc}$ can be proved similarly to \cite[Corollary 2.4]{Ki}. (We note that
in \cite[Corollary 2.4]{Ki} it is proved
for the case $\phi(t)=t^\beta$, but the modifications are easy.) One also can prove this similarly to
the approach of \cite[p.\ 551--552]{GH}. Note that \cite{GH} discusses the case $\phi(r)=r^\beta$,
but the generalization to $\phi$ is also easy.
\qed\end{proof}

\begin{proposition}\label{P:reg} Under $\VD$ and \eqref{polycon}, $\FK(\phi)$ implies that the semigroup $\{P_t\}$ is locally ultracontractive, which in turn yields that
\begin{itemize}
\item[$(1)$] there exists a properly exceptional set $\mathcal{N}\subset M$ such that, for any open subset $D\subset M$, the semigroup $\{P_t^D\}$ possesses the heat kernel $p^D(t,x,y)$ with domain $D\setminus \mathcal{N} \times D\setminus \mathcal{N}$. \\
\item[$(2)$] Let $\varphi(x,y):M_0\times M_0\to [0,\infty]$ be a upper semi-continuous function such that for some open set $D\subset M$ and for some $t>0$, $$p^D(t,x,y)\le \varphi(x,y)$$ for almost all $x,y\in D$. Then the inequality above holds for all $x,y\in D\setminus \mathcal{N} .$
\end{itemize} \end{proposition}

\begin{proof}The statement of Proposition \ref{P:3.3} tells us that, under $\VD$, \eqref{polycon} and $\FK(\phi)$, there exist  constants $C_1$, $\nu>0$ such that for any ball $B=B(x,r)$ with $x\in M$ and $r>0$, and any $t>0$,
$$ \|P^B_t\|_{L^1(B;\mu)\to L^\infty(B;\mu)}\le \frac{C_\nu}{V(x,r)} \left( \frac{\phi(r)}{t}\right)^{1/\nu
}. $$
Therefore, the semigroup $\{P_t\}$ is locally ultracontractive.
The other assertions follow from \cite[Theorem 6.1]{BBCK} and \cite[Theorem 2.12]{GT}. \end{proof}

\subsection{Some results related to
the (Dirichlet) heat kernel}

Recall that for any $\rho>0$, $(\sE^{(\rho)}, \sF)$ is the $\rho$-truncated Dirichlet form, which is obtained by $\rho$-truncation for the jump density of the original Dirichlet form $(\sE, \sF)$, i.e.
$$\sE^{(\rho)}(f,g)=\int (f(x)-f(y))(g(x)-g(y)){\bf 1}_{\{d(x,y)\le \rho\}}\, J(dx,dy).$$  As mentioned in Section \ref{section2}, if $\VD$, \eqref{polycon}  and $\J_{\phi,\le}$ hold, then $(\sE^{(\rho)}, \sF)$ is a regular
Dirichlet form on $L^2(M; \mu)$. Let $\{X_t^{(\rho)}\}$ be the
process associated with $(\sE^{(\rho)}, \sF)$.  For any non-negative
open set ${D} \subset M$, as before we denote by $\{P_t^{D}\}$ and
$\{Q_t^{(\rho),{D}}\}$ the semigroups of $(\sE, \sF_{D})$ and
$(\sE^{(\rho)}, \sF_{D})$, respectively. (We write
$\{Q_t^{(\rho),M}\}$ as $\{Q^{(\rho)}_t\}$ for simplicity.) Most of
results in this subsection have been proved in \cite{GHL2}. To be
self-contained, we present new proofs by making full use of the
probabilistic ideas.

The following lemma was proved in \cite[Proposition 4.6]{GHL2}.

\begin{lemma}\label{L:semicomp1} Suppose that $\VD$, \eqref{polycon} and $\J_{\phi,\le}$ hold. Let ${D}$ be the open subset of $M$. Then there exists a constant $c>0$ such that for any $t>0$, almost all $x\in {D}$ and any non-negative $f\in L^2({D};\mu)\cap L^\infty({D};\mu)$, $$|P_t^{D} f(x)-Q_t^{(\rho),{D}}f(x)|\le c \|f\|_\infty \frac{t}{\phi(\rho)}.$$ \end{lemma}
\begin{proof}Note that $P_t^{D} f(x)=\bE^x(f(X_t){\bf 1}_{\{\tau_{D}> t\}})$ and $Q_t^{(\rho),{D}} f(x)=\bE^x(f(X^{(\rho)}_t){\bf 1}_{\{\tau^{(\rho)}_{D}> t\}})$.  Let $$T_\rho=\inf\big\{t>0: d(X_t, X_{t-})>\rho\big\}.$$ It is clear that $X_t=X_t^{(\rho)}$ for all $t< T_\rho$.
Thus  by \cite[Lemma 3.1(a)]{BGK1},
  \begin{align*}
  |P_t^{D} f(x)-Q_t^{(\rho),{D}}f(x)|
  &\le  \big|\bE^x(f(X_t){\bf 1}_{\{T_\rho\le t <\tau_{D}\}})\big|+
 \big|\bE^x(f(X^{(\rho)}_t){\bf 1}_{\{T_\rho\le t <\tau^{(\rho)}_{D}\}})\big|\\
 & \le 2\|f\|_\infty  \, \bP^x(T_\rho\le t)\\
 &\le2\|f\|_\infty \left( 1-\exp\left(-t\,\esssup_{z\in M}\int_{B(z,\rho)^c}J(z,y)\,\mu(dy)\right)\right)
 \\
 &\le2\|f\|_\infty t\,\esssup_{z\in M}\int_{B(z,\rho)^c}J(z,y)\,\mu(dy),
 \end{align*}
where the inequality $1-e^{-r}\le r$ for all $r>0$ was used in the last inequality.
The desired  conclusion   now follows from Lemma \ref{intelem}. \qed
\end{proof}

   We need the following comparison of heat kernels in different domains.

\begin{lemma}\label{L:com-domain} Let $V$, $U$ and ${D}$ be open subsets of $M$ such that $U_\rho:=\{z\in M: d(z,U)<\rho\}$ is precompact, $V\subset U$ and $U_\rho\subset {D}$. Then for all $t$, $s>0$,
\begin{equation}\label{uvcom}\begin{split} &\esssup_{x,y\in V}q^{(\rho),{D}}(t+s, x,y)\\
&\le  \esssup_{x,y\in U} q^{(\rho),U}(t, x,y) \\
&\quad + \esssup_{x\in V} \bP^x(\tau^{(\rho)}_U\le t)\,\esssup_{x,y\in U_\rho}q^{(\rho),{D}}(s, x,y).\end{split}\end{equation}\end{lemma}

\begin{proof} For simplicity, in the proof $\{X_t^{(\rho)}\}$ denotes the subprocess of $\{X_t^{(\rho)}\}$ on exiting $D$. For any fixed $x$, $y\in V$, one can choose $r>0$ small enough such that $B(x,r)\subset V$ and $B(y,r)\subset V$. Let $f, g \in L^1({D}; \mu)$ be such that $0\le f,g\le 1$,
${\rm supp}\, f\subset B(x,r)$ and ${\rm supp }\, g\subset B(y,r)$. We set $\bE^\mu[\cdot]:=\int_D \bE^x[\cdot]\,\mu(dx)$.
 Then  we have
\begin{align*}
&  \bE^\mu \left[ f(X_0^{(\rho)})g(X_{t+s}^{(\rho)}) \right]\\
 &=  \bE^\mu  \left[ f(X_0^{(\rho)})g(X_{t+s}^{(\rho)}): \tau_U^{(\rho)}>t \right]
    + \bE^\mu  \left[ f(X_0^{(\rho)})g(X_{t+s}^{(\rho)}): \tau_U^{(\rho)}\le t \right] \\
&= \bE^\mu \left[f(X_0^{(\rho)}){\bf 1}_{\{\tau_U^{(\rho)}>t\}} \bE^{X_t^{(\rho)}} g(X_s^{(\rho)})\right] + \bE^\mu \left[f(X_0^{(\rho)}){\bf 1}_{\{ \tau_U^{(\rho)}\le t\}} \bE^{X^{(\rho)}_{ \tau_U^{(\rho)}}} g(X^{(\rho)}_{t+s- \tau_U^{(\rho)}})\right] \\
&= \bE^\mu \left[f(X_0^{(\rho)}){\bf 1}_{\{\tau_U^{(\rho)}>t\}} Q_s g(X_t^{(\rho)})\right] + \bE^\mu \left[f(X_0^{(\rho)}){\bf 1}_{\{ \tau_U^{(\rho)}\le t\}} \bE^{X^{(\rho)}_{ \tau_U^{(\rho)}}} g(X^{(\rho)}_{t+s- \tau_U^{(\rho)}})\right] \\
&=  \bE^\mu \left[f(X_0^{(\rho)})Q_t^{(\rho), U}(Q_s g)(X_0^{(\rho)})\right] + \bE^\mu \left[f(X_0^{(\rho)}){\bf 1}_{\{ \tau_U^{(\rho)}\le t\}} \bE^{X^{(\rho)}_{ \tau_U^{(\rho)}}} g(X^{(\rho)}_{t+s- \tau_U^{(\rho)}})\right] \\
&\le  \|f\|_{L^1({D};\mu)} \|Q_sg\|_{L^1({D};\mu)} \,\esssup_{x',y'\in U} q^{(\rho), U}(t,x',y')\\
&\quad + \|f\|_{L^1({D};\mu)} \,\esssup_{x'\in V}\bP^{x'} (\tau_U^{(\rho)}\le t)\|g\|_{L^1({D};\mu)}\, \esssup_{x',y'\in U^{\rho}, s\le t'\le t+s}
q^{(\rho),{D}}(t', x',y')\\
&\le  \|f\|_{L^1({D};\mu)} \|g\|_{L^1({D};\mu)} \,\esssup_{x',y'\in U} q^{(\rho), U}(t,x',y')\\
&\quad + \|f\|_{L^1({D};\mu)}\, \esssup_{x'\in V}\bP^{x'} (\tau_U^{(\rho)}\le t)\|g\|_{L^1({D};\mu)} \,\esssup_{x',y'\in U^{\rho}, s\le t'\le t+s}
q^{(\rho),{D}}(t', x',y'),\end{align*} where we have used the strong Markov property and the fact that $X^{(\rho)}_{ \tau_U^{(\rho)}}\in U^{(\rho)}$ in the first inequality.

Furthermore, by the Cauchy-Schwarz inequality,
\begin{align*}q^{(\rho),{D}}(t, x',y')=&\int q^{(\rho),{D}}(t/2, x',z) q^{(\rho),{D}}(t/2, z,y')\,\mu(dz)\\
\le & \sqrt{\int \left( q^{(\rho),{D}}(t/2, x',z)\right)^2\,\mu(dz) }\sqrt{\int \left( q^{(\rho),{D}}(t/2, y',z)\right)^2\,\mu(dz)}\\
=& \sqrt{q^{(\rho),{D}}(t, x',x') }\sqrt{q^{(\rho),{D}}(t, y',y')},
\end{align*} and so $$
\esssup_{x',y'\in U^{\rho}}q^{(\rho),{D}}(t, x',y')= \esssup_{x'\in U^{\rho}}q^{(\rho),{D}}(t, x',x').
$$
Therefore,
$$\esssup_{x'\in U^{\rho}}q^{(\rho),{D}}(t, x',x') =\sup_{\|f\|_{L^1(U_\rho;\mu)}\le 1}\langle Q_t^{(\rho), {D}} f, f\rangle=\sup_{\|f\|_{L^1(U_\rho;\mu)}\le 1}\langle Q_{t/2}^{(\rho), {D}} f, Q_{t/2}^{(\rho), {D}}f\rangle,
$$
which implies that the function $s\mapsto\esssup_{x',y'\in U^{\rho}}q^{(\rho),{D}}(s, x',y')$ is decreasing, i.e.
$$
\esssup_{x',y'\in U^{\rho}, s\le t'\le t+s} q^{(\rho),{D}}(t', x',y')=\esssup_{x',y'\in U^{\rho}}q^{(\rho),{D}}(s, x',y').
$$
Hence,
\begin{align*}\frac{ \bE^\mu  \left[ f(X_0^{(\rho)})g(X_{t+s}^{(\rho)})\right]}{\|f\|_{L^1({D};\mu)} \|g\|_{L^1({D};\mu)}}
&\le   \esssup_{x',y'\in U} q^{(\rho), U}(t, x',y')\\
&\quad +\esssup_{x'\in V}\bP^{x'} (\tau_U^{(\rho)}\le t)\esssup_{x',y'\in U^{\rho}}
q^{(\rho),{D}}(s, x',y'). \end{align*}

Letting $f\uparrow {\bf 1}_{B(x,r)}$, $g\uparrow {\bf 1}_{B(y,r)}$ and $r\to 0$, we can get that for almost all $x,y\in V$,
\begin{align*}q^{(\rho),{D}}(t+s, x,y)\le & \esssup_{x',y'\in U} q^{(\rho),U}(t, x',y') \\
&+ \esssup_{x\in V} \bP^x(\tau^{(\rho)}_U\le t)\,\esssup_{x',y'\in U_\rho}q^{(\rho),{D}}(s, x',y')\end{align*} proving
the desired assertion.
\qed\end{proof}

The following lemma gives us the way to get heat kernel estimates in term of the exit time and the on-diagonal heat kernel estimates, e.g.\ see \cite[Theorem 5.1 and (5.13)]{GHL1}.

\begin{lemma}\label{hk-exittime}
Let $U$ and $V$ be open subsets of $M$ such that $U\cap V= \emptyset$. For any $t,s>0$ and almost all $x\in U$ and $y\in V$,
\begin{align*}
q ^{(\rho)}(t+s, x,y) & \le \bP^x(\tau^{(\rho)}_U\le t) \,\esssup_{s\le t'\le t+s}\|q^{(\rho)}(t', \cdot,y)\|_{L^\infty(U_\rho;\mu)} \\
&\quad +\bP^y(\tau_V^{(\rho)}\le s) \, \esssup_{t\le t'\le s+t}\|q^{(\rho)}(t', \cdot,x)\|_{L^\infty(V_\rho;\mu)}.
\end{align*}
\end{lemma}

\begin{proof}
Let $\bE^\mu [\cdot]:=\int_M \bE^x[\cdot]\,\mu(dx)$.
For any fixed $x\in U$ and $y\in V$, choose $0<r< \frac{1}{2}d(x,y)$
small so that $B(x, r) \subset U$ and $B(y, r)\subset V$. 
Let $f={\bf1}_{B(x,r)}$ and $g={\bf1}_{B(y,r)}$.   Then,
 by \cite[Lemma 2.1]{BGK1} (which follows from the time reversal property 
applied at time $t+s$ and the strong Markov property of the symmetric Hunt process $\{X_t^{(\rho)}\}$),
\begin{align*}
\bE^\mu \left[f(X_0^{(\rho)}) g(X_{t+s}^{(\rho)})\right]
 &=\bE^\mu\left[f(X_0^{(\rho)}) g(X_{t+s}^{(\rho)}); \tau^{(\rho)}_U\leq t\right]+ \bE^\mu\left[f(X_0^{(\rho)}) g(X_{t+s}^{(\rho)});   \tau^{(\rho)}_U >t \right] \\
  &\le \bE^\mu \left[f(X_0^{(\rho)}){\bf 1}_{\{ \tau^{(\rho)}_U\le t\}}\bE^{X_{\tau_U^{(\rho)}}^{(\rho)}}g(X^{(\rho)}_{t+s-\tau^{(\rho)}_U})\right]\\
&\quad +\bE^\mu \left[g(X_0^{(\rho)}){\bf 1}_{\{ \tau^{(\rho)}_V\le s\}}\bE^{X_{\tau_V^{(\rho)}}^{(\rho)}}f(X^{(\rho)}_{t+s-\tau^{(\rho)}_V})\right] \\
&\le \bE^\mu \left[f(X_0^{(\rho)}){\bf 1}_{\{ \tau^{(\rho)}_U\le t\}}\right] \,\esssup_{z\in U_\rho, s\le t'\le t+s}\bE^{z}g (X^{(\rho)}_{t'})\\
&\quad+ \bE^\mu \left[g(X_0^{(\rho)}){\bf 1}_{\{ \tau^{(\rho)}_V\le
s\}}\right] \,\esssup_{z\in V_\rho, t\le t'\le
t+s}\bE^{z}f(X^{(\rho)}_{t'}).\end{align*} Dividing both
sides with $\mu(B(x,r))\mu(B(y,r))$ and letting $r\to 0$, we can
obtain the desired estimate. \qed\end{proof}

The following result was proved in \cite[Theorem 3.1]{GHL2}.

\begin{lemma}\label{probability-exittime}  Assume that for any ball $B$ with radius $r>0$ and any $t>0$,
$$\bP^z(\tau_B^{(\rho)}\le t)\le \psi(r,t) \quad \textrm{  for almost all  } z\in \frac{1}{4}B,$$ where $\psi(r,\cdot)$ is a non-decreasing function for all $r>0.$  Then  for any ball $B(x,r)$, $t>0$ and any integer $k\ge 1$,
\begin{equation}\label{p-e-1} Q_t^{(\rho)}{\bf 1}_{B(x,k(r+\rho))^c}(z) \le  \psi(r,t)^k\quad \textrm{  for almost all  } z\in B(x,r/4).\end{equation} Consequently, for any ball $B(x,R)$ with $R>\rho$, $t>0$ and any integer $k\ge 1$,
$$Q_t^{(\rho)} {\bf 1}_{B(x,kR)^c}(z) \le  \psi(R-\rho,t)^{k-1}\quad \textrm{  for almost all  } z\in B(x,R).$$  \end{lemma}
\begin{proof}We prove \eqref{p-e-1} by induction in $k$. Indeed, for $k=1$,
$$Q_t^{(\rho)}{\bf 1}_{B(x,r+\rho)^c}(z)\le \bP^z(\tau_{B(x,r)}^{(\rho)}< t)\le \psi(r,t) \quad \textrm{  for almost all  } z\in B(x,r/4).$$ For the inductive step from $k$ to $k+1$, we use the strong Markov property and get that for almost all $z\in B(x,r/4)$,
\begin{align*} Q_t^{(\rho)}{\bf 1}_{B(x,(k+1)(r+\rho))^c}(z) &=\bE^z\left[{\bf 1}_{\{\tau^{(\rho)}_{B(x,r)}<t\}} \bP^{X^{(\rho)}_{\tau_{B(x,r)}^{(\rho)}}}\Big(X_{t-\tau^{(\rho)}_{B(x,r)}}^{(\rho)}\notin B(x,(k+1)(r+\rho))\Big) \right]\\
&\le \bP^z(\tau_{B(x,r)}^{(\rho)}<t) \,\esssup_{y\in B(x,r+\rho), s\le t} Q_s^{(\rho)}{\bf 1}_{B(y,k(r+\rho))^c}(y)\\
&\le   \psi(r,t)^{k+1}.  \end{align*} Here, in the first inequality above
we have used the facts that $X^{(\rho)}_{\tau_{B(x,r)}^{(\rho)}}\in
B(x,r+\rho)$, and for $z\notin B(x,(k+1)(r+\rho))$ and $y\in
B(x,r+\rho)$, it holds $d(z,y)\ge d(x,z)-d(y,x)\ge k(r+\rho).$ The last inequality above follows from the assumption that $\psi(r,\cdot)$ is
a non-decreasing function for all $r>0.$ This proves \eqref{p-e-1}.

Finally, let $r=R-\rho>0$.  Then  by \eqref{p-e-1}, for any $y\in B(x,R)$ and $k\ge 1$,
$$Q_t^{(\rho)} {\bf 1}_{B(x,({k+1})R)^c}(z) \le Q_t^{(\rho)} {\bf 1}_{B(y,kR)^c}(z)\le    \phi(R-\rho,t)^k\quad \textrm{  for almost all  } z\in B (y,r/4).$$ Covering $B(x,R)$ by a countable family of balls like $B(y,r/4)$ with $y\in B(x,R)$ and renaming $k$ to $k-1$, we prove the second assertion.   \qed\end{proof}

\subsection{$\SCSJ(\phi)+ \J_{\phi,\le}\Longrightarrow
(\sE, \sF) $  is conservative } \label{S:7.5}

We will prove the following statement in this subsection of the Appendix.
Although this theorem is not used in the main body of the paper, we
include it here
 since it indicates that $\FK(\phi)$ is not required to deduce the conservativeness.
See the paragraph  after  the statement of
 Theorem \ref{T:main-1} for related discussions.

\begin{theorem}\label{T:conser}Assume that $\VD$ and \eqref{polycon} hold. Then,
$$\SCSJ(\phi)+\J_{\phi,\le}\Longrightarrow
(\sE, \sF) \hbox{ is conservative} .
$$
\end{theorem}

Under $\VD$, \eqref{polycon} and $\J_{\phi,\le}$, in view of Lemma \ref{intelem} and Meyer's construction of adding and removing jumps in Subsection \ref{Sect9.2},
$(\sE,\sF)$ is conservative
if and only if so is $(\sE^{(\rho)},\sF)$
 for some (and hence for any)  $\rho>0$.
Therefore, to prove the conservativeness of $(\sE,\sF)$, it suffices
to establish it for $(\sE^{(\rho)},\sF)$ for some $\rho>0$. Our
proof is based on Davies' method \cite{D},
similar to what  is done in \cite[Section 6]{AB} for diffusion
processes.

We first give some notations. Fix $x_0\in M$ and $r>0$, let
$B_{r}=B(x_0,r)$.
Suppose  $\SCSJ(\phi)$ holds.  Let $\vp_n$ be the associated cut-off function for
$B_{n\rho}\subset B_{(n+1)\rho}$ in $\SCSJ(\phi)$, and  $\{a_n; n\ge -1\}$ an
increasing sequence with $a_{-1}=a_0\ge0$. Set
\begin{equation}\label{e:7.5}
 \wt \vp= a_0+\sum_{n=0}^\infty(a_{n+1}-a_n)(1-\vp_n) .
 \end{equation}
 Note that
\begin{equation}\label{e:7.5-6}
\wt \vp =a_0 + \sum_{k=0}^{n-1}
(a_{k+1}-a_k) (1-\vp_k)\le a_n\ \hbox{ on } B_{n\rho},
\end{equation}
 and  for $0\leq j  <n$,
\begin{equation}\label{e:7.6}
 \tilde\vp \geq  a_0+\sum_{k=0}^{j-1} (a_{k+1}-a_k) (1-\vp_k)=a_j
 \ \hbox{ on } M\setminus B_{j\rho}.
 \end{equation}
We have the following statement.

\begin{lemma}\label{L:con-1}Assume that $\VD$, \eqref{polycon}, $\J_{\phi,\le}$ and $\SCSJ(\phi)$ hold. Then  for any
$f\in \sF_b$,
\begin{equation}\label{e:7.7}
\int_M f^2\, d\Gamma^{( \rho)}(\wt \vp, \wt \vp)\le  A_0
\left(\frac{1}{8}\int_M
\tilde\vp^2\,d\Gamma^{(\rho)}(f,f)+\frac{C_0}{\phi(\rho)}\int_{M}\tilde\vp^2f^2\,d\mu\right),
\end{equation}
where
\begin{equation}\label{e:7.8}
A_0 :=\sup_{n\ge0} \left( \frac{ a_{n+1}-a_n}{a_{n-1}} \right)^2.
\end{equation}
\end{lemma}

\begin{proof}
By considering $f \vp_n$ in place of $f$ and then taking $n\to
\infty$ if needed, we may assume without loss of generality that
$f\in \FF_b$ has compact support. Thus in view of \eqref{e:7.5-6}, the
right hand side of \eqref{e:7.7} is finite. Let $U_n=
B_{(n+1)\rho}\setminus B_{n\rho}$ and $U^*_n=B_{(n+2)\rho}\setminus
B_{(n-1)\rho}$. Note that
$$ \Gamma^{(\rho)}(1-\vp_n, 1- \vp_m) = \Gamma^{(\rho)}(\vp_n,\vp_m)=0
$$
for any $m\ge n+3$, and
$\Gamma^{(\rho)}(1-\vp_n, 1- \vp_n)= \Gamma^{(\rho)}(\vp_n,\vp_n)=0$
outside $U_n^*$. Then  using the
Cauchy-Schwarz inequality, $\SCSJ(\phi)$ and Proposition
\ref{L:cswk}(2) (with $\eps=\frac{1}{48}$ in $\CSAJ^{(\rho)}(\phi)_+$),
we have
\begin{align*}
\int_M f^2\,d\Gamma^{(\rho)}(\wt \vp, \wt \vp)
&\le2\sum_{n=0}^\infty\sum_{n\le m}
(a_{n+1}-a_n)(a_{m+1}-a_m) \int_M f^2\,d\Gamma^{(\rho)}(\vp_n,\vp_m)\\
&=2\sum_{n=0}^\infty\sum_{n\le m\le n+2} (a_{n+1}-a_n)(a_{m+1}-a_m) \int_M
f^2\,d\Gamma^{(\rho)}(\vp_n,\vp_m)\\
&\le \sum_{n=0}^\infty\sum_{n\le m\le n+2} (a_{n+1}-a_n)^2 \int_M
f^2\,d\Gamma^{(\rho)}(\vp_n,\vp_n)\\
&\quad +\sum_{n=0}^\infty\sum_{n\le m\le n+2} (a_{m+1}-a_m)^2 \int_M
f^2\,d\Gamma^{(\rho)}(\vp_m,\vp_m)\\
&\le 6\sum_{n=0}^\infty (a_{n+1}-a_n)^2\int_M
f^2\,d\Gamma^{(\rho)}(\vp_n,\vp_n)\\
&=6\sum_{n=0}^\infty (a_{n+1}-a_n)^2\int_{U_n^*}
f^2\,d\Gamma^{(\rho)}(\vp_n,\vp_n)\\
&\le \sum_{n=0}^\infty (a_{n+1}-a_n)^2 \left(
\frac{1}{8}\int_{U_n}\,d\Gamma^{(\rho)}(f,f)+\frac{c_1}{\phi(\rho)}\int_{U_n^*}f^2\,d\mu\right)\\
&\le \sum_{n=0}^\infty \left(\frac{a_{n+1}-a_n}{a_{n-1}}\right)^2
\left( \frac{1}{8}\int_{U_n}\tilde
\vp^2\,d\Gamma^{(\rho)}(f,f)+\frac{c_1}{\phi(\rho)}\int_{U_n^*}\tilde\vp^2f^2\,d\mu\right),
\end{align*} where in the last inequality we have used the fact that $a_{n-1}\le \tilde \vp\le a_{n+2}$ on
$U_n^*$ from \eqref{e:7.5-6} and \eqref{e:7.6}.
The proof is complete. \qed\end{proof}

We also need the following lemma.
\begin{lemma}\label{L:con-2}Assume that $\VD$, \eqref{polycon}, $\J_{\phi,\le}$ and $\SCSJ(\phi)$ hold.
Let $\tilde \vp$ and $A_0$ be as in \eqref{e:7.5} and \eqref{e:7.8},
respectively. Suppose that $A_0\leq 1$. Let $f$ have compact
support, and set $u(t)=Q^{(\rho)}_tf$. Then, we have
 \begin{equation}\label{e:semi-2}
 \int_0^t\,ds\int_M \tilde\vp^2\,d\Gamma^{(\rho)} (u(s),u(s))\le 2\|f\tilde\vp\|^2_{2}\exp\left(\frac{4C_0t}{\phi(\rho)}\right).
 \end{equation}\end{lemma}

\begin{proof} We may assume the boundedness of $f$. Let $(a_n)_{n\ge -1}$ and $\vp_n$ as above. For any $N\ge 1$, set
$$
\tilde\vp_{0,N}=a_0+\sum_{n=0}^N(a_{n+1}-a_n)(1-\vp_n)
$$
 and
$$
h_N(t)=\|u(t)\tilde \vp_{0,N}\|_2^2.
$$
 We write $u(t,x)=Q^{(\rho)}_tf(x)$.
Since $u(t) \in\sF_b$ and $\wt\vp_{0,N}^2u(t)\in \sF_b$,
\begin{align*} h_N'(t)=&-2\sE^{(\rho)}(u(t),
\tilde \vp_{0,N}^2u(t))\\
=&-2
\int_{M \times M}
(u(t,x)-u(t,y))(\tilde \vp_{0,N}^2(x)u(t,x)-\tilde \vp_{0,N}^2(y)u(t,y))\,J^{(\rho)}(dx,dy)\\
= &-2\int_{M \times M}  (u(t,x)-u(t,y))^2\tilde\vp_{0,N}^2(x)\,J^{(\rho)}(dx,dy)\\
&-2\int_{M \times M}  (\tilde \vp_{0,N}^2(x)-\tilde \vp_{0,N}^2(y))u(t,y)(u(t,x)-u(t,y))\,J^{(\rho)}(dx,dy)\\
\le&-2\int_{M \times M}  (u(t,x)-u(t,y))^2\tilde\vp_{0,N}^2(x)\,J^{(\rho)}(dx,dy)\\
&+\frac{1}{4}\int_{M \times M}  (\tilde \vp_{0,N}(x)+\tilde\vp_{0,N}(y))^2(u(t,x)-u(t,y))^2\,J^{(\rho)}(dx,dy)\\
&+4 \int_{M \times M} u(t,y)^2(\tilde \vp_{0,N}(x)-\tilde\vp_{0,N}(y))^2\,J^{(\rho)}(dx,dy)\\
\le&-2\int_{M \times M} (u(t,x)-u(t,y))^2\tilde\vp_{0,N}^2(x)\,J^{(\rho)}(dx,dy)\\
&+\frac{1}{2}\int_{M \times M}  (\tilde \vp_{0,N}^2(x)+\tilde\vp_{0,N}^2(y))(u(t,x)-u(t,y))^2\,J^{(\rho)}(dx,dy)\\
&+4\int_{M \times M} u(t,y)^2
(\tilde \vp_{0,N}(x)-\tilde\vp_{0,N}(y))^2\,J^{(\rho)}(dx,dy)\\
= &- \int_{M \times M}  (u(t,x)-u(t,y))^2\tilde\vp_{0,N}^2(x)\,J^{(\rho)}(dx,dy)\\
&+4\int_{M \times M} u(t,x)^2(\vp_{0,N}(x)-\vp_{0,N}(y))^2\,J^{(\rho)}(dx,dy)\\
=&-\int_{M \times M}  \tilde\vp_{0,N}^2\,d\Gamma^{(\rho)}(u(t),u(t))+ 4\int_{M \times M}
u(t)^2\,d\Gamma^{(\rho)}(\vp_{0,N},\vp_{0,N}),
\end{align*}
where in the first inequality we used the fact that $2ab\le
\frac{a^2}{4}+4b^2$ for all $a,b\in \bR$, and in the last inequality
$$ \vp_{0,N}:=\sum_{n=0}^N(a_{n+1}-a_n)
\vp_n=-\tilde\vp_{0,N}+a_{N+1}.
$$
So  by (the proof of)
Lemma \ref{L:con-1} and the
assumption  $A_0\leq 1$,
\begin{equation}\label{e:lcon-2}
h_N'(t)\le -\frac{1}{2}\int_M \tilde\vp_{0,N}^2\,d\Gamma^{(\rho)}(u(t),u(t))+\frac{4C_0}{\phi(\rho)}h_N(t).\end{equation}
In particular, $$h_N'\le \frac{4C_0}{\phi(\rho)}h_N $$ and hence
$$h_N(t)\le h_N(0)\exp\left(
\frac{4C_0t}{\phi(\rho)}\right)=\|f\tilde \vp_{0,N}\|_2^2\exp\left(\frac{4C_0t}{\phi(\rho)}\right).$$
Using the inequality above and integrating
\eqref{e:lcon-2}, we obtain
$$
h_N(t)-h_N(0)+ \frac{1}{2}\int_0^t\,ds\int_M \tilde \vp_{0,N}^2\,d\Gamma^{(\rho)}(u(s),u(s))
\le \|f\tilde \vp_{0,N}\|_2^2 (e^{4C_0t/\phi(\rho)}-1).
$$
Since $h_N(0)=\|f\tilde \vp_{0,N}\|_2^2$, letting
$N\to\infty$ gives us the desired assertion. \qed
 \end{proof}

\medskip \noindent
{\bf Proof of Theorem \ref{T:conser}.} \quad We mainly follow the
argument of \cite[Theorem 7]{D} and make use of Lemma \ref{L:con-2}
above. Let $f\ge0$ be a bounded
 function with compact support and let
$u(t)=Q^{(\rho)}_tf$. As mentioned in the remark below Theorem
\ref{T:conser}, it is sufficient to verify that $Q^{(\rho)}_t{\bf
1}=1 $ $\mu$-a.e for every $t>0$. Since $\int_M Q^{(\rho)}_t f
\,d\mu = \int_M f  \, Q^{(\rho)}_t {\bf 1} \,d\mu$, it reduces to
show that
\begin{equation}\label{e:tcon}
\int_M f\,d\mu\le \int_M u(t)\,d\mu
\end{equation}
for some $t>0$.

For any $n\ge 0$, let $a_n=s^n$ with $s>1$ such that $s(s-1)\le 1$,
and set $a_{-1}=1$. In particular,  with $A_0$  defined by
\eqref{e:7.8},  we have $A_0= s^2(1-s)^2 \leq 1$. Let $\vp_n$ and
$\tilde \vp$ be defined as in the paragraph containing
\eqref{e:7.5}. Set   $U^*_n=B_{(n+2)\rho}\setminus B_{(n-1)\rho}$.
 Then for $t\in (0, 1]$,   by
the Cauchy-Schwarz inequality and Lemma \ref{L:con-2}, for any $t\in(0,1)$,
\begin{align*}\langle f, \vp_n\rangle-\langle u(t), \vp_n\rangle&=-\int_0^t\frac{d}{ds}\langle u(s), \vp_n\rangle\,ds\\
&=\int_0^t\,ds\int_M \Gamma^{(\rho)}(u(s), \vp_n)\\
&=\int_0^t\,ds\int_M\tilde\vp\cdot \tilde\vp^{-1} \,d\Gamma^{(\rho)}(u(s), \vp_n)\\
&\le \left(\int_0^t\,ds\int_M \tilde\vp^2\,d\Gamma^{(\rho)}(u(s),u(s))\right)^{1/2}\left(\int_0^t\,ds\int_M\tilde\vp^{-2}\,d\Gamma^{(\rho)}(\vp_n,\vp_n)\right)^{1/2}\\
&\le \sqrt{2}\|f\tilde\vp\|_2e^{2C_0t/\phi(\rho)}
(\sup_{U_n^*}\tilde\vp^{-1})\left(\int_{U_n^*}\,\Gamma^{(\rho)}(\vp_n,\vp_n)\right)^{1/2},\end{align*}
where in the last inequality we used again the fact that
$\Gamma^{(\rho)}(\vp_n,\vp_n)=0$ outside $U_n^*$. Note that on
$U_n^*$, we have from \eqref{e:7.5-6} and \eqref{e:7.6} that
$a_{n-1}\le \wt \vp \le a_{n+2}$ and  so $ \sup_{U_n^*} \wt \vp^{-1}
\le a_{n-1}^{-1} $.
 On the other hand,
using $\SCSJ(\phi)$ with
 $f\in \sF \cap C_c(M)$
such that $f|_{B_{(n+2)\rho}}=1$,
we find that
$$\int_{U_n^*}\,\Gamma^{(\rho)}(\vp_n,\vp_n)\le \frac{c_1}{\phi(\rho)}\mu({U_n^*}).$$
Combining all the conclusions above, we get
$$
\langle f, \vp_n\rangle-\langle u(t), \vp_n\rangle\le \sqrt{2} \,
\|f\tilde\vp\|_2 \, \exp\left( \frac{2C_0t}{\phi(\rho)}+
\frac{1}{2}\log \left(\frac{c_1}{\phi(\rho)}\mu({U_n^*})\right)-\log
a_{n-1}\right) .
$$
Noting that due to $\VD$, $\mu({U_n^*})\le \mu(B_{(n+2)\rho})\le
c_2(\rho) n^{d_2}$ for any $n\ge 0$, and $a_n=s^n$ for $s>1$, one
can easily see that the right hand side of the inequality above
converges to 0 when $n\to \infty$. Since
$$ \int_M  u(t)\,d\mu=\lim_{n\to\infty}\int_M u(t)\vp_n\,d\mu \quad \textrm{and} \quad
\int_M f\,d\mu=\lim_{n\to\infty}\int_M  f\vp_n\,d \mu,$$
 we get  \eqref{e:tcon} and
the conservativeness of $(\sE, \sF)$. \qed

\begin{remark}\rm By using the arguments above, one can study the stochastic completeness in terms of $\SCSJ(\phi)$ for jump processes in
general settings, namely to obtain some sufficient condition for the
stochastic completeness without $\VD$ assumption. See \cite[Theorem
1.16 and Section 7]{AB} for related discussions about diffusions.
\end{remark}

\ \

\noindent{\bf Acknowledgement.}   The authors thank  the referee for  helpful comments.

 \vskip 0.3truein

\noindent {\bf Zhen-Qing Chen}

\smallskip \noindent
Department of Mathematics, University of Washington, Seattle,
WA 98195, USA

\noindent
E-mail: \texttt{zqchen@uw.edu}

\medskip

\noindent {\bf Takashi Kumagai}:

\smallskip \noindent
 Research Institute for Mathematical Sciences,
Kyoto University, Kyoto 606-8502, Japan

\noindent Email: \texttt{kumagai@kurims.kyoto-u.ac.jp}

\medskip

\noindent {\bf Jian Wang}:

\smallskip \noindent
School of Mathematics and Computer Science \& Fujian Key Laboratory of Mathematical Analysis and Applications (FJKLMAA), Fujian Normal
University, 350007, Fuzhou, P.R. China.

\noindent Email: \texttt{jianwang@fjnu.edu.cn}

 \end{document}